%% file: main.tex
\documentclass{article}

\usepackage[english]{babel}
\usepackage[letterpaper,top=2cm,bottom=2cm,left=3cm,right=3cm,marginparwidth=1.75cm]{geometry}

\usepackage{tikz-cd}
\usepackage{linestyles}

\usepackage{subcaption}

\usepackage{amsmath}
\usepackage{amssymb}
\usepackage{amsthm}
\usepackage{graphicx}
\usepackage{xcolor}
\usepackage{hyperref}
\usepackage{tikz}
\usetikzlibrary{arrows}
\usepackage{standalone}
\usepackage{caption}
\usepackage{subcaption}

\usepackage{algorithm}
\usepackage{algpseudocode}

\usepackage{stmaryrd} 
\SetSymbolFont{stmry}{bold}{U}{stmry}{m}{n}

\newtheorem{theorem}{Theorem}[section]

\newtheorem{lemma}[theorem]{Lemma}
\newtheorem{proposition}[theorem]{Proposition}
\newtheorem{conjecture}[theorem]{Conjecture}

\theoremstyle{remark}

\newtheorem{remark}[theorem]{Remark}
\newtheorem{question}[theorem]{Question}

\DeclareMathOperator{\id}{id}
\DeclareMathOperator{\rank}{rank}

\newcommand{\Z}{\mathbb{Z}}
\newcommand{\R}{\mathbb{R}}
\newcommand{\CKh}{\operatorname{CKh}}
\newcommand{\HKh}{\operatorname{Kh}}
\newcommand{\rKh}{\tilde{\operatorname{Kh}}}
\newcommand{\rankCob}{\operatorname{rank}_{\mathcal C ob}}
\newcommand{\Cob}{\mathcal{C}ob_\bullet}
\newcommand{\Cobl}{\mathcal{C}ob_{\bullet/l}}
\newcommand{\Mlex}{M_{\operatorname{lex}}}
\newcommand{\term}[1]{\textit{#1}} 

\newcommand{\grayCirc}{\begin{tikzpicture}[baseline={(0,-0.1)}]
    \draw[thinredline] (0,0) circle (0.125);
\end{tikzpicture}}
\newcommand{\blackCirc}{\begin{tikzpicture}[baseline={(0,-0.1)}]
    \draw[thinblueline] (0,0) circle (0.125);
\end{tikzpicture}}

\newcommand{\crosNeg}
    {\begin{tikzpicture}[baseline={(0,0.05)},scale=0.3]
    \draw[thick] (0,0)  -- (0.35,0.35);
    \draw[thick] (0.65,0.65)  -- (1,1);
    \draw[thick] (1,0) -- (0,1);
\end{tikzpicture}}

\newcommand{\crosPos}{
    \begin{tikzpicture}[baseline={(0,0.05)},scale=0.3]
    \draw[thick] (0,0)  -- (1,1);
    \draw[thick] (0,1)  -- (0.35,.65);
    \draw[thick] (.65,.35) -- (1,0);
\end{tikzpicture}
}

\newcommand{\crosNg}[1]
    {\begin{tikzpicture}[baseline={(0,0.05)},scale=#1]
    \draw[thick] (0,0)  -- (0.35,0.35);
    \draw[thick] (0.65,0.65)  -- (1,1);
    \draw[thick] (1,0) -- (0,1);
\end{tikzpicture}}

\newcommand{\crosPs}[1]{
    \begin{tikzpicture}[baseline={(0,0.05)},scale=#1]
    \draw[thick] (0,0)  -- (1,1);
    \draw[thick] (0,1)  -- (0.35,.65);
    \draw[thick] (.65,.35) -- (1,0);
\end{tikzpicture}
}

\newcommand{\lineup}[1]{
\begin{tikzpicture}[baseline={(0,0.05)},scale=#1]
\draw[thick] (0,0)  -- (0,1);
\end{tikzpicture}
}

\newcommand{\orCrosNeg}
    {\begin{tikzpicture}[baseline={(0,0.05)},scale=0.3]
    \draw[thick] (0,0)  -- (0.35,0.35);
    \draw[thick, -{Stealth[length=1mm]}] (0.65,0.65)  -- (1,1);
    \draw[thick,-{Stealth[length=1mm]}] (1,0) -- (0,1);
\end{tikzpicture}}

\newcommand{\orCrosPos}{
    \begin{tikzpicture}[baseline={(0,0.05)},scale=0.3]
    \draw[thick,-{Stealth[length=1mm]}] (0,0)  -- (1,1);
    \draw[thick,{Stealth[length=1mm]}-] (0,1)  -- (0.35,.65);
    \draw[thick] (.65,.35) -- (1,0);
\end{tikzpicture}
}

\newcommand{\smoothingZero}{
\begin{tikzpicture}[baseline={(0,0.05)},scale=0.3]
    \draw[thick] (0,0) to[bend right=50] (0,1);
    \draw[thick] (1,0) to[bend left=50] (1,1);
\end{tikzpicture}
}

\newcommand{\smoothingOne}{
\begin{tikzpicture}[baseline={(0,0.05)},scale=0.3]
    \draw[thick] (0,0) to[bend left=50] (1,0);
    \draw[thick] (0,1) to[bend right=50] (1,1);
\end{tikzpicture}
}

\newcommand{\drawblack}[4] {
\draw[->] (#1) -- node [above, scale = 0.7, #3] {$#4$} (#2);
}
\newcommand{\drawblackw}[4] {
\draw[-, line width = 6pt, color = white] (#1) -- (#2);
\draw[->] (#1) -- node [above, scale = 0.7, #3] {$#4$} (#2);
}

\title{On computational complexity of Khovanov homology}
\author{Tuomas Kelom\"aki and Dirk Sch\"utz}

\begin{document}
\maketitle

\begin{abstract}
Computing the Jones polynomial of general link diagrams is known to be $\#$P-hard, while restricting the computation to braid closures on fixed number of strands allows for a polynomial time algorithm. We investigate polynomial time algorithms for Khovanov homology of braids and show that for $3$-braids there is one. In contrast, we show that Bar-Natan's scanning algorithm runs in exponential time when restricted to simple classes of $3$-braids. For more general braids, we obtain that a variation of the scanning algorithm computes the Khovanov homology for a bounded set of homological degrees in polynomial time. We also prove upper and lower bounds on the ranks of Khovanov homology groups.
\end{abstract}


\section{Introduction}

The original construction of Khovanov presented a link invariant, Khovanov homology, which is theoretically computable for any link diagram \cite{MR1740682}, but the number of steps required to do so grows exponentially in the number of crossings. A few years later, Bar-Natan's work on tangles made it practically possible to compute Khovanov homology for links with 50 or more crossings \cite{MR2320156}, which would be impossible from the defintion. This advancement in algorithms has opened up new avenues in $4$-dimensional topology: In \cite{MR2657647} a possible (yet to be realized) attack for disproving the smooth $4$-dimensional Poincaré conjecture was formulated and in \cite{zbMATH07168645} a computer calculation helped to prove that the Conway knot is not smoothly slice.

The Jones polynomial $J_L(q)$ of a link $L$ can easily be recovered from its Khovanov homology $\HKh(L)$ and hence Khovanov homology is at least as hard to compute as the Jones polynomial. By reducing the computation of the Jones polynomial of a link diagram $L$ to the evaluation of the  Tutte polynomial of a related planar graph, Jaeger, Vertigan and Welsh showed that $J_L(q)$ is $\#P$-hard to compute with respect the number of crossings on $L$ \cite{zbMATH00006335}. It follows that the existence of a polynomial time algorithm which, given a link diagram $L$, computes the Khovanov homology $\HKh(L)$ is very unlikely. On the other hand, Morton and Short showed that the Jones polynomial of a closed braid on a fixed number of strands can be computed in polynomial time \cite{MR1041170}. This led to the following conjecture which provides the main motivation for this article.
\begin{conjecture}[Przytycki, Silvero \cite{zbMATH07862431}]\label{conjecture: Kh braid poly}
    Computing Khovanov homology of a closed braid with fixed number of strands has polynomial time complexity with respect to the number of crossings.
\end{conjecture}

The state-of-the-art for computing Khovanov homology is given by Bar-Natan's scanning algorithm \cite{MR2320156}. However, there exist relatively simple links which make this algorithm very slow.

\begin{theorem}\label{thm:slowscanning}
The scanning algorithm and the divide-and-conquer algorithm of \cite{MR2320156} run in exponential time with respect to the number of crossings in a link diagram, even when restricted to positive $3$-braids, or to alternating $3$-braids.
\end{theorem}

The reason for the slow runtime of these algorithms is that there are such $3$-braids whose total Betti numbers grow exponentially with respect to the number of crossings. Any algorithm, which works with a basis on the chain complex, will therefore necessarily have an exponential runtime. As with classical algorithms, the large Betti numbers also played a significant role in the recent study of quantum algorithms for Khovanov homology \cite{schmidhuber2025quantumalgorithmkhovanovhomology}. The exponential Betti numbers make it impossible for the main algorithm of \cite{schmidhuber2025quantumalgorithmkhovanovhomology} to accurately estimate the Khovanov homology of $3$-braids in polynomial time. For other standard quantum homology algorithms \cite[\S 3]{schmidhuber2025quantumalgorithmkhovanovhomology} large Betti numbers seem to be an asset although it remains unclear how apply them to Khovanov homology.  Despite the exponential Betti numbers, we give an affirmative answer to Conjecture \ref{conjecture: Kh braid poly} in the case of braids on $3$ strands.

\begin{theorem} \label{thm:alg3braids} 
    Conjecture \ref{conjecture: Kh braid poly} is true for $3$-braids.
\end{theorem}
The algorithm given in Section \ref{Section: 3braid polytime} to prove Theorem \ref{thm:alg3braids} is based on the earlier work of the second author \cite{schuetz2025kh3braids}, which showed that the integer Khovanov homology of a closed $3$-braid either is described by a formula in \cite{MR4430925}, or decomposes roughly as a direct sum of Khovanov homologies of a torus link $T(3,3k)$ and an alternating link. The homology groups $\HKh(T(3,3k))$ have also been computed in \cite{MR4430925} while the Khovanov homology of the alternating links can be obtained from their Jones polynomials and the signatures. These alternating links are themselves braid closures, so the polynomial time algorithm of Morton and Short can be used.

The Khovanov complex $\CKh(L)$ of a link diagram $L$ is a bigraded complex of finitely many finite rank $\Z$-modules. Denote the minimal square of $\Z^2$ gradings, which supports $\CKh(L)$ by $[i_{\min}, i_{\max}] \times [q_{\min}, q_{\max}]$.  As evidence for Conjecture \ref{conjecture: Kh braid poly} Przytycki and Silvero gave a polynomial time algorithm, which computed $\HKh^{\ast,q_{\min}}(L)$ (and by symmetry $\HKh^{\ast,q_{\max}}(L)$) for $4$-braids.
The construction of $\CKh(L)$ gives $i_{\min}=-n_-(L)$ and $i_{\max}=n_+(L)$ where $n_+(L)$ and $n_-(L)$ denote the number of positive and negative crossings of $L$. We show that the Khovanov homology at any number of homological gradings near $i_{\min}$ and $i_{\max}$ for braids on any number of strands can be obtained in polynomial time.

\begin{theorem}\label{thm:Extremal KH for braids}
For every $k,t \geq 0$ there is an algorithm $\mathcal B_{k,t}$ which takes in a braid $b$ on $t$ strands and outputs the integral Khovanov homology $\operatorname{Kh}^{i,\ast}(L_b)$ of the braid closure $L_b$ in homological degrees $i\leq -n_-(L_b)+k$  and $i\geq  n_+(L_b)-k$. The algorithms $\mathcal B_{k,t}$ run in polynomial time with respect to the length of $b$.
\end{theorem}

The algorithms $\mathcal B_{k,t}$ are only a minor variation of Bar-Natan's scanning algorithm. While scanning, the intermediate complexes of tangles can be truncated without loosing information relevant to the extremal homological degrees. A similar trick was used in a fast algorithm for computing $s$-invariants \cite{MR4244204}. Even in the non-truncated degrees, the Khovanov hypercube complex may have exponential rank so we have to ensure that the Gaussian elimination of the scanning algorithm cancels out `most' of the summands. Since the intermediate complexes contain integer coefficients $c$, which need to be stored in the memory, we also need some polynomial $p$ for which $\log_2(c) \leq p(n(L_b))$ for all $c$ we process. As a by-product of our proof we obtain the following upper bound where $n(L)=n_+(L)+n_-(L)$:
\begin{proposition}\label{proposition: rank bounds}
For a connected link diagram $L$ and a field $\mathbb F$  
$$
 \dim_{\mathbb F} ( \operatorname{Kh}^{i,\ast} (L; \mathbb F)) \leq 2\binom{n(L)}{i+n_-(L)}.
$$
\end{proposition}

The bounds of Proposition \ref{proposition: rank bounds} are asymptotically strict in the following sense:
\begin{proposition}\label{proposition: asymptotical strictness of the bounds}
    Let $L_t$ be the braid closure of $(\sigma_1 \sigma_3 \sigma_2^4)^t \sigma_1 \sigma_3$. For every $k\geq 0$ there exists $\varepsilon >0$ for which 
    $$
    \liminf_{t\to \infty}   \left(\frac{\operatorname{rank}(\operatorname{Kh}^{n(L_t)-k,\ast} (L_t))}{\binom{n(L_t)}{n(L_t) -k}} \right)  \geq \varepsilon .
    $$
\end{proposition}
Proposition \ref{proposition: asymptotical strictness of the bounds} is a non-vanishing result so proving it requires us to generate a lot of distinct homology cycles. To do so, we take the Khovanov complex of $L_t$ and  employ a Morse matching on it from the earlier work of the first author \cite{kelo2025torus4}. This creates a chain homotopy equivalent Morse complex, where sufficiently many copies of $\mathbb Z$ split off. The same asymptotic lower bound is also proven for odd Khovanov homology.

\textbf{Is Conjecture \ref{conjecture: Kh braid poly} true?}
As stated in \cite{zbMATH07862431}, establishing Conjecture \ref{conjecture: Kh braid poly} would be a game changer in stating and testing conjectures about Khovanov homology. While Theorem \ref{thm:alg3braids} gives hope for this conjecture to be true, its proof very much uses the particular structure of $3$-braids and does not seem to generalize easily to more strands. Furthermore, it reveals that the Khovanov homology of any closed $3$-braid roughly decomposes into a part coming from an alternating braid word and a part coming from a torus link word.
Already in the case of $4$-braids the general structure seems much more complicated. For example, \cite{MR3894728} produced positive $4$-braids containing torsion of order $3$ in their Khovanov homology, as well as others that contain torsion of order $2^k$ for $k$ up to $23$. Also, computations show that $4$-strand torus links have a much more complicated Khovanov homology than their $3$-stranded counterparts.

Disproving Conjecture \ref{conjecture: Kh braid poly} would amount to constructing a reduction to a known NP-hard problem. The authors are not aware of any promising candidates for such reduction.

\textbf{Acknowledgments:} The authors would like to thank Qiuyu Ren for presenting a proof to Proposition \ref{Proposition: polybound on torus}, Ilmo Salmenperä for providing his insight to quantum algorithms, and Marithania Silvero for comments on extremal homology. TK was supported by the Väisälä Fund of the Finnish Academy of Science and Letters.

\section{Preliminaries}
In this section we fix the notation for the Khovanov complex and a few variations, as well as our conventions for braids and algebraic discrete Morse theory.

By a {\em Frobenius algebra} $A$ over a commutative ring $R$ we mean a commutative ring $A$ that contains $R$ as a subring that comes with an $R$-module map $\varepsilon\colon A \to R$ and an $A$-bimodule map $\Delta\colon A \to A\otimes_RA$ that is co-associative and co-commutative, and such that $(\varepsilon\otimes \id)\circ A = \id$.

In the case that $A$ is free of rank $2$ over $R$, Khovanov  \cite{MR2232858} showed that for a link diagram $L$ one can define a cochain complex $C(L)$ over $R$ whose homology is a link invariant. The Frobenius algebra leading to Khovanov's original link homology \cite{MR1740682} is given by $A= \Z[X]/(X^2)$, with $\varepsilon(1) = 0$, $\varepsilon(X) = 1$, and $\Delta(1) = 1\otimes X+X\otimes 1$. Here $A$ can be given a $q$-grading by $|1|_q = 1$ and $|X|_q=-1$, which turns the {\em Khovanov complex} $\CKh(L)$ into a bigraded cochain complex. We write $\CKh^{i,j}(L)$ to indicate the homological grading and the $q$-grading, and $\HKh^{i,j}(L)$ for the corresponding homology groups. We may also omit the gradings, or write $\HKh^{i,j}(L;R)$ if we want to indicate coefficients in a ring $R$.

For tangles $T$ there is also a Khovanov complex, we follow mainly Bar-Natan's dotted version \cite{MR2174270}. Let $B\subset \R^2$ be a compact surface and $\dot{B}\subset \partial B$ an oriented $0$-dimensional compact manifold bordant to the empty set. We denote by $\Cob(B,\dot{B})$ the category whose objects are compact smooth $1$-dimensional submanifolds $S\subset B$ intersecting $\partial B$ transversely in $\partial S = \dot{B}$. Morphisms between objects $S_0$ and $S_1$ are isotopy classes of dotted cobordisms between $S_0\times\{0\}$ and $S_1\times\{1\}$ embedded in $B\times [0,1]$ and which are a product cobordism near $\dot{B}\times [0,1]$. Here dotted means that there are finitely many specified points in the interior of the cobordism, which can move freely there.

If $R$ is a commutative ring, we denote by $\Cobl^R(B,\dot{B})$ the additive category whose objects are finitely generated free and based $R$-modules where basis elements are objects $q^jS$ with $S$ an object of $\Cob(B,\dot{B})$ and $j\in \Z$. Here $q^j$ stands for a shift in the $q$-grading. Morphisms are given by matrices $(M_{nm})$ with each matrix entry $M_{nm}$ an element of the free $R$-module generated by the morphism set between $S_n$ and $S_m$ in $\Cob(B,\dot{B})$, modulo the relations in Figure \ref{Figure: relations of Cob}. The category $\Cobl^R(B,\dot{B})$ admits a delooping isomorphism, see Figure \ref{Delooping isomorphism}, which is at the heart of many of our computations.

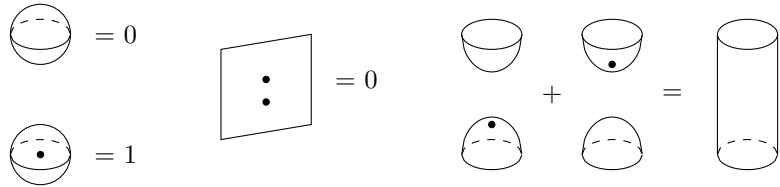
\begin{figure}[ht]
    \centering    
\begin{tikzpicture}[scale=0.4]
  
  \def\xradius{1};
  \def\yradius{0.5};
  \def\ballsxcoor{-15}
  \def\rectanglexcoor{-9}

      \draw (\ballsxcoor,4) ellipse (\xradius cm and \xradius cm);
    \draw[dashed] plot[domain=-\xradius:\xradius] ({\x +\ballsxcoor},{4+1*\yradius * sqrt(1 - (\x / \xradius)^2)});
  \draw plot[domain=-\xradius:\xradius] ({\x +\ballsxcoor},{4-1*\yradius * sqrt(1 - (\x / \xradius)^2)});
\draw (\ballsxcoor+2.5,4) node { = 0};

      \draw (\ballsxcoor,0) ellipse (\xradius cm and \xradius cm);
    \draw[dashed] plot[domain=-\xradius:\xradius] ({\x +\ballsxcoor},{1*\yradius * sqrt(1 - (\x / \xradius)^2)});
  \draw plot[domain=-\xradius:\xradius] ({\x +\ballsxcoor},{-1*\yradius * sqrt(1 - (\x / \xradius)^2)});
\draw (\ballsxcoor+2.5,0) node { = 1};

  \coordinate (A) at (\rectanglexcoor,0.5);
  \coordinate (B) at (\rectanglexcoor+3,1);
  \coordinate (C) at (\rectanglexcoor+3,4);
  \coordinate (D) at (\rectanglexcoor,3.5);

  \draw (A) -- (B) -- (C) -- (D) -- cycle;

    \filldraw (\rectanglexcoor+1.5,1.75) circle (3pt);
    \filldraw (\rectanglexcoor+1.5,2.5) circle (3pt);
    \draw (\rectanglexcoor+4.5,2.5) node {= 0};

    \filldraw (\ballsxcoor,0) circle (3pt);

  \draw (0,4) ellipse (\xradius cm and \yradius cm);
  \draw (4,4) ellipse (\xradius cm and \yradius cm);

\draw[dashed] plot[domain=-\xradius:\xradius] ({\x},{1*\yradius * sqrt(1 - (\x / \xradius)^2)});
\draw plot[domain=-\xradius:\xradius] ({\x},{-1*\yradius * sqrt(1 - (\x / \xradius)^2)});

  \draw plot[domain=-\xradius:\xradius] ({\x},{4 - 2.5*\yradius * sqrt(1 - (\x / \xradius)^2)});
  \draw plot[domain=-\xradius:\xradius] ({\x + 4},{4 - 2.5*\yradius * sqrt(1 - (\x / \xradius)^2)});

  \draw plot[domain=-\xradius:\xradius] ({\x},{2.5*\yradius * sqrt(1 - (\x / \xradius)^2)});
  \draw plot[domain=-\xradius:\xradius] ({\x + 4},{2.5*\yradius * sqrt(1 - (\x / \xradius)^2)});
  \draw[dashed] plot[domain=-\xradius:\xradius] ({\x + 4},{1*\yradius * sqrt(1 - (\x / \xradius)^2)});
  \draw plot[domain=-\xradius:\xradius] ({\x + 4},{-1*\yradius * sqrt(1 - (\x / \xradius)^2)});
  
  \draw (2,2) node { +};

    \draw (6,2) node { =};

    \draw (8.5,4) ellipse (\xradius cm and \yradius cm);
  \draw[dashed] plot[domain=-\xradius:\xradius] ({\x + 8.5},{1*\yradius * sqrt(1 - (\x / \xradius)^2)});
  \draw plot[domain=-\xradius:\xradius] ({\x + 8.5},{-1*\yradius * sqrt(1 - (\x / \xradius)^2)});

    \draw (7.5,0) -- (7.5,4);
    \draw (9.5,0) -- (9.5,4);
    
    \filldraw (4,3) circle (3pt);
    \filldraw (0,1) circle (3pt);
  
\end{tikzpicture}

    \caption{The dotted relations of $\Cobl^{\mathbb Z}(B,\dot{B})$. We refer to the two leftmost as \term{sphere relations}, the middle one as \term{double dot relation} and the right one as \term{neck cutting relation}.}
    \label{Figure: relations of Cob}
\end{figure}

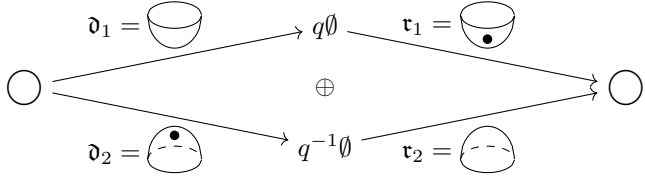
\begin{figure}[ht]
  \centering
    \begin{tikzpicture}[scale=0.8]

  \node (A) at (-5, 2) {\Large$\bigcirc$};
\node (B) at (0, 3) { $q\emptyset$};
  \node (C) at (0, 1) { $q^{-1}\emptyset $};
  \node (D) at (5, 2) {\Large$\bigcirc$};

\node (E) at (0, 2) {$\oplus$};

  \draw[->] (A) -- (B);
  \draw[->] (A) -- (C);
  \draw[->] (B) -- (D);
  \draw[->] (C) -- (D);

  \def\xradius{0.45};
  \def\yradius{0.22};

\def\a{2.7}
\def\b{3.2}

\def\c{-2.5}
\def\d{3.2}

\def\e{2.7}
\def\f{0.8}

\def\g{-2.5}
\def\h{0.8}

\node at (\a-1,\b-0.2) {$\mathfrak r_1=$ };
\node at (\e-1,\f+0.1) {$\mathfrak r_2=$ };

\node at (\c-1,\d-0.2) {$\mathfrak d_1=$ };
\node at (\g-1,\h+0.1) {$\mathfrak d_2=$ };

\draw plot[domain=-\xradius:\xradius] ({\a+\x},{\b+1*\yradius * sqrt(1 - (\x / \xradius)^2)});
\draw plot[domain=-\xradius:\xradius] ({\a+\x},{\b-1*\yradius * sqrt(1 - (\x / \xradius)^2)});
\draw plot[domain=-\xradius:\xradius] ({\a+\x},{\b-2.5*\yradius * sqrt(1 - (\x / \xradius)^2)});
\filldraw (\a,\b-0.4) circle (2pt);

\draw plot[domain=-\xradius:\xradius] ({\c+\x},{\d+1*\yradius * sqrt(1 - (\x / \xradius)^2)});
\draw plot[domain=-\xradius:\xradius] ({\c+\x},{\d-1*\yradius * sqrt(1 - (\x / \xradius)^2)});
\draw plot[domain=-\xradius:\xradius] ({\c+\x},{\d-2.5*\yradius * sqrt(1 - (\x / \xradius)^2)});

\draw[dashed] plot[domain=-\xradius:\xradius] ({\e+\x},{\f+1*\yradius * sqrt(1 - (\x / \xradius)^2)});
\draw plot[domain=-\xradius:\xradius] ({\e+\x},{\f-1*\yradius * sqrt(1 - (\x / \xradius)^2)});
\draw plot[domain=-\xradius:\xradius] ({\e+\x},{\f+2.5*\yradius * sqrt(1 - (\x / \xradius)^2)});
\draw[dashed] plot[domain=-\xradius:\xradius] ({\g+\x},{\h+1*\yradius * sqrt(1 - (\x / \xradius)^2)});
\draw plot[domain=-\xradius:\xradius] ({\g+\x},{\h-1*\yradius * sqrt(1 - (\x / \xradius)^2)});
\draw plot[domain=-\xradius:\xradius] ({\g+\x},{\h+2.5*\yradius * sqrt(1 - (\x / \xradius)^2)});
\filldraw (\g,\h+0.4) circle (2pt);

\end{tikzpicture}

    \caption{The local delooping isomorphism 
$\begin{bmatrix} \mathfrak d_1 \\ \mathfrak d_2 \end{bmatrix}$ its inverse $\begin{bmatrix} \mathfrak r_1,\mathfrak r_2 \end{bmatrix}$. }
    \label{Delooping isomorphism}
\end{figure}

Given an oriented crossing $c\in\{ \orCrosPos, \orCrosNeg \}$ we denote by $\llbracket c \rrbracket$ the tangle complex of $c$, or Bar-Natan complex of $c$, defined as 
\begin{align*}
\llbracket \orCrosPos \rrbracket&= \quad \cdots \longrightarrow 0 \xrightarrow{\qquad \ } u^0q^1 \,  \smoothingZero \ \xrightarrow{\text{saddle}} \  u^1q^2 \, \smoothingOne \ \xrightarrow{\quad \ } 0 \longrightarrow \cdots \\
\llbracket \orCrosNeg \rrbracket&= \quad  \cdots \longrightarrow 0 \longrightarrow u^{-1}q^{-2}\,  \smoothingOne \ \xrightarrow{\text{saddle}}  \  u^0q^{-2} \, \smoothingZero \ \longrightarrow 0 \longrightarrow \cdots
\end{align*}
over the additive category $\Cobl^R(B, \dot{B})$, where $\dot{B}$ are four points with orientation induced by $c$. Here $u^i$ denotes the homological degree and $q^j$ the $q$-degree.

The categories $\Cobl^R(B, \dot{B})$ compose well, see \cite[\S 5]{MR2174270}, allowing us to form tensor products between complexes $\llbracket c_1 \rrbracket$ and $\llbracket c_2 \rrbracket$, which in turn leads to tangle complexes $\llbracket T \rrbracket$ over $\Cobl^R(B, \dot{B})$ for any tangle diagram $T$ embedded in $B$ with endpoints given by $\dot{B}$.

Given a word $w$ in an alphabet $\{\sigma_i, \sigma^{-1}_i\mid 1\leq i\leq n-1\}$ for some $n\geq 2$, we get a tangle $T_w$ by using a standard tangle for each $\sigma_i^{\pm}$ as in Figure \ref{fig:braid_conv} and stacking them on top of each other. This gives rise to a braid on $n$ strands which can be closed to a link.

\begin{figure}[ht]
\[
\begin{tikzpicture}
\node at (1,1.75) {$\sigma_i$};
\node at (-0.75,1) {$\lineup{0.5}$};
\node at (0.25,1) {$\lineup{0.5}$};
\node at (1,1) {$\crosPs{0.5}$};
\node[scale = 0.6] at (-0.75,0.5) {$1$};
\node at (-0.25,1) {$\cdots$};
\node[scale = 0.6] at (0.25,0.5) {$i-1$};
\node[scale = 0.6] at (0.7,0.5) {$i$};
\node[scale = 0.6] at (1.3,0.5) {$i+1$};
\node at (1.8,1) {$\cdots$};
\node at (2.3,1) {$\lineup{0.5}$};
\node[scale = 0.6] at (2.3,0.5) {$n$};
\node at (5,1.75) {$\sigma_i^{-1}$};
\node at (3.25,1) {$\lineup{0.5}$};
\node at (4.25,1) {$\lineup{0.5}$};
\node at (5,1) {$\crosNg{0.5}$};
\node[scale = 0.6] at (3.25,0.5) {$1$};
\node at (3.75,1) {$\cdots$};
\node[scale = 0.6] at (4.25,0.5) {$i-1$};
\node[scale = 0.6] at (4.7,0.5) {$i$};
\node[scale = 0.6] at (5.3,0.5) {$i+1$};
\node at (5.8,1) {$\cdots$};
\node at (6.3,1) {$\lineup{0.5}$};
\node[scale = 0.6] at (6.3,0.5) {$n$};
\node at (8,0.5) {$\crosPs{0.5}$};
\node at (7.75,1) {$\lineup{0.5}$};
\node at (7.75,1.5) {$\lineup{0.5}$};
\node at (8.25,1) {$\lineup{0.5}$};
\node at (8.5,1.5) {$\crosNg{0.5}$};
\node at (9,1) {$\crosPs{0.5}$};
\node at (8.75,0.5) {$\lineup{0.5}$};
\node at (9.25,0.5) {$\lineup{0.5}$};
\node at (9.25,1.5) {$\lineup{0.5}$};
\end{tikzpicture}
\]
\caption{\label{fig:braid_conv} The tangles for $\sigma_i$, $\sigma^{-1}_i$, and $\sigma_1\sigma_3\sigma_2^{-1}$.}
\end{figure}
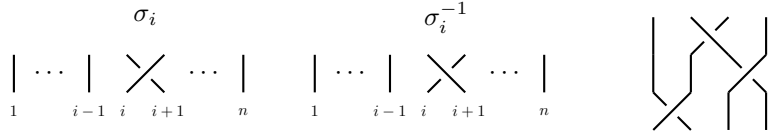

We note that our conventions for braid diagrams lead to the mirror of the braid diagrams in \cite{MR356023}. The reason we use our convention is that we want positive braid words to give rise to links whose Khovanov homology is concentrated in non-negative homological degrees.

Also, the tangle diagram $T_w$ fits in a rectangle with $n$ points on the bottom and $n$ points on the top. We simply write $\Cobl^R(B^n_n)$ for the corresponding category over which the Khovanov complex is defined. In later sections, the compact surface $B\subset \R^2$ is simply a disc with $\dot{B}$ consisting of $2n$ points. We will then write $\Cobl^R(2n)$ for $\Cobl^R(B,\dot{B})$.

\subsection{Gaussian elimination and algebraic discrete Morse theory}

The following standard lemma is a key tool for simplifying combinatorial chain complexes, both from a theoretical and an algorithmic viewpoint. 
\begin{lemma}[Gaussian elimination]\label{lemma: gaussian elimination}
    Let $(C,d)$ be a chain complex over an additive category $\mathfrak C$. Assume that for $n$ there are decomposition of the chain spaces $C^n=C^n_1 \oplus C^n_2$ and $C^{n+1}=C^{n+1}_1 \oplus C^{n+1}_2$ and differentials $d^{n-1}$, $d^n$, and $d^{n+1}$ so that
    $$
    C= \qquad \dots \rightarrow C^{n-1} \xrightarrow{\ \begin{bmatrix}
        \alpha \\ \beta
    \end{bmatrix} \ }  C^n_1 \oplus C^n_2 \xrightarrow{\ \begin{bmatrix}
        \varphi & \delta   \\ \gamma & \varepsilon
    \end{bmatrix} \ } C^{n+1}_1 \oplus C^{n+1}_2 \xrightarrow{\ \begin{bmatrix}
        \mu & \nu   
    \end{bmatrix} \ } C^{n+2} \rightarrow \cdots.
    $$
    If $\varphi\colon C^n_1 \to C_1^{n+1}$ is an isomorphism, then $C$ is chain homotopy equivalent to the complex $C'$ which in homological degrees $[n-1,n+2]$ is defined by 
    $$
    C'= \qquad  \dots \rightarrow C^{n-1} \xrightarrow{\ \beta \ }   C^n_2 \xrightarrow{ \varepsilon- \gamma \varphi^{-1} \delta }  C^{n+1}_2 \xrightarrow{\ \nu \ } C^{n+2} \rightarrow \cdots
    $$
    and outside these degrees $C'$ agrees with $C$.
\end{lemma}

Applying  Gaussian elimination once on a large chain complex $C$ often does not cut it and one needs to iterate Lemma \ref{lemma: gaussian elimination} in order to obtain a reasonable description of $C$. Since the Gaussian elimination introduces a new morphism between the remaining summands, it is not immediately obvious whether an isomorphism of $C$ continues to be a cancellable isomorphism after some number of other Gaussian eliminations have been performed. 
Algebraic discrete Morse theory \cite{MR2171225} is a framework which provides a sufficient condition for iteratively applying Lemma \ref{lemma: gaussian elimination} and a description of the resulting complex.

A \term{based chain complex} is a chain complex $(C,d)$ over an additive category $\mathfrak C$ and a fixed direct sum decomposition $C^i= \bigoplus_k C_k^i$ on every chain space $C^i$. The decomposition of the chain spaces also gives out a decomposition of the differentials: we call the maps
$$
d_{C^{i+1}_k,C^{i}_j}\colon C^i_j \to C^{i+1}_k, \quad d_{C^{i+1}_k,C^{i}_j}= \pi d^i \iota
$$
\term{matrix elements of} $C$, where $\iota$ and $\pi$ are the canonical inclusions and projections associated to $C^i_j$ and $C^{i+1}_k$. The based complex $C$ induces a directed graph $G(C)=(V,E)$, whose vertices $V$ are the direct summands $C_k^i$ and whose directed edges $E$ are the non-zero matrix elements of $C$.
Reversing a subset of edges $M\subset E$ gives out another graph $G(C,M)=(V',E')$ with $V'=V$ and 
$$
E'=(E\setminus M) \cup \{ b\to a \mid (a\to b)\in M \}. 
$$
We call $M$ a \term{Morse matching on} $C$, if
\begin{enumerate}
    \item $M$ is finite.\label{Condition: finiteness condition of Morse}
    \item $M$ is a matching, i.e., its edges are pairwise non-adjacent. \label{Condition: matching condition of Morse}
    \item For every edge $f\in M$, the corresponding matrix element $f$ is an isomorphism in $\mathfrak C$.\label{Condition: isomorphism condition of Morse}
    \item $G(C, M )$ has no directed cycles. \label{Condition: acyclicity condition of Morse}
\end{enumerate}

The graph $G(C,M)$ can be viewed as a category, whose objects and morphisms are the vertices and directed paths of $G(C,M)$. This allows us to define a functor $R\colon G(C,M) \to \mathfrak C$ on vertices by $R(C^i_k)=C^i_k$. On single edges, we set $R(f)=f$ if $f\notin M$ and $R(g)=-g^{-1}$, if $g\in M$ and on longer paths we extend $R$ functorially. If $M$ is a Morse matching on $C$, we can define the based \term{Morse complex} $(M(C), \partial )$, whose direct summands are the summands of $C$ which are not matched by $M$. The differentials $\partial^i$ are described by their matrix elements with
$$
\partial^i_{B,A}= \sum_{p\in \{ \text{paths: } A\to  B\} } R(p)
$$
where $A$ and $B$ are unmatched cell of $C^i$ and $C^{i+1}$ respectively. 

\begin{theorem}[Sköldberg \cite{MR2171225}]\label{Theorem: discrete Morse theory}
If $M$ is a Morse matching on $C$, then $C \simeq M(C)$.    
\end{theorem}
\begin{proof}[Sketch of proof]
    Inductively use Lemma \ref{lemma: gaussian elimination} on the isomorphisms of $M$. Condition \ref{Condition: acyclicity condition of Morse} ensures that the $k+1$:th isomorphism remains an isomorphisms after $k$ cancellations.    
\end{proof}

From the four conditions of Morse matching (\ref{Condition: matching condition of Morse}) is often easy to prove and it greatly simplifies the combinatorics of $G(C,M)$. If (\ref{Condition: matching condition of Morse}) holds, then no path can descend two homological degrees and to prove (\ref{Condition: acyclicity condition of Morse}) it suffices show that there exist no loops which zig-zag between two adjacent homological degrees. Similarly, all paths that contribute to the differentials $\partial$ of the Morse complex $M(C)$ have to zig-zag between adjacent homological degrees.

\section{Khovanov homology of 3-braids in polynomial time}\label{Section: 3braid polytime}

The braid group on $3$ strands is given by $\mathbf{B}_3 = \langle \sigma_1, \sigma_2\mid \sigma_1\sigma_2\sigma_1 = \sigma_2\sigma_1\sigma_2\rangle$. We denote the element $\Delta = \sigma_1\sigma_2\sigma_1$, so that $\Delta^2$ generates the center of $\mathbf{B}_3$.

The first step in our algorithm is to turn a word $w\in \mathbf{B}_3$ into its Murasugi normal form. Recall that Murasugi normal forms for conjugacy classes are given by

\begin{align*}
\Omega_0 &= \{ \Delta^{2k}\mid k\in \Z\},\\
\Omega_1 &= \{ \Delta^{2k}\sigma_1\sigma_2\mid k\in \Z\},\\
\Omega_2 &= \{ \Delta^{2k}(\sigma_1\sigma_2)^2\mid k\in \Z\},\\
\Omega_3 &= \{ \Delta^{2k+1}\mid k\in \Z\},\\
\Omega_4 &= \{ \Delta^{2k} \sigma_1^{-p}\mid k,p\in \Z, p>0\},\\
\Omega_5 &= \{ \Delta^{2k} \sigma_2^q \mid k,q\in \Z, q>0\},\\
\Omega_6 &= \{\Delta^{2k} \sigma_1^{-p_1}\sigma_2^{q_1}\cdots \sigma_1^{-p_r}\sigma_2^{q_r} \mid k,r,p_i,q_i\in \Z, r,p_i,q_i>0\}.
\end{align*}

It is shown in \cite[Prop.2.1]{MR356023} that any word $w\in \mathbf{B}_3$ is conjugate to an element of $\Omega_0\cup \cdots \cup \Omega_6$. We first want to ensure that we can obtain this normal form in polynomial time with respect to the length of $w$. We denote the length of a word $w$ by $|w|$.

\begin{proposition}\label{prop:alg_normalform}
There is an algorithm $\mathcal{A}_{Mnf}$ which turns a word $w\in \mathbf{B}_3$ into its Murasugi normal form $N(w)$, and which runs in linear time with respect to the length of $w$. Furthermore, the length of $N(w)$ as a word in $\Delta, \sigma_1, \sigma_2$ is linearly bounded in $|w|$.
\end{proposition}

\begin{proof}
We closely follow the proof of \cite[Prop.2.1]{MR356023}, turning it into an algorithm with the required properties.

The braid group $\mathbf{B}_3$ is isomorphic to $G=\langle a,b\mid a^2=b^3\rangle$, using the isomorphism $\varphi\colon G\to \mathbf{B}_3$ given by $\varphi(a) = \sigma_1\sigma_2\sigma_1$ and $\varphi(b) = \sigma_1\sigma_2$. Then inverse satisfies $\varphi^{-1}(\sigma_1) = b^{-1}a$ and $\varphi^{-1}(\sigma_2) = a^{-1}b^2$. Write $c = a^2 = b^3$ for the element that generates the center of $G$. Then
\[
\varphi^{-1}(\sigma_1) = c^{-1}b^2a, \hspace{0.5cm} \varphi^{-1}(\sigma_1^{-1}) = c^{-1}ab, \hspace{0.5cm}\varphi^{-1}(\sigma_2) = c^{-1}ab^2,\hspace{0.5cm}\varphi^{-1}(\sigma_2^{-1}) = c^{-1}ba,
\]
and we get $\varphi^{-1}(w) = c^{-|w|}v_1\cdots v_{|w|}$ with each $v_i\in \{b^2a, ab,ab^2, ba\}$. The length of $\varphi^{-1}(w)$ in $a,b,c$ is at most $4|w|$. Some of the $v_iv_{i+1}$ can contain $a^2$ or $b^3$ in which case we replace them with $c$ and move it to the front. That is, we can write $\varphi^{-1}(w)$ as
\begin{equation}\label{eq:ab_alm_norm_form}
c^n, c^na, c^nb^l, \mbox{ or }c^na^{\varepsilon_1}b^{l_1}ab^{l_2}\cdots ab^{l_m}a^{\varepsilon_2},
\end{equation}
with $n\in \Z$, $l,l_1,\ldots, l_m\in \{1,2\}$, $\varepsilon_1,\varepsilon_2\in \{0,1\}$. Furthermore, the length of this word in $a,b,c$ is still at most $4|w|$, and this form can be obtained in linear time depending on $|w|$.

We want to get $\varphi^{-1}(w)$ to be conjugate to one of
\begin{equation}\label{eq:ab_norm_form}
c^n, c^na, c^nb^l, \mbox{ or }c^nab^{l_1}ab^{l_2}\cdots ab^{l_m},
\end{equation}
that is, in the case of the last word in (\ref{eq:ab_alm_norm_form}) with $\varepsilon_1=1$ and $\varepsilon_2=0$. Let us check the other cases for this word. If $\varepsilon_1=0$ and $\varepsilon_2 = 1$, we have
\[
c^nb^{l_1}ab^{l_2}\cdots ab^{l_m}a \sim c^nab^{l_1}\cdots ab^{l_m},
\]
where $\sim$ stands for the conjugation relation. The latter is in the form of (\ref{eq:ab_norm_form}).

If $\varepsilon_1=0=\varepsilon_2$, we have
\[
c^nb^{l_1}ab^{l_2}\cdots ab^{l_m} \sim c^nab^{l_2}\cdots ab^{l_m+l_1}
\]
with $l_m+l_1\in \{2,3,4\}$. If $l_m+l_1=2$, this is in the form of (\ref{eq:ab_norm_form}). If $l_m+l_1 = 4$ we get an extra $c$, but it is also in the form of (\ref{eq:ab_norm_form}). If $l_m+l_1 = 3$, we get another $c$ and we have case $\varepsilon_1=1=\varepsilon_2$, but with $m$ replaced by $m-2$. Note that we can assume $m\geq 2$, since for $m=1$ the original word is $c^nb^l$. Also, if $m = 2$, we now have the word $c^{n+1}a^2 = c^{n+2}$, which is also in the form of (\ref{eq:ab_norm_form}).

So, now if $\varepsilon_1=1=\varepsilon_2$, we have
\[
c^nab^{l_1}\cdots ab^{l_m}a \sim c^na^2b^{l_1}\cdots ab^{l_m} = c^{n+1}b^{l_1}\cdots ab^{l_m},
\]
which brings us back to the case $\varepsilon_1=0=\varepsilon_2$. Nevertheless, after at most $m$ steps, we get $\varphi^{-1}(w)$ to be in the form of (\ref{eq:ab_norm_form}).

Note that $\varphi(c) = \Delta^2$, $\varphi(ab) = \Delta^2\sigma_1^{-1}$ and $\varphi(ab^2) = \Delta^2\sigma_2$, so applying $\varphi$ to a word in (\ref{eq:ab_norm_form}) lands in one of $\Omega_0,\ldots,\Omega_6$, and this is the required normal form $N(w)$.
\end{proof}

For words in $\Omega_0, \Omega_1, \Omega_2$ and $\Omega_3$ the integral Khovanov homology of the corresponding braid closure is explicitly given in \cite[Cor.5.7]{MR4430925}. For $\Omega_4$ and $\Omega_5$, the integral Khovanov homology can be obtained using \cite[Alg.1]{kelo2024discrete}. We will give a slightly different approach using \cite{schuetz2025kh3braids}, and express the Khovanov homology of the braid closure of $\Delta^{2k}w\in \Omega_4\cup \Omega_5\cup \Omega_6$ in terms of the Khovanov homology of the torus link and the braid closure of $w$.

\begin{theorem}\label{thm:case_omega_4}
Let $L$ be the braid closure corresponding to the word $\Delta^{2k}\sigma_1^{-l}$ with $k,l\geq 1$. Then
\begin{align*}
\HKh^{4k, 12k-l\pm 1}(L) &= \Z \\
\HKh^{4k-1, 12k-l-1}(L) &= \left\{
\begin{array}{cc}
\Z & l = 1, \\
0 & l\geq 2,
\end{array}
\right.\\
\HKh^{4k-1, 12k-l-3}(L) &= \left\{
\begin{array}{cc}
\Z/2\Z & l = 1,\\
0 & l\geq 2,
\end{array}
\right.\\
\HKh^{4k-2, 12k-l-3}(L) &= \left\{
\begin{array}{cc}
0 & l = 1,\\
\Z^2 & l = 2,\\
\Z& l\geq 3,
\end{array}
\right.\\
\HKh^{4k-2, 12k-l-5}(L) &= \left\{
\begin{array}{cc}
\Z & l = 1,\\
\Z^2 & l = 2,\\
\Z/2\Z& l\geq 3,
\end{array}
\right.\\
\HKh^{4k-3, 12k-l-7}(L) &= \left\{
\begin{array}{cc}
0 & l\leq 2,\\
\Z & l\geq 3,
\end{array}
\right.\\
\HKh^{4k-4, 12k-l-7}(L) &= \left\{
\begin{array}{cc}
\Z & l\leq 3,\\
\Z^3 & l = 4, \\
\Z^2 & l \geq 5,
\end{array}
\right.\\
\HKh^{4k-4, 12k-l-9}(L) &= \left\{
\begin{array}{cc}
\Z & l\leq 3,\\
\Z^2 & l = 4,\\
\Z/2\Z & l\geq 5.
\end{array}
\right.
\end{align*}
For any integer $j$ we have
\[
\HKh^{4k-5, j}(L) = \HKh^{4k-5, j+l}(T(3, 3k)) \oplus \HKh^{-5, j-12k+1}(T(2, -l)),
\]
and for any integers $i \leq 4k-6$ and $j$ we have
\[
\HKh^{i,j}(L) = \HKh^{i, j+l}(T(3, 3k)) \oplus \HKh^{i-4k, j-12k+1}(T(2,-l))\oplus \HKh^{i-4k, j-12k-1}(T(2,-l)).
\]
If $k = 1$ we also have
\[
\HKh^{0, 12-l-5}(L) = \left\{
\begin{array}{cc}
0 & l \leq 3,\\
\Z & l \geq 4,
\end{array}
\right.
\]
and if $k\geq 2$ we also have
\begin{align*}
\HKh^{4k-3, 12k-l-3}(L) &= \Z\\
\HKh^{4k-3, 12k-l-5}(L) &= \left\{
\begin{array}{cc}
\Z & l \leq 3,\\
\Z/2\Z & l \geq 4.
\end{array}
\right.
\end{align*}
In all other bigradings, the homology groups are $0$.
\end{theorem}

We delay the proof of this theorem until the next section.

\begin{theorem}\label{thm:case_omega_5}
Let $L$ be the braid closure corresponding to the word $\Delta^{2k}\sigma_2^l$ with $k,l\geq 1$. For $i < 4k$ we have
\[
\HKh^{i,j}(L) = \HKh^{i, j-l}(T(3, 3k)).
\]
For $i > 4k+1$ we have
\[
\HKh^{i,j}(L) = \HKh^{i-4k, j-12k+1}(T(2,l))\oplus \HKh^{i-4k, j-12k-1}(T(2,l)).
\]
The other non-zero homology groups are given by
\begin{align*}
\HKh^{4k,12k+l-3}(L) &= \Z, \\
\HKh^{4k, 12k+l-1}(L) &= \Z^2,\\
\HKh^{4k+1, 12k+l+1}(L) &= \Z/2\Z,\\
\HKh^{4k+1, 12k+l+3}(L) &= \Z.
\end{align*}
\end{theorem}

This theorem follows directly from \cite[Thm.6.4]{schuetz2025kh3braids}.

\begin{theorem}\label{thm:case_omega_6}
Let $w = \sigma_1^{-p_1}\sigma_2^{q_1}\cdots \sigma_1^{-p_r}\sigma_2^{q_r}$ with positive integers $r, p_1,q_1,\ldots, p_r, q_r$, let $L_w$ be the braid closure corresponding to $w$ and $L$ be the braid closure corresponding to $\Delta^{2k}w$ with $k\geq 1$. Then for integers $i \not=4k, 4k+1$ and $j$ we have
\begin{equation}\label{eq:generic_formula}
\HKh^{i,j}(L) = \HKh^{i,j-t}(T(3, 3k))\oplus \HKh^{i-4k, j-12k}(L_w),
\end{equation}
with
\[
t = q_1+\cdots +q_r - (p_1 + \cdots + p_r).
\]
The remaining non-zero homology groups are given by
\begin{align*}
\HKh^{4k, 12k+t-1}(L) &= \HKh^{0,t-1}(L_w),\\
\HKh^{4k, 12k+t+1}(L) &= \HKh^{0,t+1}(L_w)/\Z,\\
\HKh^{4k+1, 12k+t+1}(L) &= \HKh^{1,t+1}(L_w)\oplus \Z/2\Z,\\
\HKh^{4k+1, 12k+t+3}(L) &= \HKh^{1,t+3}(L_w)\oplus \Z.
\end{align*}
\end{theorem}

Notice that $\HKh^{0, t+1}(L_w)$ contains at least one direct summand of $\Z$, and here $\HKh^{0,t+1}(L_w)/\Z$ refers to the homology group with this summand of $\Z$ removed. The proof of this theorem is also delayed until the next section.

In all three cases, the homology behaves as in (\ref{eq:generic_formula}) with a few exceptional bidegrees. Perhaps surprisingly, we have more exceptions in the case corresponding to $\Omega_4$ than to $\Omega_6$. This is because the closure of the braid word $\sigma_1^{-l}$ is a split link whose Khovanov homology is not as thin as the Khovanov homology corresponding to the alternating words used in $\Omega_6$.

Theorems \ref{thm:case_omega_4} - \ref{thm:case_omega_6} assume $k\geq 1$. For negative $k$, observe that inverting braid words flip the sets $\Omega_4$ and $\Omega_5$, while the inverse of $\Delta^{2k}\sigma_1^{-p_1}\sigma_2^{q_1}\cdots \sigma_1^{-p_r}\sigma_2^{q_r}$ is conjugate to $\Delta^{-2k}\sigma_1^{-q_r}\sigma_2^{p_r}\cdots \sigma_1^{-q_1}\sigma_2^{p_1}$. We can therefore obtain the Khovanov homology by dualizing the Khovanov homology of the mirror.

The Khovanov homology of $T(3,3k)$ and $T(2,\pm l)$ is given explicitly in \cite[Cor.5.7]{MR4430925} and \cite[\S 6.2]{MR1740682}, so the only remaining ingredient is the Khovanov homology of the braid closure of $\sigma_1^{-p_1}\sigma_2^{q_1}\cdots \sigma_1^{-p_r}\sigma_2^{q_r}$. There does not seem to be a simple formula for it, but since these links are alternating and non-split, we can derive it from the Jones polynomial and the signature, see \cite{MR4407084}. The signature for our braid closure is $t = q_1+\cdots +q_r - (p_1 + \cdots + p_r)$, compare \cite[Cor.8.3]{schuetz2025kh3braids}. For the convenience of the reader, and also to see that deriving the Khovanov homology is done in polynomial time, we now give an explicit formula for the Khovanov homology of a quasi-alternating oriented link $L$, given the Jones polynomial and the signature $s=s(L)$.

First note that the Jones polynomial is usually given as a Laurent polynomial in a variable $t^{\frac{1}{2}}$. In order to be closer to Khovanov homology we use the substitution $q=t^2$, and we write the Jones polynomial of $L$ as
\[
J_L(q) = \sum_{j\in \Z}a_j q^j.
\]
It is worth noting that $a_j = 0$ unless $j\in 1+|L|+2\Z$, and only finitely many $a_j$ are non-zero. The reduced Khovanov homology groups are easily obtained from the coefficients $a_j$. Namely, by \cite{MR2509750} the homology groups are free abelian, with non-zero groups only occuring in bidegrees $(j,s+2j)$, and the Betti numbers given by
\[
\tilde{b}^{j,s+2j} = |a_{s+2j}|.
\]
Expressing the unreduced Khovanov homology is slightly more difficult. We use the reduced Bar-Natan--Lee--Turner spectral sequence over $\Z$, which for quasi-alternating links collapses at the $E_2$-page, with these groups being free abelian and of total rank $2^{|L|-1}$, compare \cite{MR4873797}.

In particular, there exists a homomorphism $d\colon \rKh^{j,s+2j}(L)\to \rKh^{j+1, s+2j+2}(L)$ with $d^2=0$ and the corresponding homology groups $E_2^{j,s+2j}$ free abelian. Let us denote the Betti numbers for these groups by $\tilde{b}_{\mathrm{BLT}}^{j,s+2j}(L)$. Then $0\leq \tilde{b}_{\mathrm{BLT}}^{j,s+2j}(L)\leq \tilde{b}^{j,s+2j}(L)$. The exact values can be determined using linking numbers, but for now we simply note that $\tilde{b}_{\mathrm{BLT}}^{j,s+2j}(L)=0$ for odd $j$, and $\tilde{b}_{\mathrm{BLT}}^{0,s}(L)\geq 1$. Define
\[
\overline{b}^j = \tilde{b}^{j,s+2j}(L) - \tilde{b}_{\mathrm{BLT}}^{j,s+2j}(L).
\]

\begin{lemma}\label{lem:red2unred}
Let $L$ be a quasi-alternating oriented link with signature $s$. The Betti numbers of unreduced Khovanov homology are given by
\begin{align*}
b^{j,s+2j-1}(L) &= \sum_{i = j_{\min}}^j(-1)^{j+i}\overline{b}^i+\tilde{b}_{\mathrm{BLT}}^{j,s+2j}(L),\\
b^{j,s+2j+1}(L) &= \sum_{i = j_{\min}}^{j-1}(-1)^{j+i+1}\overline{b}^i+\tilde{b}_{\mathrm{BLT}}^{j,s+2j}(L).
\end{align*}
Here $j_{\min} = \min\{j\in \Z\mid \tilde{b}^{j,s+2j}\not=0\}$.
All other Betti numbers are $0$. Any torsion elements have order $2$, and the only possibly non-zero torsion coefficients are given by
\[
t_2^{j,s+2j-1}(L) = \sum_{i = j_{\min}}^{j-1}(-1)^{j+i+1}\overline{b}^i.
\]
\end{lemma}

\begin{proof}
We have the long exact sequence
\[
0\longrightarrow \HKh^{j,s+2j+1}(L)\longrightarrow \widetilde{\HKh}{}^{j,s+2j}(L)\stackrel{\delta}{\longrightarrow}\widetilde{\HKh}{}^{j+1,s+2(j+1)}(L)\longrightarrow \HKh^{j+1,s+2(j+1)-1}(L)\longrightarrow 0,
\]
with $\delta = 2d$ by \cite[Lm.5.6]{MR4873797}. It follows that
\begin{equation}\label{eq:betti-1}
b^{j,s+2j-1} = \tilde{b}^{j,s+2j}-\tilde{b}^{j-1,s+2(j-1)}+b^{j-1,s+2(j-1)+1},
\end{equation}
and since the reduced Khovanov homology is torsion-free, we have
\[
b^{j,s+2j+1} = \rank \ker (d\colon \widetilde{\HKh}{}^{j,s+2j}(L) \to \widetilde{\HKh}{}^{j+1,s+2(j+1)}(L)).
\]
Note that we here drop reference to $L$ in the Betti numbers. Let us write $d^j$ to indicate the homological degree of the domain of $d$. Since $(\widetilde{\HKh}{}^{j,s+2j}(L),d^j)$ form a cochain complex whose homology is determined by the $\tilde{b}^j_{\mathrm{BLT}}$, we have
\begin{equation}\label{eq:betti+1}
b^{j,s+2j+1} = \tilde{b}^j_{\mathrm{BLT}}+\rank(d^{j-1}) = \tilde{b}^j_{\mathrm{BLT}}+\tilde{b}^{j-1}-b^{j-1,s+2(j-1)+1}.
\end{equation}
The statement about the Betti numbers follow via an induction over $j$, starting with $j_{\min}$, using (\ref{eq:betti-1}) and (\ref{eq:betti+1}). We leave the details to the reader.

For the statement about torsion coefficients, note that $\HKh^{j,s+2j+1}(L)$ is a subgroup of a torsion-free abelian group, hence torsion-free. Since $\delta = 2d$ and the second page of the reduced integral BLT-spectral sequence is torsion-free, the torsion coefficient $t_2^{j,s+2j-1}$ agrees with $\rank(d^{j-1})$. The statement therefore follows from (\ref{eq:betti+1}).
\end{proof}

\begin{algorithm}
\caption{Polynomial time algorithm for Khovanov homology of $3$-braids}  
 \textbf{Input} Braid word $w$ on the alphabet $\{\sigma_1, \sigma_2, \sigma_1^{-1}, \sigma_2^{-1}$ \}  \\
  \textbf{Output}  Khovanov homology of the braid closure $\HKh(L_w) $
   \begin{algorithmic}[1]
\State $\triangle^{2k} v \gets$ Murasugi normal form of $w$
\If{$\triangle^{2k} v \in \Omega_0\cup \Omega_1\cup \Omega_2 \cup \Omega_3$}
    \State Read off $\HKh(L_w)$ from \cite[Cor.5.7]{MR4430925}. 
\EndIf
\If{$\triangle^{2k} v \in \Omega_4\cup \Omega_5$}
    \State Enforce $k\geq 0$ by mirroring and rotating the braid.
    \State If $k=0$, then $L_w=T(2,l)\sqcup (\text{unknot})$ and $\HKh(L_w)$ is obtained with \cite[\S 6.2]{MR1740682}.
    \State Otherwise, get $\HKh(L_w)$ with \cite[\S 6.2]{MR1740682}, \cite[Cor.5.7]{MR4430925}, Theorem \ref{thm:case_omega_4}, and Theorem \ref{thm:case_omega_5}.
\EndIf
\If{$\triangle^{2k} v \in \Omega_6$}
    \State Enforce $k\geq 0$ by mirroring and rotating the braid.
    \State Compute $J_{L}(q)$ where $L=L_v$ with polynomial time algorithm of  \cite{MR1041170}.
    \State Compute the signature of $L_v$.
    \State Compute the linking numbers and from them the BLT-generators of $L_v$.
    \State Use Lemma \ref{lem:red2unred} to combine the above into $\HKh (L_v)$
    \State If $k=0$ then $\HKh (L_w)=\HKh (L_v)$.
    \State If $k> 0$ then use $\HKh(L_v)$, \cite[Cor.5.7]{MR4430925}, and Theorem \ref{thm:case_omega_6} to obtain $\HKh(L_w)$.
\EndIf
\end{algorithmic}
\end{algorithm}

\begin{proof}[Proof of Theorem \ref{thm:alg3braids}]
Let $w$ be a braid word in $\{\sigma_1,\sigma_2,\sigma_1^{-1},\sigma_2^{-1}\}$. By Proposition \ref{prop:alg_normalform} we get the Murasugi normal form $N(w)$ of $w$ in linear time. If $N(w)\in \Omega_0\cup \Omega_1\cup \Omega_2\cup \Omega_3$, we can read off the Khovanov homology of $L_w=L_{N(w)}$ from \cite[Cor.5.7]{MR4430925}. If $N(w)\in \Omega_4\cup \Omega_5$ with $k\not=0$ we can read off the Khovanov homology using Theorem \ref{thm:case_omega_4} and \ref{thm:case_omega_5} with \cite[\S 6.2]{MR1740682}. If $k=0$, $L_w$ is a split union of a $2$-stranded torus link with an unknot, so we can read off its Khovanov homology using \cite[\S 6.2]{MR1740682}.

If $N(w)\in \Omega_6$, write $N(w) = \Delta^{2k}v$ with $k\in\Z$ and $v = \sigma_1^{-p_1}\sigma_2^{q_1}\cdots \sigma_1^{-p_r}\sigma_2^{q_r}$. The Jones polynomial of $L_v$ can be obtained in polynomial time, see \cite{MR1041170}, and therefore we can obtain the Khovanov homology of $L_v$ in polynomial time from Lemma \ref{lem:red2unred}. Note that to compute the $\tilde{b}^j_{\mathrm{BLT}}(L_v)$, we need to work out the number of components of $L_v$, as well as the differences in linking numbers. Since there can be at most three components, this can be done in linear time. If $k\not=0$, we use Theorem \ref{thm:case_omega_6} to read off the Khovanov homology of $L_w$.
\end{proof}

The algorithm for determining the Jones polynomial of a braid closure given by Morton and Short \cite{MR1041170} runs in cubic time. Particularly for $3$-braids this is very fast. Bar-Natan's algorithm \cite{MR2320156} on the other hand can be very slow for certain $3$-braids. We do not need to look at the details of this algorithm to see why. The main reason is that both Bar-Natan's scanning algorithm and the divide-and-conquer algorithm produce a specific basis for a cochain complex chain homotopy equivalent to the Khovanov complex. Since there are links whose total Betti number grows exponentially, these algorithms spend exponential time on them.

A very simple class of such links is given by $3$-stranded {\em weaving links}. For a positive integer $n$, the weaving link $W(3,n)$ is the closure of the $3$-braid word $(\sigma_1^{-1}\sigma_2)^n$. The determinant of $W(3,n)$ is given by
\begin{equation}\label{eq:exp_det}
\det(W(3,n)) = \left(\frac{3+\sqrt{5}}{2}\right)^n+\left(\frac{3-\sqrt{5}}{2}\right)^n-2,
\end{equation}
see \cite[Thm.2.2]{MR4561084}. Since these weaving links are alternating, the determinant agrees with the total Betti number of reduced Khovanov homology.

For example, the determinant of $W(3,20)$, a knot with $40$ crossings, is 228,826,125. Attempting to run the scanning algorithm on this knot almost certainly leads to a memory overflow, while computing the Jones polynomial is done practically immediately. 

\begin{proof}[Proof of Theorem \ref{thm:slowscanning}]
Applying Bar-Natan's scanning algorithm or the divide-and-conquer algorithm to $W(3,n)$ leads to a basis with more than $\det(W(3,n))$ elements, from which the Khovanov homology is determined via standard linear algebra methods. As this basis grows exponentially in $n$ by (\ref{eq:exp_det}), the number of steps to do this is exponential in the number of crossings of $W(3,n)$.

Now consider the positive word $v = \sigma_1\sigma_2\sigma_1^2\sigma_2^2$. Then 
\[
v = \sigma_1\sigma_2\sigma_1^2\sigma_2^2 = \sigma_1\sigma_2\sigma_1\sigma_1\sigma_2\sigma_1\sigma_1^{-1}\sigma_2 = \Delta^2\sigma_1^{-1}\sigma_2.
\]
In particular, $v^n = \Delta^{2n}(\sigma_1^{-1}\sigma_2)^n$ has $6n$ crossings, and its total Betti number grows exponentially in $n$ by Theorem \ref{thm:case_omega_6}.
\end{proof}

Another family of links with exponentially large Betti number is constructed in Section \ref{section: Asymptotict nonvanishing}.

\section{Proofs of Theorem \ref{thm:case_omega_4} and \ref{thm:case_omega_6}}

As with Theorem \ref{thm:case_omega_5}, both theorems will follow from results in \cite{schuetz2025kh3braids}. However, we need to fill in more details in these cases. It will be useful to treat links as based links. This way we can view the Khovanov complexes as complexes over $A=\Z[X]/(X^2)$, even though we only view the homology as abelian groups. The base point is placed on the middle strand of the diagram corresponding to the braid word. 

\subsection{Proof of Theorem \ref{thm:case_omega_4}}
Let $w=\Delta^{2k}\sigma_1^{-l}$ with $k,l\geq 1$, and let $T_{k,l}$ be the tangle obtained from the braid word $w$ after connecting the two endpoints of the leftmost strand.

Then $q^{l-6k}\llbracket T_{k,l}\rrbracket$ is chain homotopy equivalent to a sub-complex of $\mathcal{B}_k$ depicted in Figure \ref{fig:big_complex}. More precisely, the sub-complex is obtained by removing all objects from the top row of the form $q^{6k-2r+1}$\smoothingZero for $r > l$. This follows from \cite[Prop.7.3]{schuetz2025kh3braids}.

\begin{figure}[ht]
\begin{center}
\begin{tikzpicture}
\node at (-6,1) {$\cdots$};
\node at (-5, 1) {$q^{-2k-5}\smoothingZero$};
\node at (-2.5,1) {$q^{-2k-3}\smoothingZero$};
\node at (0,1) {$q^{-2k-1}\smoothingZero$};
\node at (2.5,0) {$q^{-1}\smoothingZero$};
\node at (2.5,1) {$q^{-2k+1}\smoothingZero$};
\node at (5,0) {$\smoothingOne$};
\node at (5,1) {$q^{-2k+3}\smoothingZero$};
\node at (7.5,0) {$q^2\smoothingOne$};
\node at (7.5,1) {$q^{-2k+5}\smoothingZero$};
\node at (8.5,0.5) {$\cdots$};
\drawblack{-4.2,1}{-3.3,1}{}{e}
\drawblack{0.8,0.9}{2,0.1}{sloped}{D^k}
\drawblack{0.8,1}{1.7,1}{}{e}
\drawblack{3.1,0}{4.6,0}{}{S}
\drawblack{3.3,0.9}{4.6,0.1}{sloped}{-SD^{k-1}}
\drawblack{5.4,0}{7,0}{}{c}
\drawblack{5.8,1}{6.7,1}{}{e}
\drawblack{5.8,0.9}{7,0.1}{sloped}{SD^{k-1}}
\end{tikzpicture}

\begin{tikzpicture}[scale = 0.9, transform shape]
\node at (-0.8,0.5) {$\cdots$};
\node at (0,0) {$q^{6k-10}\smoothingOne$};
\node at (0,1) {$q^{6k-11}\smoothingZero$};
\node at (2.5,0) {$q^{6k-8}\smoothingOne$};
\node at (2.5,1) {$q^{6k-9}\smoothingZero$};
\node at (5,-1) {$q^{6k-7}\smoothingZero$};
\node at (5,0) {$q^{6k-6}\smoothingOne$};
\node at (5,1) {$q^{6k-7}\smoothingZero$};
\node at (7.5,-1) {$q^{6k-5}\smoothingZero$};
\node at (7.5,0) {$q^{6k-6}\smoothingOne$};
\node at (7.5,1) {$q^{6k-5}\smoothingZero$};
\node at (10,0) {$q^{6k-4}\smoothingOne$};
\node at (10,1) {$q^{6k-3}\smoothingZero$};
\node at (12.5,0) {$q^{6k-2}\smoothingOne$};
\node at (15,0) {$q^{6k}\smoothingOne$};
\drawblack{0.8,0}{1.8,0}{}{d}
\drawblack{0.8,0.9}{1.8,0.1}{sloped}{-SD}
\drawblack{3.2,0.8}{4.3,-0.8}{sloped, near start}{D}
\drawblackw{3.2,0}{4.3,0}{near end}{c}
\drawblack{3.2,1}{4.3,1}{}{e}
\drawblack{5.7,1}{6.8,0.1}{sloped}{-S}
\drawblack{5.7,0.8}{6.8,-0.8}{sloped, near start}{-D}
\drawblackw{5.7,-0.9}{6.8,-0.1}{sloped,near start}{S}
\drawblack{5.7,-1}{6.8,-1}{}{e}
\drawblackw{5.7,-0.1}{6.8,-0.9}{sloped, near start}{S}
\drawblack{8.2,0}{9.3,0}{}{c}
\drawblack{8.2,1}{9.3,1}{}{e}
\drawblack{8.2,0.9}{9.3,0.1}{sloped}{S}
\drawblack{10.7,0.9}{11.8,0.1}{sloped}{-S}
\drawblack{10.7,0}{11.8,0}{}{d}
\drawblack{13.2,0}{14.4,0}{}{c}
\draw[dashed] (-0.8,1.5) -- (15,1.5);
\draw[dashed] (-0.8,-1.5) -- (15, -1.5);
\end{tikzpicture}

\begin{tikzpicture}
\node at (0,0) {$e=2\,\smoothingZero$};
\node[scale = 0.7] at (0.35,0) {$\bullet$};
\node at (3,0) {$d= \smoothingOne + \smoothingOne$};
\node at (6,0) {$c = \smoothingOne - \smoothingOne$};
\node[scale = 0.7] at (2.95,-0.09) {$\bullet$};
\node[scale = 0.7] at (3.7,0.09) {$\bullet$};
\node[scale = 0.7] at (5.95,-0.09) {$\bullet$};
\node[scale = 0.7] at (6.7,0.09) {$\bullet$};
\end{tikzpicture}
\end{center}
\caption{\label{fig:big_complex}The cochain complex $\mathcal{B}_k$ over $\Cobl^\Z(B^2_2)$. The object with grading shift $q^{6k}$ is in homological degree $4k$, and the object with grading shift $q^{-1}$ is in homological degree $0$. The complex continues in negative homological degrees via the two objects connected by the $e$-morphism. The letter $S$ stands for a surgery between the two smoothings, and $D$ stands for back-and-forth surgeries.}
\end{figure}

Now consider the functor $G\colon \Cobl^\Z(B^2_2)\to \mathfrak{Mod}^q_A$ obtained by connecting the two endpoints on the right and identifying $\Cobl^\Z(B^1_1)$ with finitely generated, free $q$-graded modules over $A=\Z[X]/(X^2)$. Then $G(c) = 0$, so $G(\mathcal{B}_k)$ decomposes into a direct sum of relatively small cochain complexes. These cochain complexes have been analysed in detail in the proof of \cite[Thm.7.4]{schuetz2025kh3braids}. In particular, $G(\mathcal{B}_k)$ is chain homotopy equivalent to the complex $\mathcal{C}_k$ given in Figure \ref{fig:bet_complex}. Moreover, $G(q^{l-6k}\llbracket T_{k,l}\rrbracket)$ is chain homotopy equivalent to the sub-complex of $\mathcal{C}_k$ obtained by removing all gray objects of homological degree less than $4k-l$.

\begin{figure}[ht]
\begin{center}
\begin{tikzpicture}
\node at (0,1.5) {$\cdots$};
\node[color=gray] at (1, 2) {$q^{-2k-4}A$};
\node[color=gray] at (1,1) {$q^{-2k-6}A$};
\node[color=gray] at (4,2) {$q^{-2k-2}A$};
\node[color=gray] at (4,1) {$q^{-2k-4}A$};
\node[color=gray] at (7,2) {$q^{-2k}A$};
\node[color=gray] at (7,1) {$q^{-2k-2}A$};
\node[color=gray] at (10,2) {$q^{-2k+2}A$};
\node[color=gray] at (10,1) {$q^{-2k}A$};
\node at (10,0) {$u^0q^{-2}A$};
\node at (12, 0) {$0$};
\node at (12.5,1) {$\cdots$};
\draw[color = gray, ->] (1.8,2) -- node [above, scale = 0.7] {$2X$} (3.2,2);
\draw[color = gray, ->] (1.8,1) -- node [above, scale = 0.7] {$2X$} (3.2,1);
\draw[color = gray, ->] (7.8,1.9) -- node [above, sloped, scale = 0.7, very near end] {$(2X)^{k-1}$} (9.4, 0.2);
\draw[color = gray, ->] (7.8, 0.9) -- node [above, sloped, scale = 0.7] {$X^k$} (9.4, 0.1);
\draw[-, line width=6pt, color = white] (7.8,1) -- (9.2,1);
\draw[color = gray, ->] (7.8,2) -- node [above, scale = 0.7] {$2X$} (9.2,2);
\draw[color = gray, ->] (7.8,1) -- node [above, scale = 0.7] {$2X$} (9.2,1);
\end{tikzpicture}

\begin{tikzpicture}
\node at (-0.2,1) {$\cdots$};
\node[color = gray] at (1,2) {$q^{6k-4-8i}A$};
\node[color = gray] at (1,1) {$q^{6k-6-8i}A$};
\node[color = gray] at (4,2) {$q^{6k-2-8i}A$};
\node[color = gray] at (4,1) {$q^{6k-4-8i}A$};
\node at (4,0) {$q^{6k-4-6i}A$};
\node at (2.5,0) {$u^{4(k-i)-2}$};
\node[color = gray] at (7,2) {$q^{6k-8i}A$};
\node[color = gray] at (7,1) {$q^{6k-2-8i}A$};
\node at (7,0) {$q^{6k-2-6i}A$};
\node[color = gray] at (10,2) {$q^{6k+2-8i}A$};
\node[color = gray] at (10,1) {$q^{6k-8i}A$};
\node at (10,0) {$q^{6k-2-6i}A$};
\node at (13,0) {$q^{6k+2-6i}A$};
\node at (14,1) {$\cdots$};
\node at (13, 2) {$1\leq i \leq k-1$};
\draw[->, color = gray] (1.9, 1.9) -- node [sloped, scale = 0.7, above, very near end] {$(2X)^i$} (3.1, 0.1);
\draw[->, color = gray] (1.9, 2) -- node [scale = 0.7, above] {$2X$} (3.1, 2);
\draw[-, line width = 6pt, color = white] (1.9, 1) -- (3.1, 1);
\draw[->, color = gray] (1.9, 1) -- node [above, scale = 0.7] {$2X$}(3.1,1);
\draw[->, color = gray] (4.9, 1.9) -- node [sloped, scale = 0.7, above] {$(2X)^i$} (6.1, 0.1);
\draw[->] (4.9, 0) -- node [above, scale = 0.7] {$2X$} (6.1, 0);
\draw[->, color = gray] (7.8, 1.9) -- node [sloped, scale = 0.7, above, very near end] {$(2X)^{i-1}$} (9.2, 0.2);
\draw[->, color = gray] (7.8,0.9) -- node [sloped, scale = 0.7, above] {$X^i$} (9.2,0.1);
\draw[->, color = gray] (7.8, 2) -- node [above, scale = 0.7] {$2X$} (9.2,2);
\draw[-, line width=6pt, color = white] (7.8, 1) -- (9.2,1);
\draw[->, color = gray] (7.8,1) -- node[above, scale = 0.7] {$2X$} (9.2,1);
\draw[->, color = gray] (10.9, 1.9) -- node [sloped, scale = 0.7, above] {$(2X)^i$} (12.1, 0.1);
\draw[->] (10.9, 0) -- node [above, scale = 0.7] {$0$} (12.1,0);
\draw[dashed] (-0.2, -0.5) -- (14, -0.5);
\draw[dashed] (-0.2, 2.5) -- (14, 2.5);
\end{tikzpicture}

\begin{tikzpicture}
\node at (0,1) {$\cdots$};
\node[color = gray] at (1,2) {$q^{6k-4}A$};
\node[color = gray] at (1,1) {$q^{6k-6}A$};
\node[color = gray] at (4,2) {$q^{6k-2}A$};
\node[color = gray] at (4,1) {$q^{6k-4}A$};
\node at (4,0) {$q^{6k-4}A$};
\node at (7,0) {$q^{6k-2}A$};
\node at (10,0) {$u^{4k}q^{6k}A$};
\draw[->, color = gray] (1.8, 1.9) -- node [above, sloped, scale = 0.7, near end] {$\id$} (3.2, 0.2);
\draw[->, color = gray] (1.8, 0.9) -- node [above, sloped, scale = 0.7] {$X$} (3.2, 0.1);
\draw[->, color = gray] (1.8, 2) -- node [above, scale = 0.7] {$2X$} (3.2, 2);
\draw[-, line width = 6pt, color = white] (1.8,1) -- (3.2, 1);
\draw[->, color = gray] (1.8, 1) -- node [above, scale = 0.7] {$2X$} (3.2, 1);
\draw[->, color = gray] (4.8, 1.9) -- node [above, sloped, scale = 0.7] {$-\id$} (6.2, 0.2);
\draw[->, color = gray] (4.8, 0.9) -- node [above, sloped, scale = 0.7] {$-X$} (6.2, 0.1);
\draw[->] (4.8, 0) -- node [above, scale = 0.7] {$2X$} (6.2, 0);
\draw[->] (7.8, 0) -- node [above, scale = 0.7] {$0$} (9.2,0);
\end{tikzpicture}
\end{center}
\caption{\label{fig:bet_complex}The cochain complex $\mathcal{C}_k$. Notice that $(2X)^i$ is $0$ for $i\geq 2$, and $1$ for $i=0$. Homological degrees are indicated by $u^i$ for some modules.}
\end{figure}

To prove Theorem \ref{thm:case_omega_4} we only need to analyse the various sub-complexes of $\mathcal{C}_k$ depending on $l$. First observe that $\mathcal{C}_k$ contains identity morphisms between homological degrees $4k-2$ and $4k-1$, between $4k-3$ and $4k-2$, and between $4k-5$ and $4k-4$ (coming from the $(2X)^{i-1}$ with $i=1$). To have these identities in the sub-complexes, we additionally need $l\geq 2$ for the first, $l\geq 3$ for the second, and $l\geq 5$ for the third identity. If we can cancel, the morphisms labelled $X$ will disappear, and if we cannot cancel, the morphisms labelled $X$ are not present in the sub-complex.

Furthermore, the diagonal morphisms labelled $(2X)^i$ in the middle part of Figure \ref{fig:bet_complex} are only non-zero if $i=1$, and even if $i=1$ we can remove them with a change of basis. This implies that up to isomorphism the complex $\mathcal{C}_k$ and its relevant sub-complexes are isomorphic to a direct sum of complexes of the form $u^iq^jA\stackrel{2X}{\longrightarrow}u^{i+1}q^{j+2}A$ and $q^jA$. We refer to the former complex as a {\em knight}, and note its homology is given by $u^iq^{j-1}\Z\oplus u^{i+1}q^{j+1}\Z/2\Z\oplus u^{i+1}q^{j+3}\Z$. The latter complex is called a {\em pawn}.

The gray knight complexes resemble the Khovanov complexes of the split union of $T(2,-l)$ with an unknot, while the black knight and pawn complexes give the Khovanov homology of $T(3,3k)$, except in homological degrees near $4k$. Therefore we get in homological degrees $i\leq 4k-6$ the direct sums of Khovanov homologies for $T(3,3k)$ and the split union of $T(2,l)$ with an unknot up to grading shifts.

For the homological degrees between $4k-5$ and $4k$ we need to check various cases, since there is some interaction between the gray and black parts. In particular, the value of $l$ determines which cancellations we can do. In addition to distinguishing between $k=1$ and $k\geq 2$, we need to consider $l \in \{1,2,3,4\}$ and $l\geq 5$. While there are several cases to check, they are straightforward and left to the reader.

\subsection{Proof of Theorem \ref{thm:case_omega_6}}
Let $w = \sigma_1^{-p_1}\sigma_2^{q_1}\cdots \sigma_1^{-p_r}\sigma_2^{q_r}$ with positive integers $r, p_1,q_1,\ldots, p_r, q_r$. By \cite[Thm.9.2]{schuetz2025kh3braids} we have that
\[
\CKh(L_{\Delta^{2k}w}) \simeq B_k\oplus D_w,
\]
where $B_k$ is a direct sum of knight and pawn complexes concentrated in homological degress less than $4k-4$, whose homology satisfies
\[
H^{i,j-t}(B_k) \cong \HKh^{i,j}(T(3,3k))
\]
for $i\leq 4k-5$ and $t=q_1+\cdots +q_r - (p_1 + \cdots + p_r)$, and $D_w$ fits into a short exact sequence of $A$-cochain complexes
\[
0\longrightarrow T \longrightarrow D_w \longrightarrow u^{4k}q^{12k}C_w \longrightarrow 0,
\]
where $C_w\oplus q^t A\simeq \CKh(L_w)$ as $A$-cochain complexes, and
\[
T \cong \left\{ \begin{array}{cc}
u^{4k-4}q^{12k+t-8}A\oplus u^{4k-3}q^{12k+t-4}A\oplus u^{4k-2}q^{12k+t-4}K \oplus u^{4k}q^{12k+t}K & k\geq 2 \\
q^{t+4}A\oplus u^2q^{t+8}K\oplus u^4q^{t+12}K & k=1
\end{array}
\right.
\]
Here $K = (A \stackrel{2X}{\longrightarrow}uq^2A)$ is a knight complex. From this it follows directly that
\begin{equation}\label{eq:nice_sum}
\HKh^{i,j}(L_{\Delta^{2k}w}) \cong \HKh^{i,j-t}(T(3,3k))\oplus \HKh^{i-4k,j-12k}(L_w)
\end{equation}
for $i\leq 4k-5$ and $i\geq 4k+2$. It remains to check that this also holds for $i\leq 4k-1$, and for $i\in \{4k,4k+1\}$ we get the formulas as in the statement of Theorem \ref{thm:case_omega_6}.

Let us write $s=12k+t$, and consider the long exact sequence
\begin{equation}\label{eq:les_omega6}
\cdots \longrightarrow \tilde{H}^{4k+i-1,s+2i-2}(D_w)\stackrel{\delta}{\longrightarrow} \tilde{H}^{4k+i,s+2i}(D_w)\longrightarrow H^{4k+i,s+2i-1}(D_w) \longrightarrow \tilde{H}^{4k+i,s+2i-2}(D_w)
\end{equation}
Here, $\tilde{H}$ refers to the homology of the complex $D_w\otimes_A\Z$, where $X$ acts as $0$ on $\Z$. For $i\geq -5$ we have
\begin{align*}
\tilde{H}^{4k+i,s+2i}(D_w) &\cong \widetilde{\HKh}{}^{4k+i,s+2i}(L_{\Delta^{2k}w})\cong \tilde{H}^{i,t+2i}(C_w)\oplus \tilde{H}^{4k+i,s+2i}(T)\\
\tilde{H}^{4k+i,s+2i-2}(D_w) &\cong \widetilde{\HKh}{}^{4k+i,s+2i-2}(L_{\Delta^{2k}w})\cong 0,
\end{align*}
with $\tilde{H}^{i,j}(C_w) = \widetilde{\HKh}{}^{i,j}(L_w)$ for $(i,j)\not=(0,t)$, and $H^{0,t}(C_w)\oplus \Z\cong \widetilde{\HKh}{}^{0,t}(L_w)$ by \cite[Cor.9.3]{schuetz2025kh3braids}.
Furthermore, $\delta$ is the connecting homomorphism of the reduced/unreduced long exact sequence in Khovanov homology of $L_{\Delta^{2k}w}$, which by \cite[Lm.5.6]{MR4873797} is twice the differential of the first page of the reduced integral BLT spectral sequence.

Notice that $\tilde{H}^{4k+i,s+2i}(T)\cong \Z$ for $i=-4, -2,-1,0,1$. We also have $\tilde{H}^{4k-3,s-6}(T) = 0$, and if $k\geq 2$ we have
\[
\tilde{H}^{4k-3,s-4}(T) \cong \Z \cong \widetilde{\HKh}{}^{4k-3,s-4}(L_{\Delta^{2k}w}).
\]
In the reduced integral BLT spectral sequence the differential on the first page induces isomorphisms $\tilde{H}^{4k+i,s+2i}(T)\cong \tilde{H}^{4k+i+1,s+2i+2}(T)$ for $i=-2, 0$, while $\Z\cong \tilde{H}^{4k-4, s-8}(T)$ survives to the $E_2$-page. The long exact sequence (\ref{eq:les_omega6}) with $i=-4$ then turns into
\[
\widetilde{\HKh}{}^{-5,t-10}(L_w)\stackrel{\delta}{\longrightarrow}\widetilde{\HKh}{}^{-4,t-8}(L_w)\oplus \Z \longrightarrow \HKh^{4k-4,s-9}(L_{\Delta^{2k}w})\longrightarrow 0
\]
and the summand of $\Z$ injects into $\HKh^{4k-4,s-9}(L_{\Delta^{2k}w})$. In particular,
\[
\HKh^{4k-4,s-9}(L_{\Delta^{2k}w})\cong \HKh^{-4,t-9}(L_w)\oplus \Z \cong \HKh^{-4,t-9}(L_w)\oplus \HKh^{4k-4, 12k-9}(T(3,3k)).
\]
The exact same argument works for $i=-2$, thus establishing (\ref{eq:nice_sum}) for $(4k-2, s-5)$, and for $i=0$ we get
\[
\HKh^{4k,s-1}(L_{\Delta^{2k}w}) \cong H^{0,t-1}(C_w)\oplus \Z,
\]
but as was noted above, $H^{0,t-1}(C_w)\oplus \Z\cong \HKh^{0,t-1}(L_w)$.

For $i=-3$ the long exact sequence (\ref{eq:les_omega6}) is
\[
\widetilde{\HKh}{}^{-4,t-8}(L_w)\oplus \Z\stackrel{\delta}{\longrightarrow}\widetilde{\HKh}{}^{-3,t-6}(L_w)\longrightarrow \HKh^{4k-3,s-7}(L_{\Delta^{2k}w})\longrightarrow 0,
\]
with the $\Z$-summand in the kernel of $\delta$. In particular, (\ref{eq:nice_sum}) holds for $(i,j) = (4k-3, s-7)$ as the torus link does not contribute to this bigrading.

For $i=-1$ the sequence (\ref{eq:les_omega6}) is
\[
\begin{tikzpicture}
\node at (0,1.5) {$\widetilde{\HKh}{}^{4k-2,s-4}(L_{\Delta^{2k}w})$};
\node at (0,0) {$\widetilde{\HKh}{}^{-2,t-4}(L_w)\oplus \Z$};
\node at (4.5,1.5) {$\widetilde{\HKh}{}^{4k-1,s-2}(L_{\Delta^{2k}w})$};
\node at (4.5,0) {$\widetilde{\HKh}{}^{-1,t-2}(L_w)\oplus \Z$};
\node at (9,1.5) {$\HKh^{4k-1,s-3}(L_{\Delta^{2k}w})$};
\node at (9,0) {$\HKh^{4k-1,s-3}(L_{\Delta^{2k}w})$};
\node at (12,1.5) {$0$};
\node at (12,0) {$0$};
\draw[->] (1.7,1.5) -- node [above, scale = 0.7] {$\delta$} (2.9,1.5);
\draw[->] (1.7,0) -- node [above, scale = 0.7] {$\delta$} (2.9,0);
\draw[->] (6.2, 1.5) -- (7.4, 1.5);
\draw[->] (6.2,0) -- (7.4,0);
\draw[->] (10.7, 1.5) -- (11.8, 1.5);
\draw[->] (10.7, 0) -- (11.8, 0);
\draw[->] (0,1.2) -- node [right, scale = 0.7] {$\cong$} (0, 0.3);
\draw[->] (4.5,1.2) -- node [right, scale = 0.7] {$\cong$} (4.5, 0.3);
\draw[->] (9,1.2) -- node [right, scale = 0.7] {$\cong$} (9, 0.3);
\end{tikzpicture}
\]
Recall that $\delta$ is twice the differential of the first page of the reduced integral BLT-spectral sequence, which cancels the two $\Z$-summands. It follows that
\[
\HKh^{4k-1,s-3}(L_{\Delta^{2k}w}) \cong \HKh^{-1,t-3}(L_w)\oplus \Z/2\Z,
\]
with $\Z/2\Z \cong \HKh^{4k-1, 12k-3}(T(3,3k))$.

The case $i=1$ is very similar. The main difference is that $\widetilde{\HKh}{}^{4k,s}(L_{\Delta^{2k}w})\cong \tilde{H}^{0,t}(C_w)\oplus \Z$. Recall that $\tilde{H}^{0,t}(C_w)\oplus \Z\cong \widetilde{\HKh}{}^{0,t}(L_w)$, with this summand of $\Z$ in the kernel of $\delta$. We thus also get
\[
\HKh^{4k+1,s+1}(L_{\Delta^{2k}w}) \cong \HKh^{1,t+1}(L_w)\oplus \Z/2\Z,
\]
which is one of the exceptional cases in Theorem \ref{thm:case_omega_6}. It remains to prove the theorem in bigradings $(4k+i,s+2i+1)$ for $i=-4,\ldots, 1$ and $(4k-3, s-3)$.

In bidegree $(4k-3,s-3)$ only the torus link contributes, and it is $\Z$ if $k\geq 2$ and $0$ if $k=1$. But this is exactly the contribution of $T$ in this bidegree.

Now consider the long exact sequence
\begin{equation}\label{eq:les_plus}
\cdots \longrightarrow \tilde{H}^{4k+i,s+2i+2}(D_w) \longrightarrow H^{4k+i,s+2i+1}(D_w) \longrightarrow \tilde{H}^{4k+i,s+2i}(D_w) \stackrel{\delta}{\longrightarrow} \tilde{H}^{4k+i+1,s+2i+2}(D_w)
\end{equation}
The group $\tilde{H}^{4k+i,s+2i+2}(D_w) = 0$ unless $i=-3$ and $k\geq 2$, in which case it is $\Z$ that injects into $H^{4k-3, s-5}(D_w)$. Then
\[
\tilde{H}^{4k-3,s-6}(D_w) \cong \widetilde{\HKh}{}^{4k-3, s-6}(L_{\Delta^{2k}w}) \cong \widetilde{\HKh}{}^{-3,t-6}(L_w).
\]
The kernel of $\delta$ is the kernel of the first differential in the reduced integral BLT-spectral sequence. Hence
\[
\HKh^{4k-3, s-5}(L_{\Delta^{2k}w}) \cong \Z \oplus \HKh^{-3, t-5}(L_w),
\]
establishing (\ref{eq:nice_sum}) for $(4k-3, s-5)$. Notice that we always have the long exact sequence
\[
0\longrightarrow \HKh^{i,t+2i+1}(L_w) \longrightarrow \widetilde{\HKh}{}^{i,t+2i}(L_w)\stackrel{\delta}{\longrightarrow} \widetilde{\HKh}{}^{i+1,t+2i+2}(L_w)\longrightarrow \cdots
\]
For the other relevant values of $i$ we have
\[
\HKh^{4k+i,s+2i+1}(L_{\Delta^{2k}w}) \cong \ker (\delta\colon \tilde{H}^{4k+i,s+2i}(D_w)\to \tilde{H}^{4k+i+1,s+2i+2}(D_w)).
\]
For $i=-4$ we have $\tilde{H}^{4k-4,s-8}(D_w)\cong \widetilde{\HKh}{}^{-4,t-8}(L_w)\oplus \Z$, with $\Z\cong \widetilde{\HKh}{}^{4k-4,12k-8}(T(3,3k))$ in $\ker \delta$. Hence
\[
\HKh^{4k-4,s-7}(L_{\Delta^{2k}w}) \cong \HKh^{-4,t-7}(L_w)\oplus \Z \cong \HKh^{-4,t-7}(L_w)\oplus \HKh^{4k-4,12k-8}(T(3,3k)).
\]
The same argument applies to $i=-1$, establishing (\ref{eq:nice_sum}) for $(4k-1, s-1)$.

For $i=-2$ we have 
\[
\tilde{H}^{4k-2,s-4}(D_w)\cong \widetilde{\HKh}{}^{-2,t-4}(L_w)\oplus \Z.
\]
This time $\Z\cong \widetilde{\HKh}{}^{4k-2,12k-4}(T(3,3k))$ is not in $\ker\delta$. But since $\HKh^{4k-2,12k-3}(T(3,3k))=0$ we still get (\ref{eq:nice_sum}) in bidegree $(4k-2,s-3)$.

For $i=0$ we have
\[
\tilde{H}^{4k,s}(D_w)\cong \tilde{H}^{0,t}(C_w)\oplus \Z \cong \widetilde{\HKh}{}^{0,t}(L_w),
\]
but $\delta$ restricted to the summand $\Z$ is not $0$ (recall that $\widetilde{\HKh}{}^{0,t}(L_w)$ contains a summand $\Z$ which survives the reduced BLT-spectral sequence, but we have a different behaviour here). It follows that
\[
\HKh^{4k,s+1}(L_{\Delta^{2k}w}) \cong \HKh^{0,t+1}(L_w)/\Z,
\]
as stated in Theorem \ref{thm:case_omega_6}.

For the remaining case $i=1$ observe that $\tilde{H}^{4k+1,s+2}(D_w)$ has an extra summand $\Z$ coming from the homology of $T$, which produces an extra generator in $\ker\delta$. Therefore
\[
\HKh^{4k+1,s+3}(L_{\Delta^{2k}w})\cong \HKh^{1,t+3}(L_w)\oplus \Z,
\]
as claimed.

\section{Extremal Khovanov homology of braids in polynomial time}
In this section we prove  Theorem \ref{thm:Extremal KH for braids} by constructing algorithms $\mathcal A_k$ which work on link diagrams and can be specialized into braid algorithms $\mathcal B_{k,t}$. We start in (\ref{subsec: overview}) by giving an overview of algorithms $\mathcal A_k$ constructing $\mathcal B_{k,t}$ using them. In (\ref{subsec: blocks}) the algorithms $ \mathcal A_k$ are described in more detail and in (\ref{subsec: validity}) and (\ref{subsec: polytime}) their validity and polynomial runtime are proven. 

\subsection{Overview of the algorithm \texorpdfstring{$\mathcal A_{k}$}{Ak}}\label{subsec: overview}

\term{Planar arc diagrams} glue ``smaller" tangles together to generate ``larger" tangles and eventually links. Bar-Natan's theory allows one to formulate local Khovanov complexes of ``smaller" tangles, which themselves can be composed as well with planar arc diagrams and the two ways of making a ``larger" composite complex commute with each other \cite{MR2174270}. Instead of using the general theory,  it suffices for our purposes to introduce unoriented \term{simple planar arc diagrams of Type I, II, III,} and \term{IV}, see Figure \ref{Figure: Planar diagrams}.

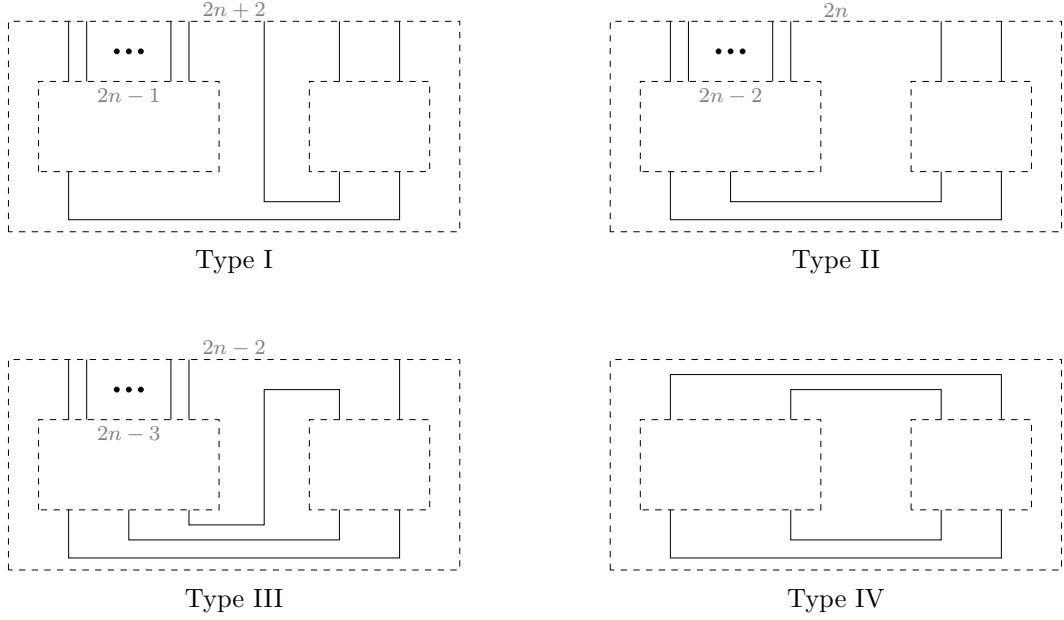
\begin{figure}[ht]
    \centering
    \input{simplePlanars/alltypes}
    \caption{The four different families of simple planar arc diagrams which we will use. Type I, II and III simple planar arc diagrams are determined by $n\geq 1$ and Type IV is unique. The planar arc diagrams are considered up to boundary preserving planar isotopy.}
    \label{Figure: Planar diagrams}
\end{figure}
\term{A nice scanning sequence} $\mathcal S$ of a link diagram $L$ with $n$ crossings is a triple of sequences 
$$\mathcal{S}=((T_1,\dots T_n),(c_1,\dots, c_n), (P_1,\dots, P_{n-1}))$$ 
where $T_i$ are unoriented subtangle diagrams of $L$, $c_i$ are unoriented crossings of $L$ and $P_i$ are simple planar arc diagrams of Type I, II or III for $i\leq n-2$ and $P_{n-1}$ is of type IV. Additionally, we require that $T_1=c_1$, $T_n=L$ and $T_i\otimes_{P_i} c_{i+1}=T_{i+1}$ as tangle/link diagrams.
A scanning sequence is also implicitly present in the scanning algorithm of \cite{MR2320156}, but our technical condition of allowing only one Type IV diagram is the extra ingredient which makes our sequences \term{nice}. In Lemma \ref{lemma: every link has nice scan} we prove that for reasonable link diagrams nice scanning sequences always exist. The girth of a nice scanning sequence, denoted by $\operatorname{girth}(\mathcal S)$, is the maximum number of boundary points on all subtangles $T_i$.

\begin{proposition}\label{prop: Extremal KH for bounded girth}
For every $k \geq 0$ there is an algorithm $\mathcal A_{k}$ which takes in a link diagram $L$ and a nice scanning sequence $\mathcal S$ for $L$ and outputs $\operatorname{Kh}^{i, \ast}(L)$  for $i< -n_-(L)+k$  and $i> n_+(L)-k$. The algorithm $\mathcal A_k$ has running time $\mathcal O(f(\operatorname{girth}(\mathcal S)) \cdot p_k(n(L)))$ for some function $f$ and polynomial $p_k$.
\end{proposition}

Theorem \ref{thm:Extremal KH for braids} is obtained as a corollary Proposition \ref{prop: Extremal KH for bounded girth}. 

\begin{proof}[Proof of Theorem \ref{thm:Extremal KH for braids}]
    The algorithm $\mathcal B_{k,t}$ takes in a word $w$ applies the algorithm $\mathcal A_{k+t-1} $ to the pair $(L_b, \mathcal S)$. The link diagram $L_b$ is the braid closure of $b=\sigma_1\dots \sigma_{t-1} w \sigma_{t-1}^{-1} \dots \sigma_1^{-1}$ and $\mathcal S$ scans the crossings of $L_b$ in the order that they occur in $b$. Conjugating $w$ with $\sigma_1 \dots \sigma_{t-1}$ before taking the braid closure does not alter the isotopy type of the link but it will ensure that $\mathcal{S}$ is a nice scanning sequence with $\operatorname{girth}(\mathcal S)=2t$. The conjugation also adds $t-1$ positive and negative crossings, which is why we need to apply $\mathcal A_{k+t-1}$ to $L_b$ in order to obtain $\operatorname{Kh}^{i,\ast}(L_w)=\operatorname{Kh}^{i,\ast}(L_b)$ in homological degrees $i\leq -n_-(L_w)+k=-n_-(L_b)+k+t-1$  and $i\geq  n_+(L_w)-k=-n_+(L_b)-k-t+1$.    
\end{proof}

To prove Proposition \ref{prop: Extremal KH for bounded girth} it suffices to describe an algorithm which obtains $\operatorname{Kh}^{i, \ast}(L)$  for $i< -n_-(L)+k$. The degrees $i> n_+(L)-k$ can be obtained by applying the same algorithm to the mirrored diagram $D^!$ and using the fact that Khovanov homology of the mirror link is obtained from the dual complex.
In order to minimize the trouble of grading shifts, we will make use of unshifted complexes and delay the shifts to the very end. Thus, for a crossing $c\in \{ \crosPos, \crosNeg \}$ we denote $\llbracket c \rrbracket_{\operatorname{US}}$ as the \term{unshifted Bar-Natan complex of $c$}, so that 
\begin{align*}
\llbracket \crosPos \rrbracket_{\operatorname{US}}&= \quad \cdots \longrightarrow 0 \longrightarrow u^0q^0 \,  \smoothingZero \ \xrightarrow{\text{saddle}} \  u^1q^1 \, \smoothingOne \ \longrightarrow 0 \longrightarrow \cdots \\
\llbracket \crosNeg \rrbracket_{\operatorname{US}}&= \quad  \cdots \longrightarrow 0 \longrightarrow u^0q^0 \,  \smoothingOne \ \xrightarrow{\text{saddle}}  \  u^1q^1 \, \smoothingZero \ \longrightarrow 0 \longrightarrow \cdots.
\end{align*}
and more generally for a tangle diagram $T$ we write $\llbracket T \rrbracket_{\operatorname{US}}= u^{n_-(L)} q^{-n_+(L)+2n_-(L)}\llbracket T \rrbracket$.

The algorithms $\mathcal{A}_{k}$ are roughly described by pseudocode in Algorithm \ref{Algorithm: extremal lines over Z} and in Subsection \ref{subsec: blocks}, we will express the precise meaning of Lines \ref{alg_line: deloop and keep block decomp} and \ref{alg_line: gaussian eliminate blocks} whose point that we do not eliminate ``too many" isomorphisms. Limiting the number of cancellations allows us to keep the lengths of integers in check, which ensures that we do not spend exponential time multiplying integers of superpolynomial bit-length.

\begin{algorithm}
\caption{For any $k$, the following describes $\mathcal A_{k}$ for the lowest degrees and with $\mathbb Z$-coefficients.} \label{Algorithm: extremal lines over Z} 
 \textbf{Input} Link diagram $L$ with a nice scanning sequence $((T_1,\dots T_{n(L)}),(c_1,\dots, c_{n(L)}), (P_1,\dots, P_{n(L)-1}))$.   \\
  \textbf{Output}   $\operatorname{Kh}^{i,\ast} (L) \text{ for } i\leq k-n_-(L) $
   \begin{algorithmic}[1]
\State Assign $C \gets \llbracket c_1 \rrbracket_{\operatorname{US}} $ \label{alg_line: initialize}
\For{$t=1,\dots, n(L)-1$}
    \State Assign $C' \gets C \otimes_{P_t}\llbracket c_{t+1} \rrbracket_{\operatorname{US}} $ \label{alg_line: compose}
    \State Deloop the newly formed circles in $C'$ and fix a block decomposition. Call the result $C''$. \label{alg_line: deloop and keep block decomp}
    \State Gaussian eliminate certain blocks of $C''$ by using the saddle of $\llbracket c_{j+1}  \rrbracket_{\operatorname{US}}$. Call the result  $C'''$. \label{alg_line: gaussian eliminate blocks}  
    \State Truncate $C'''$ at $k+1$, that is, set $(C''')^{k+2}=0$. Call the result $C$. \label{alg_line: truncation for Z}
\EndFor
\State On $C$ replace every $\emptyset$ with $\mathbb Z$. (Or to put it fancily, apply the Khovanov TQFT functor $F$.) \label{alg_line: TQFT}
\State Put the differentials of $F(C)$ into Smith normal form. \label{alg_line: homology}
\State Read of the homology from the Smith normal form and shift it with $u^{-n_-(L)} q^{n_+(L)-2n_-(L)}$.
\end{algorithmic}
\end{algorithm}
Over finite fields, one does not need to restrain the number of cancellations and pseudocode in Algorithm \ref{Algorithm: extremal Kh over Fp} presents simpler algorithms $\mathbb F_p \mathcal A_{k}$ which compute the homology with $\mathbb F_p$ coefficients. Removing  Line \ref{alg_line: truncation for Fp} from the Algorithms $\mathbb F_p \mathcal A_{k}$  gives precisely Bar-Natan's scanning algorithm \cite{MR2320156} which computes the whole Khovanov homology table $\operatorname{Kh}^{\ast,\ast}(L; \mathbb F_p)$. 

\begin{algorithm}
\caption{For any $k$ the following describes $\mathbb F_p \mathcal A_{k}$ for the lowest degrees and with $\mathbb F_p$ coefficients.} 
\label{Algorithm: extremal Kh over Fp}
  \textbf{Input} Link diagram $L$ with a nice scanning sequence $((T_1,\dots T_{n(L)}),(c_1,\dots, c_{n(L)}), (P_1,\dots, P_{n(L)-1}))$.   \\
  \textbf{Output}   $\operatorname{Kh}^{i,\ast} (L; \mathbb F_p) \text{ for } i\leq k-n_-(L) $
   \begin{algorithmic}[1]
\State Assign $C \gets \llbracket c_1 \rrbracket_{\operatorname{US}} $ 
\For{$t=1,\dots, n(L)-1$}
    \State Assign $C' \gets C \otimes_{P_t}\llbracket c_{t+1} \rrbracket_{\operatorname{US}} $
    \State Deloop the newly formed circles in $C'$. Call the result $C''$
    \State Iteratively Gaussian eliminate isomorphisms of $C''$ until none are left. Call the result $C'''$. 
    \State Truncate $C'''$ at $k+1$, that is, set $(C''')^{k+2}=0$. Call the result $C$. \label{alg_line: truncation for Fp}  
\EndFor
\State Replace every $\emptyset$ with $\mathbb Z$. (Or to put it fancily, apply the Khovanov TQFT functor $F$.) 
\State Read of the homology from $C$ and shift it with $u^{-n_-(L)} q^{n_+(L)-2n_-(L)}$.
\end{algorithmic}
\end{algorithm}

\subsection{Expanding and eliminating blocks in \texorpdfstring{$\mathcal A_{k}$}{Ak}}\label{subsec: blocks}

Every object $B$ in $\Cobl^{\mathbb Z}(2n)$ is isomorphic to a unique direct sum (up to a permutation) 
\begin{equation} \label{eq:decomposition of objects}
B\cong\bigoplus_{i=1}^n q^{c_i}D_i    
\end{equation}
where $D_{i}$ are loopless diagrams, which are planar matchings on $2n$ points, and $c_i\in \mathbb Z$. Using (\ref{eq:decomposition of objects}) any morphism $f\colon A\to B$ of $\Cobl^{\mathbb Z}(2n)$ can be described with a matrix, where each matrix element is a morphism $f_{i,j}\in \operatorname{Hom}(A_j,B_i)$ between some loopless diagrams $A_j$ and $B_i$.

A morphism $g\colon D \to E$ between loopless diagrams is a formal sum of oriented surfaces embedded in $[0,1]^3$ with boundary $X=(D\times \{1\}) \cup (E \times \{0\}) \cup (\partial D \times [0,1])$. Topologically, $X$ is a disjoint union of $\mathbb S^1$:s and we can glue disks to each loop and embed them inside the cube $[0,1]^3$. Up to a boundary preserving diffeomorphism, this generates a unique cobordism $\Sigma_{D,E} \hookrightarrow [0,1]^3$. We define $\operatorname{dec}(D,E)$ to be the set of all possible ways of decorating each component of $\Sigma_{D,E}$ with $0$ or $1$ dots. In other words, $\operatorname{dec}(D,E)$ consists of $2^c$ copies of $\Sigma_{D,E}$ with varying decorations where $c$ is the number of $\mathbb S^1$:s in $X$. 

\begin{lemma}\label{Lemma: basis for Hom:s}
    For loopless planar matchings $D,E \in \Cobl^{\mathbb Z}(2n) $, the set $\operatorname{dec}(D,E)$ freely generates $\operatorname{Hom}(D,E)$.
\end{lemma}
\begin{proof}
    To see that $\operatorname{dec}(D,E)$ generates $\operatorname{Hom}(D,E)$, take any dotted surface $S$ embedded $[0,1]^3$ with the desired boundary. By using the neck-cutting relation, one can get rid of all of the genus and disconnect every boundary component of $S$. Then $S$ is expressed as a sum of unions of dotted spheres and disks. Using the double dot and sphere relations to this sum, we get a description of $S$ as a sum over $\operatorname{dec}(D,E)$.

    Next, let $\sum_{S\in \operatorname{dec}(D,E)} c_S S =\sum_{S\in \operatorname{dec}(D,E)} d_S S $ for some integers $c_S$ and $d_S$. To see that $c_S=d_S$ for a fixed $S$, we describe a map $f_S\colon \operatorname{Hom}(D,E) \to \operatorname{Hom}(\emptyset, \emptyset)$.  The map $f_S$ takes a dotted surface $\Sigma$ in $[0,1]^3$ whose boundary is a disjoint union of $\mathbb S^1$:s in $\partial ([0,1]^3)$ and glues disks to each of those boundary loops. The glued discs are embedded ``outside of the box" and $f_S$ places a dot on each ``outside disk" if the corresponding ``inside disk" has no dot in the surface $S$. Rescaling the generated surface appropriately, we get a surface $f_S(\Sigma) \in \operatorname{Hom}(\emptyset, \emptyset)$. Moreover, this procedure describes a well-defined linear map $f_S\colon \operatorname{Hom}(D,E) \to \operatorname{Hom}(\emptyset, \emptyset)$. By plugging in the original sums and using the sphere and double-dot relations, we obtain
    $$
    c_S E=f_S \left( \sum_{S'\in \operatorname{dec}(D,E)} c_{S'} S' \right)= f_S\left( \sum_{S'\in \operatorname{dec}(D,E)} d_{S'}S' \right) =d_S E
    $$
    where $E$ is the empty cobordism $E\colon \emptyset \to \emptyset$. Since $ \operatorname{Hom}(\emptyset, \emptyset)$ is freely generated by $E$, this concludes the proof.
\end{proof}

In complexity theory, finite fields are often simple to work with since addition and multiplication of two elements can be done in constant time in terms of bit-operations on a computer. Storing an integer $c$ takes $\log_2 |c|$ amount of memory and time which can be too much if $c$ is humongous. Asymptotically optimal algorithms for multiplying integers have been extensively studied in computer science, but for us it will suffice that both addition $c+d$ and multiplication $cd$ can be computed in polynomial time with respect to $\log_2 (|cd|)$. Nevertheless, in order to prove Proposition \ref{prop: Extremal KH for bounded girth} we will need to make sure that the $\log_2(|c|)$ is bounded by a polynomial for all of the integers involved. To this end, we define $\lVert \cdot \rVert$ to be the $l_1$ norm on $\operatorname{Hom}(D,E)$ with respect to the basis $\operatorname{dec}(D,E)$, that is,
$$
\Bigg\lVert \sum_{S\in \operatorname{dec}(D,E)}c_S S \Bigg\rVert = \sum_S |c_S|.
$$
\begin{lemma}\label{lemma: horizontal composition}
    There is an algorithm which takes in loopless planar matchings $A,B,C \in \Cobl^{\mathbb Z}(2n)$ and morphisms $f_1\colon A \to B$, $f_2\colon B \to C$ expressed in the basis as
    $$
    f_1=\sum_{S\in \operatorname{dec}(A,B)}c_S S, \qquad f_2=\sum_{S'\in \operatorname{dec}(B,C)}d_{S'} S'
    $$
    and outputs
    $$
    f_2f_1=\sum_{S''\in \operatorname{dec}(A,C)}e_{S''} S''.
    $$
    For a fixed $n$, the algorithm runs in polynomial time with respect to $\log( \lVert f_2 \rVert \cdot \lVert f_1 \rVert)$ and the composition satisfies $\lVert f_2f_1 \rVert \leq 2^{n-1}  \lVert f_2 \rVert \cdot  \lVert f_1 \lVert $. 
\end{lemma}
\begin{proof}
Our first goal is to isolate the decorations and coefficients of $f_1$ and write
\begin{equation} \label{eq: horicomp f_1}
f_1=\sum_{S\in \operatorname{dec}(A,B)}c_S S=\Sigma_{A,B} \circ \sum_{P\in \operatorname{dec}(A,A)} t_P P
\end{equation}
where the underlying surface $\Sigma_{A,A}$ of $\operatorname{dec}(A,A)$ is the product cobordism $A\times [0,1]$. On each dot decoration $S\in \operatorname{dec}(A,B)$ one has to choose to which components of $A \times [0,1]$ the dots are moved, that is, one has to pick an appropriate decoration $P\in \operatorname{dec}(A,A)$. For these choices $t_P=c_S$ and the rest of the coefficients $t_P$ vanish. This quick procedure (non-uniquely) finds (\ref{eq: horicomp f_1}) and similarly expanding $f_2$ we can obtain
\begin{equation}\label{eq:horicomp f_2f_1}
    f_2f_1= \sum_{Q\in \operatorname{dec}(C,C)} r_Q Q \circ \Sigma_{B,C} \circ \Sigma_{A,B} \circ \sum_{P\in \operatorname{dec}(A,A)} t_P P.
\end{equation}

Next, we will analyze the non-decorated surface $\Sigma_{B,C}\Sigma_{A,B} $. The connected components $\{ \Sigma_i \}_i$ of  $\Sigma_{B,C}\Sigma_{A,B} $ can be obtained directly from the planar matchings $A,B, $ and $C$. Each connected $\Sigma_i$ is obtained by gluing $k$ discs of $\Sigma_{A,B}$ and $l$ discs of $\Sigma_{B,C}$ along $m$ lines yielding $\chi(\Sigma_i)=k+l-m$. By the classification of oriented surfaces with boundary  we also get $\chi(\Sigma_i)=2-r-2g$ where $r$ is the number of boundary loops of $\Sigma_i$ and $g$ is its genus. Since $k,l,m$ and $r$ are obtainable from $A,B$ and $C$ we can compute $g$. By using the neck cutting relation, we can get rid of the genus on $\Sigma_i$ at the cost of adding $g$ dots and a coefficient $2^g$. Further using the neck-cutting relation $r-1$ times we can replace all connections of $\Sigma_i$ with a sum of decorated discs which are glued on the $r$ loops. Applying this simplification to every connected component $\Sigma_i$ gives the expansion
\begin{equation} \label{eq: horicomp Sigmas}
\Sigma_{B,C} \Sigma_{A,B} = \sum_{S''\in \operatorname{dec}(A,C)} b_{S''} S''.     
\end{equation}
Finally, the expansion of $f_2f_1$ can be computed from (\ref{eq:horicomp f_2f_1}) by adding and multiplying the integers $t_P$, $r_Q$, $b_{S''}$ and using the double dot relation.

Since $\Sigma_{A,B}$ and $\Sigma_{B,C}$ contain at least $1$ disk each and $\Sigma_{B,C} \Sigma_{A,B}$ is obtained by gluing them along $n$ lines, we get $\chi(\Sigma_{B,C} \Sigma_{A,B}) \geq 2-n$. On the other hand $\Sigma_{A,C}$ contains at most $n$ disks, so $\chi(\Sigma_{A,C} )\leq n$. Thus  $\chi(\Sigma_{A,C})-\chi(\Sigma_{B,C} \Sigma_{A,B}) \leq 2n-2$ which means that the neck cutting relation can be used at most $n-1$ times when arriving at (\ref{eq: horicomp Sigmas}). Hence $\lVert \Sigma_{B,C}\Sigma_{A,B} \rVert \leq 2^{n-1}$  and we can compute 
\begin{align*}
    \lVert f_2 f_1 \rVert &= \left\lVert \sum_{Q\in \operatorname{dec}(C,C)} r_Q Q \circ \Sigma_{B,C} \circ \Sigma_{A,B} \circ \sum_{P\in \operatorname{dec}(A,A)} t_P P \right\rVert \\
    &\leq \left\lVert \sum_{Q\in \operatorname{dec}(C,C)} r_Q Q \right\rVert \cdot \left\lVert \Sigma_{B,C}\Sigma_{A,B} \right\rVert \cdot \left\lVert \sum_{P\in \operatorname{dec}(A,A)} t_P P \right\rVert \\
    & \leq 2^{n-1}\left \lVert f_2 \right\rVert \cdot \left\lVert f_1 \right\rVert
\end{align*}
concluding the proof.
\end{proof}

Now we can elaborate on what we mean by delooping, block decompositions and cancellations at Lines \ref{alg_line: deloop and keep block decomp} and \ref{alg_line: gaussian eliminate blocks} in Algorithm \ref{Algorithm: extremal lines over Z}. At Line \ref{alg_line: deloop and keep block decomp} the complex $C''$ with a block decomposition and a Morse matching $M$ are formed. Then at Line \ref{alg_line: gaussian eliminate blocks} the Morse complex $M(C'')$ is computed and renamed to be $C'''$.\footnote{Since the matching $M$ will consist of only 0, 1, or 2 block isomorphisms per homological degree, one could avoid the graph theoretic language and simply perform Gaussian elimination on the morphisms of $M$ one-by-one. The result will be identical, but our choice of notation will help us to analyze the result in Lemma \ref{lemma: rank and integer bounds}.}
The concrete block decomposition of $C''$ and the matching $M$ depend on the type of the simple planar arc diagram $P_j$, denoted by $P$ for simplicity. \textbf{Type I} is easy: no new loops are formed, we take $C''$ to have the trivial block decomposition and set $M=\emptyset$. Simply put, if $P$ is a Type I planar arc diagram Lines \ref{alg_line: deloop and keep block decomp} and \ref{alg_line: gaussian eliminate blocks} do nothing.

\textbf{Type II:} By definition, the complex $C'=C \otimes_P \llbracket c \rrbracket_{\operatorname{US}}$ can be written as 
\begin{equation} \label{eq: C'}
    \begin{tikzpicture}
        \node at (0,0) {\begin{tikzcd}
C^i \otimes_P \smoothingZero  \arrow[r,gray] \arrow[rd, "(-1)^i S^i"] & C^{i+1} \otimes_P \smoothingZero  \arrow[r,gray]  \arrow[rd, "(-1)^{i+1}S^{i+1}"] & C^{i+2} \otimes_P \smoothingZero   \\
qC^{i-1} \otimes_P \smoothingOne    \arrow[r,gray]              & qC^i \otimes_P  \smoothingOne   \arrow[r,gray]                  & qC^{i+1} \otimes_P  \smoothingOne
\end{tikzcd}
};
\node at (-4.5,0) {$\cdots$};
\node at (4.5,0) {$\cdots$};
    \end{tikzpicture}
\end{equation}       
where the maps $S^i$ are diagonal matrices of saddles. In \ref{eq: C'} the $q$:s at the bottom row denote the extra grading shifts associated to $\llbracket c \rrbracket_{\operatorname{US}}$ and whereas the homological gradings and the internal gradings of $C^i$ are hidden.  Let us write the $i$:th chain space of $C^i$ as $D^i\oplus E^i$ where $D^i$ is the direct sum of those diagrams which connect the two bottom strands and $E^i$ is the direct sum of the rest:
    $$
    \input{simplePlanars/SP2connsComb}
    $$
    The diagrams of $D^i\otimes_P \smoothingOne$ contain a circle, which can be delooped at the cost of duplicating them. 
    More formally, there is an isomorphism $\begin{bmatrix} \mathfrak d_1 \\ \mathfrak d_2 \end{bmatrix} \colon qD^i\otimes_P \smoothingOne \to  q^2\overline{D^{i-1} \otimes_P \smoothingOne } \oplus  \overline{D^{i-1} \otimes_P \smoothingOne }$ where $\overline{ D^i\otimes_P \smoothingOne }$ denotes the direct sum of diagrams of $ D^i\otimes_P \smoothingOne$ with the new circle removed on every diagram, see Figure \ref{Delooping isomorphism} for the local picture. Now decomposing $C'$ and pulling back the differential we find an isomorphic complex $C''$
    $$
    \begin{tikzpicture}
        \node at (0,0) {
\begin{tikzcd}[column sep=2.5cm]
D^i \otimes_P \smoothingZero   \arrow[rdd, "m^i"]  & D^{i+1} \otimes_P \smoothingZero   \arrow["m^{i+1}",rdd]  & D^{i+2} \otimes_P \smoothingZero         \\
E^i \otimes_P \smoothingZero           & E^{i+1} \otimes_P \smoothingZero      & E^{i+2} \otimes_P \smoothingZero         \\
  \overline{D^{i-1} \otimes_P \smoothingOne }                 & \overline{D^{i} \otimes_P \smoothingOne }          &  \overline{D^{i+1} \otimes_P \smoothingOne } \\
q^2 \overline{D^{i-1} \otimes_P \smoothingOne }                & q^2 \overline{D^{i} \otimes_P \smoothingOne }         & q^2 \overline{D^{i+1} \otimes_P \smoothingOne } \\
qE^{i-1} \otimes_P \smoothingOne                     & qE^{i} \otimes_P \smoothingOne              & qE^{i+1} \otimes_P \smoothingOne   
\end{tikzcd}        
        };
\node at (-6.5,0) {$\cdots$};
\node at (6.5,0) {$\cdots$};
    \end{tikzpicture}
    $$
    where the chain space $(C'')^i$ is the direct sum of the 5 objects in the first column. The differentials of $C''$ are represented by $5\times 5$ whose matrix element at $(1,3)$ is $m^i$. Expanding the pullback differential gives $m^i= (-1)^i\mathfrak d_2 (\pi_D \otimes_P \operatorname{id}) (\operatorname{id}\otimes_P s)  (\iota_D \otimes_P \operatorname{id})  $, where $\iota_D$ and $\pi_D$ are the canonical inclusions and projections related to $C^i=D^i \oplus E^i$ and $s$ is a saddle. Hence $m^i$ are isomorphisms, $M=\{m^i \mid 0\leq i \leq n(L)\}$ is a Morse matching and at Line \ref{alg_line: gaussian eliminate blocks} the Morse complex $M(C'')=C'''$ with $C'''\simeq C'$ is computed.

\textbf{Type III:} We decompose $C^i=D^i \oplus E^i \oplus F^i$ where $D^i$ contains the diagrams which connect the two bottom left strands, $E^i$ those which connect the two bottom right strands and $F^i$ those which connects none of them which each other: 
    $$
    \input{simplePlanars/SP3connsComb}
    $$
    As with Type II, the summands of $D^i\otimes_P \smoothingOne$ and $E^i\otimes_P \smoothingZero$ contain circles which can be delooped so that $C'=C \otimes_P \llbracket c \rrbracket_{\operatorname{US}}$ is isomorphic to the complex $C''$:
    $$
    \begin{tikzpicture}
        \node at (0,0) {
\begin{tikzcd}[column sep=2.5cm]
D^i \otimes_P \smoothingZero \arrow[rdddd, "m_1^i"description] \arrow[rdddddd, dashed,  gray!50]     & D^{i+1} \otimes_P \smoothingZero \arrow[rdddd, "m_1^{i+1}"description] \arrow[rdddddd, dashed,  gray!50]& D^{i+2} \otimes_P \smoothingZero        \\
q^{-1} \overline{E^{i} \otimes_P \smoothingZero }                        & q^{-1} \overline{E^{i+1} \otimes_P \smoothingZero }                                    & q^{-1} \overline{E^{i+2} \otimes_P \smoothingZero }    \\
q \overline{E^{i} \otimes_P \smoothingZero } \arrow[rdddd, "m_2^i"description] \arrow[rdd, dashed,  gray!50]& q \overline{E^{i+1} \otimes_P \smoothingZero } \arrow[rdddd, "m_2^{i+1}"description ] \arrow[rdd, dashed,  gray!50]          & q \overline{E^{i+2} \otimes_P \smoothingZero }    \\
F^i \otimes_P \smoothingZero                            & F^{i+1} \otimes_P \smoothingZero                                        & F^{i+2} \otimes_P \smoothingZero        \\
 \overline{D^{i-1} \otimes_P \smoothingOne }                      &  \overline{D^{i} \otimes_P \smoothingOne }                                  &  \overline{D^{i+1} \otimes_P \smoothingOne }  \\
q^2 \overline{D^{i-1} \otimes_P \smoothingOne }                      & q^2 \overline{D^{i} \otimes_P \smoothingOne }                                  & q^2 \overline{D^{i+1} \otimes_P \smoothingOne }  \\
qE^{i-1} \otimes_P \smoothingOne                         & qE^{i} \otimes_P \smoothingOne                                     & qE^{i+1} \otimes_P \smoothingOne     \\
qF^{i-1} \otimes_P \smoothingOne                         & qF^{i} \otimes_P \smoothingOne                                     & qF^{i+1} \otimes_P \smoothingOne.    
\end{tikzcd}
        };
\node at (-7,0) {$\cdots$};
\node at (7,0) {$\cdots$};
    \end{tikzpicture}
    $$
    Again by expanding, one sees that the matrix elements $m_1^i$ and $m_2^i$ are isomorphisms. On the other hand, the dashed gray arrows are zero morphisms since the first tensor component of their expansions contains an inclusion $D^i \hookrightarrow D^i\oplus E^i$ followed by a projection $D^i\oplus E^i \twoheadrightarrow E^i$ (or alternatively $E^i \hookrightarrow D^i\oplus E^i \twoheadrightarrow D^i$). Hence $M=\{m^i_1,m_2^i \mid 0 \leq i \leq n(L)\}$ is a Morse matching, and at Line \ref{alg_line: gaussian eliminate blocks} we can perform Gaussian elimination on all of the morphisms $m_1^i$ and $m_2^i$ to obtain the Morse complex $M(C)=C'''$. 
    
    \textbf{Type IV:} Let us decompose $C^i=D^i \oplus E^i$ where $D^i$ consists of those diagrams with $\smoothingOne$ connectivity and $E^i$ of those with $\smoothingZero$ connectivity. The compositions contain many new circles, but for the moment we only deloop the ones at the bottom of $D^i\otimes_P \smoothingOne$ and the inner circle of $E^i \otimes_P \smoothingZero$. This yields the complex $C''$:
        $$
    \begin{tikzpicture}
        \node at (0,0) {
        
\begin{tikzcd}[column sep=2.5cm]
D^i \otimes_P \smoothingZero  \arrow[rddddd, dashed,  gray!50] \arrow[rddd, "m_1^i" description]  & D^{i+1} \otimes_P \smoothingZero \arrow[rddddd, dashed,  gray!50]  \arrow["m_1^{i+1}" description,rddd]  & D^{i+2} \otimes_P \smoothingZero         \\
q^{-1} \overline{E^{i} \otimes_P \smoothingZero }          & q^{-1} \overline{E^{i+1} \otimes_P \smoothingZero }      & q^{-1} \overline{E^{i+2} \otimes_P \smoothingZero }         \\
q \overline{E^{i} \otimes_P \smoothingZero } \arrow[rd, dashed,  gray!50] \arrow[rddd, "m_2^i" description]         & q \overline{E^{i+1} \otimes_P \smoothingZero } \arrow[rd, dashed,  gray!50]  \arrow[rddd, "m_2^{i+1}" description]   & q \overline{E^{i+2} \otimes_P \smoothingZero }         \\
 \overline{D^{i-1} \otimes_P \smoothingOne }                 & \overline{D^{i} \otimes_P \smoothingOne }          &  \overline{D^{i+1} \otimes_P \smoothingOne } \\
q^2 \overline{D^{i-1} \otimes_P \smoothingOne }                 & q^2 \overline{D^{i} \otimes_P \smoothingOne }          & q^2 \overline{D^{i+1} \otimes_P \smoothingOne } \\
qE^{i-1} \otimes_P \smoothingOne                     & qE^{i} \otimes_P \smoothingOne              & qE^{i+1} \otimes_P \smoothingOne   
\end{tikzcd}

                };
\node at (-7,0) {$\cdots$};
\node at (7,0) {$\cdots$};
    \end{tikzpicture}
    $$
    from we set $M= \{m_1^i, m_2^i \mid 0 \leq i \leq n(L)\}$. The set $M$ is a Morse matching since the gray dashed matrix elements on $C''$ are zero morphisms. To finish off with a simple complex, we deloop the rest of the circles of $M(C)$ before calling the result $C'''$. (To more accurately follow the pseudocode of Algorithm \ref{Algorithm: extremal lines over Z} we could have delooped everything first before cancellation, but our equivalent two-step delooping process kept the notation lighter.)

\subsection{Validity}\label{subsec: validity}

\term{The truncation} of a chain complex $(C,d)$ at $k\in \mathbb Z$, denoted by $\tau_{\leq k} (C)$, is the chain complex with differential $\partial$ defined as
$$
\tau_{\leq k}(C)^i=
\begin{cases}
      C^i, & \text{if}\ i\leq k \\
      0, & \text{otherwise}
    \end{cases}
\qquad
\partial^i=
\begin{cases}
      d^i, & \text{if}\ i< k  \\
      0, &  \text{otherwise}.
    \end{cases}
$$
With this notation, Line \ref{alg_line: truncation for Z} can be written as $C \gets \tau_{\leq k+1}(C''')$.
We say that chain complexes $(C,d)$ and $(D, \partial)$ over an additive category are \term{chain homotopic up to $k\in \mathbb Z$}, denoted $C\simeq_{\leq k} D$, if for all $i\leq k$ there are morphisms 
$$
f^i\colon C^i \to D^i, \qquad g^i\colon D^i \to C^i, \qquad s^i \colon C^i \to C^{i-1}, \qquad t^i \colon D^i \to D^{i-1} 
$$
so that $f^i$ and $g^i$ commute with the differentials and $s^i$ and $t^i$ give the homotopy to the identities:
\begin{alignat*}{2}
    f^{i+1} d^i &= \partial^i f^i,    \qquad g^if^i- \operatorname{id}_{C^i} &&=d^{i-1}s^i +s^{i+1} d^i \\
    g^{i+1} \partial^i &=d^i g^i,    \qquad f^ig^i- \operatorname{id}_{D^i} &&=\partial^{i-1}t^i +t^{i+1} \partial^i.
\end{alignat*}
The usual definition of a chain homotopy is retrieved with $k=\infty$ and the following lemma is a collection of some useful basic properties.
\begin{lemma} \label{Lemma: basic properties of k-homotopy}
    \begin{enumerate}
        \item $C\simeq D$ implies $C\simeq_{\leq k} D$ for all $k$. \label{Lem:item homotopy->k-homotopy}
        \item $\tau_{\leq k+1} (C) \simeq_k C$, for all $k$.\label{Lem:item truncated homotopy}
        \item If $F$ is an additive functor, then $C\simeq_{\leq k} D$ implies $F(C)\simeq_{\leq k} F(D)$. \label{Lem:item additive functor preserves}   
        \item If $C$ and $D$ are chain complexes over an abelian category and $C\simeq_{\leq k} D$, then $H^i(C) \cong H^i(D)$ for all $i <k$.\label{Lem:item homotopy->homology} 
        \item Let $P$ be a simple planar arc diagram, $c$ a crossing, $C$ and $D$ complexes over $\Cobl^{\mathbb Z}(2n)$ which are supported on the nonnegative homological degrees and $C \simeq_{\leq k} D$. Then $C \otimes_P \llbracket c \rrbracket_{ \operatorname{US}} \simeq_{\leq k} D \otimes_P \llbracket c \rrbracket_{ \operatorname{US}}$. \label{Lem:item tensor -> k-homotopy}
    \end{enumerate}
\end{lemma}
\begin{proof}
    Claims \ref{Lem:item homotopy->k-homotopy}, \ref{Lem:item truncated homotopy}, and \ref{Lem:item additive functor preserves} follow directly from the definitions, \ref{Lem:item homotopy->homology} is a slight modification of the standard result and \ref{Lem:item tensor -> k-homotopy} uses the fact that $\llbracket c \rrbracket_{\operatorname{US}}$ is supported in nonnegative homological degrees.
\end{proof}

For analyzing correctness and estimating the time complexity of the algorithms $\mathcal{A}_k$ we need to define the following \term{intermediate complexes}. We denote by $C_{[t]}$ the complex stored at variable $C$ before running the for loop of $\mathcal A_k$ (Lines \ref{alg_line: compose}-\ref{alg_line: truncation for Z} in Algorithm \ref{Algorithm: extremal lines over Z}) for $t$:th time where $t\leq n(L)-1$ and by $C_{[n(L)]}$ the complex which is stored at variable $C$ after the for loop.

\begin{proof}[Proof that $\mathcal A_k$ computes $\operatorname{Kh}^{i,\ast} (L)$ for $i\leq k-n_-(L) $ correctly. ]
    At every line of Algorithm \ref{Algorithm: extremal lines over Z} one can use a corresponding claim of Lemma \ref{Lemma: basic properties of k-homotopy} to see that chain homotopy up to $k+1$ is preserved: At line \ref{alg_line: initialize} we initialize the complexes to be the same: $C_{[1]}=\llbracket c_1 \rrbracket_{\operatorname{US}}$ so in particular $C_{[1]}\simeq_{\leq k+1}\llbracket T_1 \rrbracket_{\operatorname{US}}$. Then to proceed with induction, we use Claim \ref{Lem:item tensor -> k-homotopy} at Line \ref{alg_line: compose}, Claim \ref{Lem:item homotopy->k-homotopy} at Lines \ref{alg_line: deloop and keep block decomp} and \ref{alg_line: gaussian eliminate blocks}, Claim \ref{Lem:item truncated homotopy} at Line \ref{alg_line: truncation for Z} and Claim \ref{Lem:item additive functor preserves} at Line \ref{alg_line: TQFT}. This guarantees that 
    $$
    F(C_{[n]}) \simeq_{\leq k+1} F\llbracket L \rrbracket_{\operatorname{US}}=u^{n_-(L)} q^{-n_+(L)+2n_-(L)}\CKh(L)
    $$ and hence by Claim \ref{Lem:item homotopy->homology} the algorithm $\mathcal{A}_k$ computes $\operatorname{Kh}^{i,\ast} (L)$ for $i\leq k-n_-(L) $ correctly.
\end{proof}

\subsection{Proof of polynomial time complexity for \texorpdfstring{$\mathcal A_{k}$}{Ak}} \label{subsec: polytime}

In order to measure the sizes of objects and limit the number of operations, we establish two definitions.
Firstly, for an object $B$ in $\Cobl^{\Z} (2n)$,  we set  $\rankCob(B)=n$ where $n$ is the number of loopless diagrams in the decomposition (\ref{eq:decomposition of objects}). Secondly, for a complex $(C,d)$ over $\Cobl^{\mathbb Z}(2n)$, we define $\lVert C \rVert_i= \max \{ \lVert f \rVert : f \text{ matrix element of } d^i \}
$ for every integer $i$.

\begin{lemma} \label{lemma: rank and integer bounds} 
    The following bounds hold for intermediate complexes $C_{[t]}$ of the algorithm $\mathcal A_k$
    \begin{alignat}{2}
    \rankCob( C_{[t]}^i) &\leq \binom{t}{i} \qquad &&\text{for } t<n(L)  \label{eq: rankcob bound t<n} \\
    \rankCob( C_{[n(L)]}^i) &\leq 2\binom{n(L)}{i} && \label{eq: rankcob bound n}\\
    \lVert C_{[t]} \rVert_i &\leq \exp_{2}\left( \left(\frac{g}{2}+1\right)\left(\binom{t+1}{i+1}-1 \right) \right)  \qquad &&\text{for all } t \label{eq: integer bound for intermediate complexes} 
    \end{alignat}
    where $g$ is the girth of the nice scanning sequence and $\exp_{2}(a)$ denotes $2^a$. 
\end{lemma}
\begin{proof}
    All three claims are proven with an induction on $t$. For (\ref{eq: rankcob bound t<n}) and (\ref{eq: rankcob bound n}) the base case $t=1$ is trivial and the induction step splits into 4 simple cases depending on the type of the planar arc diagram $P$. For example, with Type II diagram and using the notation of Section \ref{subsec: blocks} we can directly compute
    $$
    \rankCob( C_{[t+1]}^i) = \rankCob ( E^i \otimes_P \smoothingZero \oplus q \overline{D^{i-1} \otimes_P \smoothingOne} \oplus E^{i-1} \otimes_P \smoothingOne ) \leq \binom{t}{i}+ \binom{t}{i-1}=\binom{t+1}{i}.
    $$
    The additional factor $2$ in (\ref{eq: rankcob bound n}) is associated to the extra deloopings involved with the Type IV diagram. 

    For (\ref{eq: integer bound for intermediate complexes}) the base case $t=1$ is routinely verified and we assume that (\ref{eq: integer bound for intermediate complexes}) holds for some $t$. It is straightforward to see that 
    \begin{equation}
     \lVert C_{[p]} \rVert_i=0 \qquad \text{when } i<0 \text{ or } i\geq p \label{eq: norm C vanishes}  
    \end{equation}
    since the complex $\llbracket T_p \rrbracket$ is also $0$ in homological degrees $i<0$ and $i>p$. The morphisms of $C_{[t]}$ are also the morphisms of untruncated Morse complex $M(C'')=C'''$ which are sums of zig-zag paths. In the Morse complexes one can see that there are at most 3 paths on every sum and in Type III and IV they will use either $(m_1^i)^{-1}$, or $(m_2^i)^{-1}$ or neither but not both. The norm of each of the three summands will be bounded by $2^{g/2-1} \max( \lVert C_{[t]}\rVert_i ,1)\max( \lVert C_{[t]}\rVert_{i-1} ,1)$ where the $2^{g/2-1}$ term comes from Lemma \ref{lemma: horizontal composition}. In particular we get
    \begin{equation}
        \lVert C_{[t+1]} \rVert_i \leq 3\cdot 2^{\frac{g}{2}-1} \max( \lVert C_{[t]}\rVert_i ,1)\max( \lVert C_{[t]}\rVert_{i-1} ,1). \label{eq: C norm ineq}
    \end{equation}
    If $1\leq i \leq t-1$, we can use the induction assumption and (\ref{eq: C norm ineq}) to compute: 
    \begin{align*}
        \lVert C_{[t+1]} \rVert_i & \leq 3 \cdot 2^{g-1} \exp_{2} \left( (g+1) \left( \binom{t+1}{i+1} -1 \right) \right) \exp_{2} \left( (g+1) \left( \binom{t+1}{i} -1 \right) \right) \\
        &\leq \exp_2\left( (g+1) \left( \binom{t+2}{i+1}-1 \right) \right).
    \end{align*}
    For $i=0$ and $i=t$, the claim is obtained by using (\ref{eq: norm C vanishes}), (\ref{eq: C norm ineq}), and the induction assumption.
\end{proof}

\begin{proof}[Proof of the running time estimate for $\mathcal{A}_k$]
    As an input to $\mathcal A_k$, let us limit to link diagrams $L$ with nice scanning sequences $\mathcal S$ with a global upper bound $\operatorname{girth}(\mathcal S) \leq g$  on them. By analyzing the pseudocode of Algorithm \ref{Algorithm: extremal lines over Z} line by line, we aim to show that on this fixed set of inputs $\mathcal{A}_k$ runs in polynomial time with respect to $n(L)$. 

    By Lemma \ref{lemma: rank and integer bounds} the intermediate complex $C_{[t]}$ has at most $2 \sum_{j=0}^{k+1} \binom{t}{j}=\mathcal O(t^{k+1})$ objects. Since each matrix element of $C_{[t]}$ connects a pair of objects, there are at most $\mathcal O(t^{2k+2})$ matrix elements $f$ and by Lemma $\ref{lemma: rank and integer bounds}$ their norms are bounded by
    $$
    \log_2 \lVert f \rVert \leq \max_{j\in \{0,\dots, k+1\} } \left( \binom{t+1}{j+1} -1  \right)(g+1) = \mathcal O(t^{k+1}).
    $$
    To generate $C'= C_{[t]}\otimes_{P_t} \llbracket c_{t+1} \rrbracket_{\operatorname{US}}$ from $C_{[t]}$ we perform at most $\mathcal O(t^{k+1})$ vertical compositions of objects, $\mathcal O(t^{2k+2})$ vertical compositions of morphisms and create at most $\mathcal O(t^{k+1})$ new saddle morphisms. When creating $C''$ from $C'$ we deloop  $\mathcal O(t^{k+1})$ objects and compose $\mathcal O(t^{2k+2})$ morphisms with the delooping $\mathfrak d_1$, $\mathfrak d_2$ and relooping $\mathfrak r_1$, $\mathfrak r_2 $ morphisms. In constructing $C'''$ from $C''$ we perform $\mathcal O(t^{2k+2})$ horizontal compositions and summations of morphisms. By Lemma \ref{lemma: horizontal composition} each horizontal composition takes polynomial time with respect to $\mathcal O(t^{k+1})$ since $\operatorname{girth} (T_t)\leq \operatorname{girth}(\mathcal S) \leq g$. The time taken for the truncation at Line \ref{alg_line: truncation for Z} is negligible. We have now shown that there exists some constant $a\in \mathbb Z$, which is independent of $t$, and which ensures that only $\mathcal O(t^a)$ operations are performed at the $t$:th iteration of the for loop. 
    
    To execute the whole for loop, $n(L)-1$ iterations are made with different $t\leq n(L)-1$ values, so in particular at most
    $$
    \sum_{t=1}^{n(L)-1} \mathcal{O}(t^a)\leq \sum_{t=1}^{n(L)-1} \mathcal{O}(n(L)^a) =\mathcal O(n(L)^{a+1})
    $$
    operations are performed before the Smith normal form at Line \ref{alg_line: homology}. Smith normal form of an integer matrix $A^{n\times m}$ can be computed in polynomial time with respect to $n$, $m$ and $\max_{i,j}\{\log_2(|A_{i,j}|)\}$, \cite{10.1145/236869.237084}. By Lemma \ref{lemma: rank and integer bounds} the matrices of $F(C_{[n(L)]}) $  admit polynomial upper bounds with respect to $n(L)$ for these quantities. Hence $\mathcal A_k$ runs in polynomial time for the inputs $(L,\mathcal{S})$ with $\operatorname{girth}(\mathcal S) \leq g$. Thus the running time of algorithm $\mathcal{A}_k$ for all inputs $(L,\mathcal S)$ is $\mathcal O(f(\operatorname{girth}(S)) \cdot p_k(n(L)))$ for some function $f$ and polynomial $p_k$.
\end{proof}

\section{Binomial rank bounds via nice scanning sequences}

A graph $G$ is \term{connected} if there is a path between every two vertices is $G$. In addition, a graph $G$ is called \term{$2$-connected} if for every vertex $v\in G$ the induced graph of $G\setminus \{v\}$ is connected. 

\begin{lemma}\label{Lemma: 2-connected graph property}
    Let $G=(V,E)$ be a $2$-connected graph and $A\subset V$ a subset of vertices so that $\# A\geq 1$ and $\#(V \setminus A) \geq 2$. If the induced graphs of $A$ and $V \setminus A$ are connected there exists $x\in V\setminus A$ so that the induced subgraphs of $A\cup \{x\}$ and $V\setminus (A\cup \{x\})$ are also connected. 
\end{lemma}
\begin{proof}
    Assume that $x$ is a vertex of $V \setminus A$ for which the following maximum is obtained: 
    $$
    M=\max \big\{ \# C \mid x\in (V\setminus A), \ A\cup \{x\} \text{ connected, }  C \text{ is a component of } V\setminus (A\cup \{x\}) \big\}.
    $$
    If $M=\#V - \#A -1$, then $V \setminus (A \cup \{x\})$ is connected so we can assume towards contradiction that $M <\# V - \# A -1$.
    
    Let $C$ and $D$ be two components of $V \setminus (A \cup \{x\})$ with $M= \# C$ and pick any vertices $c\in C$ and $d\in D$. Since $G$ is $2$-connected, there is a path $P$ from $d$ to $c$ which does not cross $x$. Denote $d'\in D$ as the last vertex of $P$ in $D$, so that $V \cup \{d'\}$ is connected and write $C'$ for the component of $c$ in $ V\setminus (A\cup \{d'\})$. It follows that $\#C < \#C'$ since the component $C'$ additionally contains the vertex $x$. On the other hand, $\# C' \leq M= \# C$ by maximality of $M$ which gives us the desired contradiction.
\end{proof} 

A link diagram is called \term{reduced} if it does not contain any nugatory crossings, see Figure \ref{fig: nugatory crossing}. It is straightforward to see that a connected link diagram $L$ can be transformed into an isotopic, reduced, and connected link diagram $L'$ with $n_+(L) \geq n_+(L')$ and $n_-(L) \geq n_-(L)$. 

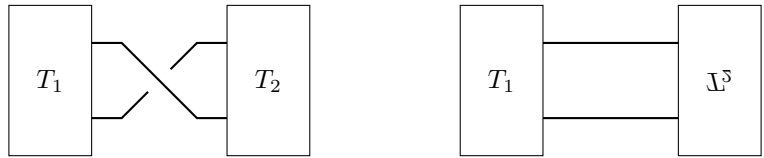
\begin{figure}[ht]
    \centering
    \input{simplePlanars/nugatory}
    \caption{A nugatory crossing (left) and removing it (right). Removing the nugatory crossing flips the right hand side tangle $T_2$ by  $180^{\circ}$ around the $x$-axis which is indicated by the flipping the symbol $T_2$ to $\rotatebox[origin=c]{180}{\reflectbox{$T_2$}}$.  }
    \label{fig: nugatory crossing}
\end{figure}

\begin{lemma}\label{lemma: every link has nice scan}
    Every connected reduced link diagram admits a nice scanning sequence on it.
\end{lemma}
\begin{proof}
    Let $L$ be a connected reduced link diagram and $G_{L}=(V,E)$ the $4$-valent non-simple planar graph of $L$, whose vertices are the crossings of $L$ and whose edges correspond to the strands between crossings. From the fact that $G_{L}$ is connected and reduced, one can see that $G_{L}$ is $2$-connected. (Any vertex separator of $G_{L}$ has to correspond to a nugatory crossing in $L$, but $L$ does not have nugatory crossings since it is reduced.) 

    The sequence of crossings $c_1,\dots, c_n$ for $\mathcal S$ are retrieved from the graph $G_L$ by a greedy algorithm. Start by picking any vertex $c_1\in V$ and observe that the induced graphs $\{c_1\}$ and $V \setminus \{c_1\}$ are connected. Suppose next that $c_1,\dots , c_i$ have been picked, so that $i\leq n-2$ and the induced subgraphs of $\{c_1,\dots, c_i\}$ and $V \setminus \{c_1,\dots, c_i\}$ are connected. By Lemma \ref{Lemma: 2-connected graph property} there is a vertex $c_{i+1}$ for which that the induced subgraphs of $\{c_1,\dots, c_{i+1}\}$ and $V \setminus \{c_1,\dots, c_{i+1}\}$ also remain connected. Finally, once $c_1,\dots, c_{n-1}$ have been chosen, add the last vertex in to complete the sequence of crossings $c_1,\dots, c_n$.

    The sequence of subtangle diagrams $T_1,\dots, T_n$ and simple planar diagrams $P_1,\dots, P_{n-1}$ can be retrieved by setting $T_1=c_1$ and analyzing how $c_{i+1}$ inductively glues into $T_i$. The induced graph of $c_1,\dots, c_{i}$ remains connected after gluing in $c_{i+1}$ which is why $P_i$ has to connect $c_{i+1}$ to $T_i$ with $1,2,3$, or $4$ strands. Moreover the vertex $c_{i+1}$ cannot be connected to $T_i$ with $4$ strands except when $i+1=n$, since the induced graph of $\{c_{i+1},\dots , c_{n}\}$ is connected. Finally, no strand can go from a crossing to itself since $L$ was reduced. Thus $c_{i+1}$ is glued to $T_i$ with a Type I, II, or III simple planar arc diagram, when $i\leq n-2$ and with a Type IV diagram when $i=n-1$ and so $\mathcal{S}=((T_1,\dots T_n),(c_1,\dots, c_n), (P_1,\dots, P_{n-1}))$ forms a nice scanning sequence for $L$.
\end{proof}

\begin{proof}[Proof of Proposition \ref{proposition: rank bounds}]
    Let $L'$ denote the link diagram which is obtained from $L$ by removing the nugatory crossings. By Lemma \ref{lemma: every link has nice scan} there exists a nice scanning sequence $\mathcal S$ on $L'$. When running the algorithm $A_{k}$ on $(L', \mathcal S)$ with $k=n(L')$ no nonzero chain spaces get truncated so there is an honest chain homotopy $u^{-n_-(L)} q^{n_+(L)-2n_-(L)}C_{[n(L)]}\simeq \llbracket L' \rrbracket$.  Hence for any field $\mathbb F$ we can compute
    \begin{multline*}
    \dim_{\mathbb F} (\operatorname{Kh}^{i,\ast}(L; \mathbb F)) =\dim_{\mathbb F} (\operatorname{Kh}^{i,\ast}(L'; \mathbb F)) \leq  \rank \left(F(u^{-n_-(L)} q^{n_+(L)-2n_-(L)}C_{[n(L)]})^i\right) \\ \leq \rankCob\left(C_{[n(L)]}^{i+n_-(L)}\right) \leq 2\binom{n(L')}{i+n_-(L')} \leq 2\binom{n(L)}{i+n_-(L)}     
    \end{multline*}
    by bounding the homology with the size of the chain complex and using Lemma \ref{lemma: rank and integer bounds}. 
    \end{proof}

\section{Asymptotical strictness of rank bounds}\label{section: Asymptotict nonvanishing}
Proposition \ref{proposition: asymptotical strictness of the bounds}  is a non-vanishing result which claims an asymptotic lower bound on the the rank of certain Khovanov homology groups. To prove it we consider the chain complexes $\CKh (L_t)$, where $L_t$ is the braid closure of the braid diagram $(\sigma_1 \sigma_3 \sigma_2^4)^t \sigma_1 \sigma_3$, see Figure \ref{fig:Lt alpha beta Cw}. On $\CKh(L_t)$ we employ an algorithmic Morse matching $M_{\operatorname{lex}}$ from the previous work of the first author \cite{kelo2025torus4}. This produces a Morse complex $\Mlex (\CKh(L_t))$ where we find $2^t$ explicit, distinct, non-vanishing  homology cycles. 

Since we are using $\Mlex$ on link diagrams $L_t$ which only contain crossings of the form $\crosPos$ the construction of $\Mlex$ can be simplified quite a bit. In \cite{kelo2025torus4} the matching $\Mlex$ was defined on Bar-Natan's tangle complexes but since we only work with specific links, we can define it directly for $\CKh(L_t)$ which are chain complexes of free $\mathbb Z$-modules. The direct sum decomposition which we fix on $\CKh(L_t)$ consists of single copies of $\mathbb Z$: one for each $2$-coloring of each smoothing of $L_t$. The two grayscale colors we use are dashed gray \begin{tikzpicture}[baseline={(0,-0.1)}]
    \draw[thinredline] (0,0) circle (0.125);
\end{tikzpicture} for the unit $1$ and thick black \begin{tikzpicture}[baseline={(0,-0.1)}]
    \draw[thinblueline] (0,0) circle (0.125);
\end{tikzpicture} for the counit $x$.
This allows us to write merge $m$ and split $\Delta$ maps of the Khovanov complex as
\begin{alignat*}{2}
     &m(\grayCirc \otimes \grayCirc)= \grayCirc \qquad&&\Delta(\grayCirc)= \grayCirc \otimes \blackCirc + \blackCirc \otimes \grayCirc  \\
     &m(\blackCirc \otimes \grayCirc)=m(\grayCirc \otimes \blackCirc) =\blackCirc \qquad&&\Delta(\blackCirc)=  \blackCirc \otimes \blackCirc  \\
     &m(\blackCirc \otimes \blackCirc )=0.     
\end{alignat*}
In particular, splitting out $\blackCirc$ will always be an isomorphism and our Morse matching $\Mlex$ will consist of such splits.

We order the crossings of $L_t$ from bottom to top or equivalently the characters of $(\sigma_1 \sigma_3 \sigma_2^4)^t \sigma_1 \sigma_3$ from left to right. For a matrix element $f$ of $\CKh(L_t)$ which changes a $0$ smoothing to a $1$ smoothing at the $k$:th crossing, we denote $i(f)=k$. Diagrammatically, we define 
$$
M=\left\{ f \in \{ \text{matrix elements of } \CKh(L_t) \} \middle\vert \begin{array}{c}
f \text{ splits out a }\blackCirc \text{ loop which completely located} \\
\text{at or below the $i(f)$:th crossing in the diagram } L_t
\end{array}  \right\}.
$$
For every vertex $x$ we define $N(x)$ to be the set of neighboring morphisms from $M$:
$$
N(x) =\{ f\colon a\to x \mid  f \in M \} \cup \{ g\colon x\to b \mid  g \in M \}.
$$
This allows us to finally set
$$
\Mlex= \{ f\colon x \to y \mid  f\in M, \ i(f) \leq i(g)  \text{ for all } g\in N(x)\cup N(y) \}.
$$
Proposition 3.5 of \cite{kelo2025torus4} shows that $\Mlex$ is a Morse matching for every tangle/link diagram and in particular:
\begin{lemma}\label{lemma: Mlex is Morse matching}
    The matching $\Mlex$ is a Morse matching on $\CKh(L_t)$.
\end{lemma}

\begin{figure}
    \centering
    \input{Strictness_diags/alpha_beta}
    \caption{The link diagram $L_t$ and local pictures $\alpha$, $\beta$. The special vertices $u_w$ of $G(\CKh(L_t), \Mlex)$ are constructed by taking a tuple $w\in \{\alpha, \beta \}^{t}$ and gluing the local pictures of $w$ on top of each other, $1$ smoothing the two last vertices and using the dashed gray color on all loops.}
    \label{fig:Lt alpha beta Cw}
\end{figure}
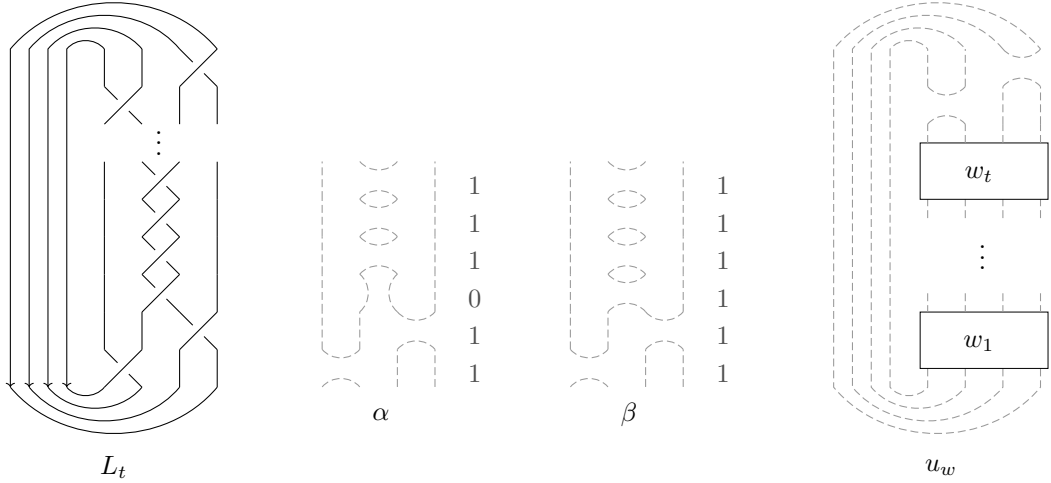

In Figure \ref{fig:Lt alpha beta Cw} we construct special vertices $u_w$ of $G(\CKh(L_t), \Mlex)$ for every $w\in \{\alpha, \beta \}^{t}$. These pictures $u_w$ end up representing distinct homology cycles in the Morse complex $\Mlex(\CKh(L_t))$.
The proof of Proposition \ref{proposition: asymptotical strictness of the bounds} comes as a consequence of the following three claims on the diagrams $u_w$:
\begin{enumerate}
    \item \label{claim: u_w unmatched}
    The vertex $u_w$ is unmatched in $\Mlex$ and thus it represents a copy of $\mathbb Z$ in homological degree $6t+2- \#\{j \mid w_j=\alpha \}$ of the Morse complex $M_{\operatorname{lex}}(\CKh (L_t))$.

    \item \label{claim: d to u_w is 0}
    Suppose the homological degree of $u_w$ is $h$ and denote the projection of $(\Mlex(L_t))^h$ to $u_w$ by $\pi$. Then $\pi\partial^{h-1} =0$.
    
    \item \label{claim: d from u_w is 0}
    Suppose the homological degree of $u_w$ is $h$ and denote the inclusion of $u_w$ into $(\Mlex(L_t))^h$ by $\iota$. Then $\partial^{h}\iota =0$.
\end{enumerate}

\begin{proof}[Proof of Proposition \ref{proposition: asymptotical strictness of the bounds}]
     By Claim \ref{claim: u_w unmatched} every $u_w$ represents a copy of $\mathbb Z$ in the complex $\Mlex(\CKh(L_t))$ and by Claims \ref{claim: d to u_w is 0} and \ref{claim: d from u_w is 0} that $\mathbb Z$ fully contributes to the homology $H(\Mlex( \CKh(L_t)))\cong \HKh(L_t)$. This yields
     $$
     \HKh^{n(L_t)-k, \ast} (L_t)\cong \mathbb Z^{\binom{t}{k} } \oplus A_{k,t} 
     $$
     for some $\mathbb Z$-modules $A_{k,t}$. Thus
     $$
     \liminf_{t\to \infty}   \left(\frac{\operatorname{rank}(\operatorname{Kh}^{n(L_t)-k,\ast} (L_t))}{\binom{n(L_t)}{n(L_t) -k}} \right)  \geq \liminf_{t\to \infty }\left(\frac{\binom{t}{k}}{\binom{6t+2}{6t+2-k}} \right) = \frac{1}{6^k}
     $$
    concluding the proof.    
\end{proof}

\begin{proof}[Proof of Claim \ref{claim: u_w unmatched}] 
    Changing any $0$-smoothing of $u_w$ to a $1$-smoothing generates a split morphism $f\colon u_w \to y$. No matter what crossing we chose to act on by $f$, neither of the loops in $y$ which are split by $f$ will be completely contained below $i(f)$ in $y$. Hence $f \notin M$. On the other hand, changing any $1$-smoothing of $u_w$ to a $0$-smoothing is geometrically a split in the reverse direction. However $u_w$ does not contain any $\blackCirc$ loops and every split morphisms generates at least one $\blackCirc$ loop, so there cannot exist an edge $g\colon x \to u_w$ in the graph $G(\CKh(L_t))$. We conclude that $N (u_w) = \emptyset$ and so $u_w$ is unmatched.
\end{proof}

\begin{proof}[Proof of Claim \ref{claim: d to u_w is 0}]
    Let $x$ be an unmatched vertex of $G(\CKh(L_t), \Mlex)$ with homological degree $h-1$. In the previous claim we argued that there are no edges in $x \to u_w$ in the graph $G(\CKh(L_t))$. It follows that cannot exist any zig-zag paths from $x$ to $u_w$. Thus $\partial_{u_w,x}=0$ and so $\pi \partial^{h-1}=0$ as well.
\end{proof}

\subsection{Proof of Claim \ref{claim: d from u_w is 0}}
Let $y$ be an unmatched vertex of $G(\CKh(L_t), \Mlex)$ with homological degree $h+1$ and set
$
P=\{ \text{paths: } u_w\to y \}.
$
In Claim \ref{claim: d to u_w is 0} we deduced $\pi \partial^{h-1}=0$ from the fact that there were no paths nontrivial paths from any vertex to $u_w$ in $G(\CKh(L_t), \Mlex)$. Since $P$ can be non-empty the same tactic cannot be used to  prove $\partial^h \iota=0$. Instead, we will show that $P$ can be arranged into distinct pairs $(p,\lambda(p))$ for which $R(\lambda(p))=-R(p)$. From this pairing, one can easily compute:
$$
\partial_{y,u_w} = \sum_{q\in P} R(q) = \sum_{(p,\lambda(p))} R(p)+R(\lambda(p))=0.
$$

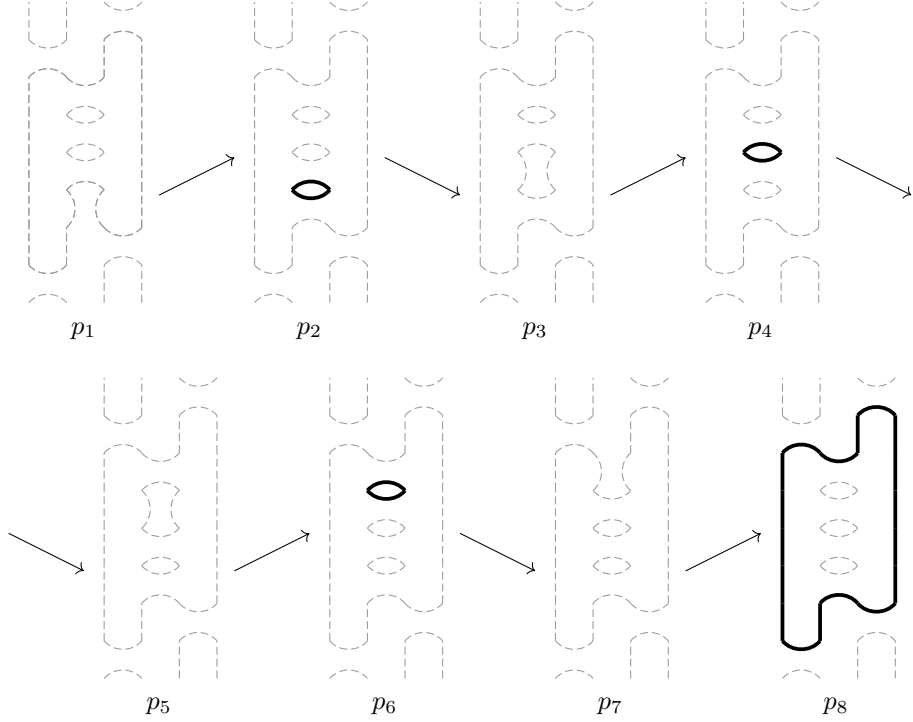
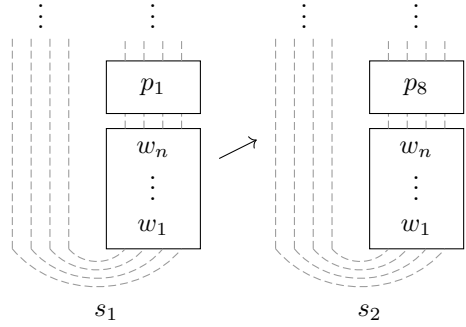
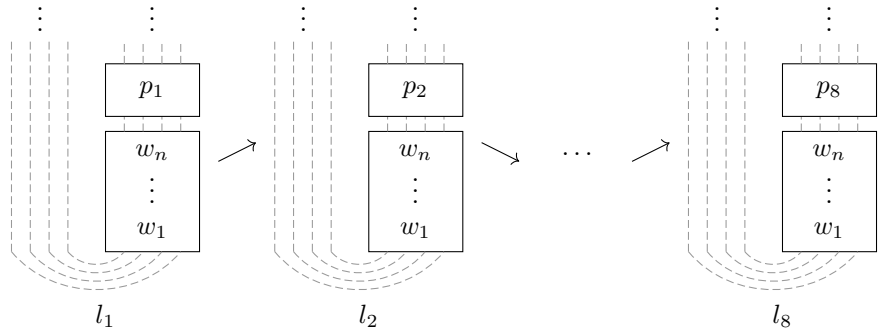
\begin{figure}
    \centering
    \begin{subfigure}{\textwidth}
        \centering        
        \input{Strictness_diags/L-path/Lpath}
        \caption{The local pictures $p_1,\dots ,p_8$ which are used to make the $s$- and $l$-subpaths.}
    \end{subfigure}
        \begin{subfigure}{\textwidth}
        \centering
        \input{Strictness_diags/L-path/boxS}
        \caption{An $s$-subpath. The lower boxes $(w_1,\dots, w_n) \in\{\alpha, \beta \}^n$ contain  $n$-copies of the $\alpha$ and $\beta$ pictures from Figure \ref{fig:Lt alpha beta Cw} stacked on top of each other.}
        \label{Figure: Spath}
    \end{subfigure}
        \begin{subfigure}{\textwidth}
        \centering
        \input{Strictness_diags/L-path/boxL}
        \caption{An $l$-subpath. The lower boxes $(w_1,\dots, w_n) \in\{\alpha, \beta \}^n$ contain  $n$-copies of the $\alpha$ and $\beta$ pictures from Figure \ref{fig:Lt alpha beta Cw} stacked on top of each other.}
        \label{Figure: Lpath}
    \end{subfigure}
    \caption{The $s$-subpaths, Figure \ref{Figure: Spath}, and $l$-subpaths, Figure \ref{Figure: Lpath}, are paths in $G(\CKh(L_t), \Mlex)$. The vertices of $s$- and $l$-subpaths are global diagrams whose bottom halves match the pictures of $\ref{Figure: Spath}$ and \ref{Figure: Lpath}.}
\end{figure}

We say that $l_1 \to \dots \to l_8$ is an $l$-subpath if when cutting the diagrams of $l_1,\dots, l_8$  above the $6n+8$:th crossing for some $0\leq n\leq t-1$  generates the pictures in Figure \ref{Figure: Lpath}. Similarly, we call $s_1 \to s_2$ a $s$-subpath if a similar cut generates diagrams of Figure \ref{Figure: Spath}. In order to prove that every $p\in P$ contains either an $s$- or an $l$-subpath we will need the following result.

\begin{lemma}[Lemma 3.4 from \cite{kelo2025torus4}]\label{lemma: Mlex i ineq}
    Let $(b \to a) \in \Mlex$ and $a\to b \to c $ be a subpath of a zig-zag path between unmatched cells in the graph $G(\CKh(L_t),\Mlex)$. Then $i(b\to a) \leq i(b\to c)$.
\end{lemma}
Let  $p=(v_1\to\dots \to v_{2m})\in P$ be a zig-zag path which does not contain an $s$-subpath. By observing the diagram of $u_w$ one can see that the $v_1 \to v_2$ has to be of the form $s_1\to s_2$ from Figure \ref{Figure: Spath} or $l_1\to l_2$ from Figure \ref{Figure: Lpath}. Since $p$ does not have an $s$-subpath the latter must hold. Next, let $v_{2k+1} \to v_{2k+2}$ be an edge in $p$ which is of the form $l_1\to l_2$ and for which $i(v_{2k+1}\to v_{2k+2}) \leq i(v_{2j+1}\to v_{2j+2})$ for every other edge $v_{2j+1}\to v_{2j+2}$ in $p$ of the form $l_1 \to l_2$. By verifying that vertices of type $l_2$ are matched to vertices of type $l_3$ we obtain that $v_{2k+3}$ is of type $l_3$. From Lemma \ref{lemma: Mlex i ineq} we get $i(v_{2k+3}\to v_{2k+4}) \leq i(v_{2k+3}\to v_{2k+2})$. The only edges from $v_{2k+3}$ with $i$ value strictly below $i(v_{2k+3}\to v_{2k+2})$ are either of the type $s_1\to s_2$ or $l_1\to l_2$ but neither of these are possible due to $p$ not containing an $s$-subpath and the minimality of $i(v_{2k+1}\to v_{2k+2})$. It follows that $i(v_{2k+3}\to v_{2k+4}) =i(v_{2k+3}\to v_{2k+2})$ and so $v_{2k+4} $ must be of type $l_4$ (as the arrow to type $l_2$ vertex is reversed). By alternatingly repeating the two aforementioned arguments one can deduce that $v_{2k+1} \to \dots \to v_{2k+8}$ is an $l$-subpath. Hence every $p\in P$ contains an $s$- or an $l$-subpath.

The pairing we foreshadowed takes the form of a bijection $\lambda\colon P\to P$. The function $\lambda$ takes in a path $p$, cuts out the first occurrence of an $s$- or $l$-subpath and glues back in an $l$- or an $s$-subpath. If $s$-subpath was cut out, then $l$-subpath is glued in and vice versa. Since $s$- and $l$-subpaths begin and end at similar vertices, the map $\lambda$ is a well defined and since $\lambda(\lambda(p))=p$, it is a bijection.

To confirm that $R(\lambda(p))=-R(p)$ we take a path $p\in P$ where an $s$-path occurs first. Denote $s_1 \to s_2$ and $l_1\to \dots \to l_8$ as the cut out $s$-path of $p$ and the $l$-path which is glued back in by $\lambda$. Perhaps surprisingly, we do not need to fix a sign convention on the complex $\CKh (L_t)$ to see that the signs of $R(s_1\to s_2)$ and $R(l_1\to l_2)$ agree with each other. On the other hand, the signs of $R(l_2\to l_3)$ and $R(l_3\to l_4)$ disagree with each other as do those of $R(l_4\to l_5)$ and $R(l_5\to l_6)$ and further those of $R(l_6\to l_7)$ and $R(l_7\to l_8)$. Putting all of this together gives $R(s_1\to s_2)=-R(l_1\to \dots \to l_8)$ and since the paths agree elsewhere we also get $R(\lambda(p))=-R(p)$ which concludes the proof. \qed

\begin{remark}
    In proving Proposition \ref{proposition: asymptotical strictness of the bounds}, the signs were only needed at the last step. There, it sufficed that every edge $e\colon a\to b$ of the hypercube $\CKh(L_t)$ between two smoothings $a$ and $b$ has a consistent sign no matter what colors are chosen on the loops of $a$ and $b$. This fact also holds for all sign patterns of odd Khovanov homology \cite{MR3071132}. Hence the proof can be carried out in the odd setting and Proposition \ref{proposition: asymptotical strictness of the bounds} holds for odd Khovanov homology as well.
\end{remark}

\begin{remark}
    While much of the construction of the non-vanishing homology cycles $C_w$ was local, the global requirement that all neighboring morphisms of $C_w$ were splits was also essential. This leaves some room for tweaking the construction of $L_t$. By considering braid closures $L'_t$ of $(\sigma_1\sigma_3 \sigma_2^4)^{2t}\sigma_1 \sigma_3 \sigma_2^3$ one can build similar homology cycles $C_w'$ by $1$-smoothing the uppermost $5$ crossings. 
    It follows that the asymptotic lower bounds of Proposition \ref{proposition: asymptotical strictness of the bounds} can be realized with knots and therefore it is not a special feature of $2$ or $4$ component links.
\end{remark}

Typically, the Jones polynomial is much easier to obtain than Khovanov homology both theoretically and computationally. In the proof of Theorem, \ref{thm:alg3braids} we already used this fact to our advantage when obtaining the Khovanov homology of quasi-alternating links from the Jones polynomial and signature. In Proposition \ref{proposition: asymptotical strictness of the bounds} the situation is remarkably opposite: we obtain lower bounds on the ranks of Khovanov homology groups, which lie on a single diagonal, without being able to say anything about the Jones polynomials $J_{L_t}(q)$. Since the links $L_t$ are not homologically thin for $t$ large enough\footnote{
By a result of Islambouli and Willis  \cite{zbMATH06916081}, the complexes of positive braids $(\sigma_1 \sigma_3\sigma_2^4)^t \sigma_1 \sigma_3$ converge to the categorified Jones-Wenzel projector. Hence the Khovanov homology of links $L_t$, suitably normalized, converges to the stable Khovanov homology of $T(4,\infty)$. Since the $\HKh(T(4,\infty))$ is not homologically thin in the small ($\leq 4$) homological degrees (\cite{zbMATH05220874} Theorem 3.4) neither can $L_t$ be for large enough $t$. 
}, we cannot ensure that in decategorifying $\HKh(L_t)$ to $J_{L_t}(q)$  the monomials corresponding to the high rank homology groups will not get canceled in the alternating sum.

\section{Open questions}

In \cite{zbMATH07862431}, Przytycki and Silvero show that the extremal Khovanov homology of a closed $4$-braid $L$ can be computed in polynomial time. Furthermore, there are at most two non-trivial groups $\HKh^{i,q_{\min}}(L)$ and $\HKh^{k,q_{\min}}(L)$ with $k\geq i$, both free abelian and total rank at most $4$ \cite[Cor.1.6]{zbMATH07862431}. A priori, the gap $k-i$ could be arbitrarily large, so our methods do not give a direct proof of their result. Still, one can ask whether a variation of our Algorithm \ref{Algorithm: extremal lines over Z} which only keeps track of minimal quantum degrees can be used to recover their result.

\begin{question}
Is there a modification of Algorithm \ref{Algorithm: extremal lines over Z} which calculates the extremal Khovanov homology of a closed $k$-braid in polynomial time?
\end{question}

The scanning algorithm of Bar-Natan glues complexes of single crossings with simple planar arc diagrams of Type I-IV before simplifying the total complex with iterated Gaussian elimination. If implemented naively, this procedure requires one to build a larger complex before starting the simplification process:
    $$
    C \xrightarrow{\text{ glue }} C\otimes_P \llbracket c \rrbracket \xrightarrow{\text{deloop}} D \xrightarrow{\substack{\text{iterate Gaussian} \\ \text{elimination}}} E. 
    $$
    The Morse complexes we build for Type II, III, and IV diagrams and crossings $c\in\{\crosPos,\crosNeg\}$ in Section \ref{subsec: blocks} could be written explicitly and implemented matrix element by matrix element without first delooping, which accounts to 
    $$
    C \xrightarrow{\substack{\text{build the Morse complex} \\ \text{ associated to } c \text{ and } P }} M(C) \xrightarrow{\substack{\text{iterate Gaussian} \\ \text{elimination}}} E. 
    $$
    It is straight-forward to see that this alternative, more laborious, workflow will still take exponential time on knots with exponentially large Betti numbers. Nevertheless, there ought to be room for optimization and the following question could be investigated empirically.
\begin{question}
    Would implementing a Morse theoretic simplification as part of the scanning algorithm significantly speed up the computation for average-case knots?
\end{question}

In addition to average-case knots, one could hope that computing Khovanov homology would be fast for links of special interest. In the study of Khovanov homology, torus links have inarguably received particular attention and it is natural to ask:
\begin{question}\label{question:scanning torus}
Let $n\geq 4$. Does Bar-Natan's scanning algorithm compute the Khovanov homology of an $n$-strand torus link in polynomial time, either with integral or field coefficients? 
\end{question}

Bar-Natan's scanning algorithm computes the Khovanov homology of $2$- and $3$-strand torus links in polynomial time. This can be seen from the techniques in \cite[\S3,4]{schuetz2025kh3braids}. For $n=4$, the Bar-Natan complexes $\llbracket (\sigma_1\sigma_2 \sigma_3)^n \rrbracket$ of open braids $(\sigma_1\sigma_2 \sigma_3)^n$ can be simplified to complexes with quadratically many generators \cite[Corollary 4.2]{kelo2025torus4} which supports a positive answer to Question \ref{question:scanning torus}, at least when using coefficients from a finite field. An affirmative answer is also suggested by the following proposition, the proof of which was kindly provided to the authors by Qiuyu Ren.  
\begin{proposition}\label{Proposition: polybound on torus}
    There exist $\{c_n\}_{n\geq 2}$, such that $\dim_{\mathbb F} (\HKh(T(n,m);\mathbb F) )\leq c_n\,m^{\lfloor \frac{n}{2} \rfloor}
    $ for all fields $\mathbb F$.
\end{proposition}
\begin{proof}
    Let $L_{n,k}$ denote the braid closure of $(\sigma_1 \dots \sigma_{n-1})^m \sigma_1\dots \sigma_r$ where $k=(n-1)m+r$ and $0\leq r <n-1$. Since $T(n,m)=L_{n, (n-1)m}$, it suffices to show that there exists coefficients $c_n$, such that $\dim_{\mathbb F} (\HKh(L_{n,m};\mathbb F) )\leq c_n m^{\lfloor \frac{n}{2} \rfloor}$. We proceed with an induction on $n$; sufficient coefficients $c_2$ and $c_3$ can be obtained from \cite[\S 6.2]{MR1740682}, \cite[Cor.5.7]{MR4430925} and Theorem \ref{thm:case_omega_5}. Assume next that there exists $c_2,\dots,c_{n-1}$ with $\dim_{\mathbb F} (\HKh(L_{i,m};\mathbb F) )\leq c_i m^{\lfloor \frac{i}{2} \rfloor}$. 
    
    The last crossing of $L_{n,m}$ can be resolved in two ways, which generates two links $L_{n,m}[\smoothingZero]$ and $L_{n,m}[\smoothingOne]$. The link $L_{n,m}[\smoothingZero]$ is isotopic to $L_{n,m-1}$ whereas $L_{n,m}[\smoothingOne]$ is isotopic to some $L_{s,t}$ or $L_{s,t}\sqcup (\text{unknot})$ for some $s\leq n-2 $ and $t\leq m$. The isotopy argument for $L_{n,m}[\smoothingOne]$ is explained well in \cite[Figures 15-20]{zbMATH06708472} although our conventions are mirrored.

    Associated to the last crossing of the braid word $L_{n,m}$, there is an exact triangle
$$
\begin{tikzpicture}
    \node (A) at (0,0) {$\HKh(L_{n,m}[\smoothingOne] ; \mathbb F)$};
    \node (B) at (6,0) {$\HKh(L_{n,m} ; \mathbb F)$};
    \node (C) at (3,-1) {$\HKh(L_{n,m}[\smoothingZero] ; \mathbb F)$.};
    \draw[->] (A) --(B);
    \draw[->] (B) --(C);
    \draw[->] (C) --(A);
\end{tikzpicture}
$$
    When plugging in the aforementioned isotopies, this yields an inequality on the dimensions: 
    $$
    \dim_{\mathbb F} (\HKh(L_{n,m};\mathbb F)) \leq \dim_{\mathbb F} (\HKh(L_{n,m-1};\mathbb F)) +2\dim_{\mathbb F} (\HKh(L_{s,t};\mathbb F)).
    $$
    Hence choosing $c_n=2\max_{i<n} c_i$ gives the desired bound for $n$.
\end{proof}

\bibliography{KnotHomology}
\bibliographystyle{amsalpha}

\end{document}

%% file: simplePlanars/alltypes.tex
\begin{tikzpicture}

\node[scale=0.8] at (0,0) {\input{simplePlanars/simplePlanarsA}};

\node[scale=0.8] at (8,0) {\input{simplePlanars/simplePlanarsB}};

\node[scale=0.8] at (0,-4.5) {\input{simplePlanars/simplePlanarsC}};

\node[scale=0.8] at (8,-4.5) {\input{simplePlanars/simplePlanarsD}};

\node[gray] at (-1.4,0.4) {\footnotesize $2n-1$}; 
\node[gray] at (-1.4+8,0.4) {\footnotesize $2n-2$}; 
\node[gray] at (-1.4,0.4-4.5) {\footnotesize $2n-3$}; 

\node[gray] at (0,1.55) {\footnotesize $2n+2$}; 
\node[gray] at (0+8,1.55) {\footnotesize $2n$}; 
\node[gray] at (0,1.55-4.5) {\footnotesize $2n-2$}; 

\node at (0,-1.8) {Type I};
\node at (0+8,-1.8) {Type II};
\node at (0,-1.8-4.5) {Type III};
\node at (8,-1.8-4.5) {Type IV};

\end{tikzpicture}

%% file: simplePlanars/simplePlanarsA.tex
\begin{tikzpicture}

\draw[dashed] (-3.5,-2) rectangle (4,1.5);
\draw[dashed]  (-3,-1) rectangle (0,0.5);
\draw[dashed]  (1.5,-1) rectangle (3.5,0.5);

\draw (-2.5,0.5) -- (-2.5,1.5);
\draw (-2.2,0.5) -- (-2.2,1.5);
\draw (-0.8,0.5) -- (-0.8,1.5);

\draw (-0.5,0.5) -- (-0.5,1.5);
\draw (-2.5,-1) -- (-2.5,-1.8) -- (3,-1.8) -- (3,-1);
\draw (0.75,1.5) -- (0.75,-1.5) -- (2,-1.5) -- (2,-1);
\draw (2,1.5) -- (2,0.5);
\draw (3,0.5) -- (3,1.5);

\filldraw (-1.3,1) circle (0.04);    
\filldraw (-1.5,1) circle (0.04);
\filldraw (-1.7,1) circle (0.04);


\end{tikzpicture}

%% file: simplePlanars/simplePlanarsB.tex
\begin{tikzpicture}

\draw[dashed] (-3.5,-2) rectangle (4,1.5);
\draw[dashed]  (-3,-1) rectangle (0,0.5);
\draw[dashed]  (1.5,-1) rectangle (3.5,0.5);

\draw (-2.5,0.5) -- (-2.5,1.5);
\draw (-2.2,0.5) -- (-2.2,1.5);
\draw (-0.8,0.5) -- (-0.8,1.5);

\draw (-0.5,0.5) -- (-0.5,1.5);
\draw (-2.5,-1) -- (-2.5,-1.8) -- (3,-1.8) -- (3,-1);
\draw (-1.5,-1) -- (-1.5,-1.5) -- (2,-1.5) -- (2,-1);
\draw (2,1.5) -- (2,0.5);
\draw (3,0.5) -- (3,1.5);

\filldraw (-1.3,1) circle (0.04);    
\filldraw (-1.5,1) circle (0.04);
\filldraw (-1.7,1) circle (0.04);


\end{tikzpicture}

%% file: simplePlanars/simplePlanarsC.tex
\begin{tikzpicture}

\draw[dashed] (-3.5,-2) rectangle (4,1.5);
\draw[dashed]  (-3,-1) rectangle (0,0.5);
\draw[dashed]  (1.5,-1) rectangle (3.5,0.5);

\draw (-2.5,0.5) -- (-2.5,1.5);
\draw (-2.2,0.5) -- (-2.2,1.5);
\draw (-0.8,0.5) -- (-0.8,1.5);

\draw (-0.5,0.5) -- (-0.5,1.5);
\draw (-2.5,-1) -- (-2.5,-1.8) -- (3,-1.8) -- (3,-1);
\draw (-1.5,-1) -- (-1.5,-1.5) -- (2,-1.5) -- (2,-1);
\draw (-0.5,-1) -- (-0.5,-1.25) -- (0.75,-1.25) -- (0.75,1) -- (2,1) -- (2,0.5);
\draw (3,0.5) -- (3,1.5);

\filldraw (-1.3,1) circle (0.04);    
\filldraw (-1.5,1) circle (0.04);
\filldraw (-1.7,1) circle (0.04);


\end{tikzpicture}

%% file: simplePlanars/simplePlanarsD.tex
\begin{tikzpicture}

\draw[dashed] (-3.5,-2) rectangle (4,1.5);
\draw[dashed]  (-3,-1) rectangle (0,0.5);
\draw[dashed]  (1.5,-1) rectangle (3.5,0.5);

\draw (-2.5,0.5) -- (-2.5,1.25) -- (3,1.25) -- (3,0.5);

\draw (-2.5,-1) -- (-2.5,-1.8) -- (3,-1.8) -- (3,-1);
\draw (-0.5,-1) -- (-0.5,-1.5) -- (2,-1.5) -- (2,-1);
\draw (-0.5,0.5) -- (-0.5,1) -- (2,1) -- (2,0.5);

\end{tikzpicture}

%% file: simplePlanars/SP2connsComb.tex
\begin{tikzpicture}

\node at (0,0) {\input{simplePlanars/SP2conns1}};

\node at (6,0) {\input{simplePlanars/SP2conns2}};

\node[gray] at (0,0.9) {\footnotesize $2n-2$}; 
\node[gray] at (6,0.9) {\footnotesize $2n-2$};

\node at (0,-1) {$D^i$};
\node at (6,-1) {$E^i$};

\end{tikzpicture}

%% file: simplePlanars/SP2conns1.tex
\begin{tikzpicture}

\draw[dashed]  (-3,-1) rectangle (0,0.5);

\draw[line width=0.6mm] (-2.5,0.5) -- (-2.5,0.3);
\draw[line width=0.6mm] (-2.2,0.5) -- (-2.2,0.3);
\draw[line width=0.6mm] (-0.8,0.5) -- (-0.8,0.3);

\draw[line width=0.6mm] (-0.5,0.5) -- (-0.5,0.3);
\draw[line width=0.6mm] (-2.5,-1) to[bend left=80] (-1.5,-1) ;

\filldraw (-1.3,0.35) circle (0.02);    
\filldraw (-1.5,0.35) circle (0.02);
\filldraw (-1.7,0.35) circle (0.02);


\end{tikzpicture}

%% file: simplePlanars/SP2conns2.tex
\begin{tikzpicture}

\draw[dashed]  (-3,-1) rectangle (0,0.5);

\draw[line width=0.6mm] (-2.5,0.5) -- (-2.5,0.3);
\draw[line width=0.6mm] (-2.2,0.5) -- (-2.2,0.3);
\draw[line width=0.6mm] (-0.8,0.5) -- (-0.8,0.3);

\draw[line width=0.6mm] (-0.5,0.5) -- (-0.5,0.3);
\draw[line width=0.6mm] (-2.5,-1) to (-2.5,-0.8) ;
\draw[line width=0.6mm] (-1.5,-1) to (-1.5,-0.8) ;

\filldraw (-1.3,0.35) circle (0.02);    
\filldraw (-1.5,0.35) circle (0.02);
\filldraw (-1.7,0.35) circle (0.02);


\end{tikzpicture}

%% file: simplePlanars/SP3connsComb.tex
\begin{tikzpicture}

\node at (0,0) {\input{simplePlanars/SP3conns1}};

\node at (4,0) {\input{simplePlanars/SP3conns2}};

\node at (8,0) {\input{simplePlanars/SP3conns3}};

\node[gray] at (0,0.9) {\footnotesize $2n-3$}; 
\node[gray] at (4,0.9) {\footnotesize $2n-3$}; 
\node[gray] at (8,0.9) {\footnotesize $2n-3$};

\node at (0,-1) {$D^i$};
\node at (4,-1) {$E^i$};
\node at (8,-1) {$F^i$.};

\end{tikzpicture}

%% file: simplePlanars/SP3conns1.tex
\begin{tikzpicture}

\draw[dashed]  (-3,-1) rectangle (0,0.5);

\draw[line width=0.6mm] (-2.5,0.5) -- (-2.5,0.3);
\draw[line width=0.6mm] (-2.2,0.5) -- (-2.2,0.3);
\draw[line width=0.6mm] (-0.8,0.5) -- (-0.8,0.3);

\draw[line width=0.6mm] (-0.5,0.5) -- (-0.5,0.3);
\draw[line width=0.6mm] (-2.5,-1) to[bend left=80] (-1.5,-1) ;

\draw[line width=0.6mm] (-0.5,-1) -- (-0.5,-0.8);

\filldraw (-1.3,0.35) circle (0.02);    
\filldraw (-1.5,0.35) circle (0.02);
\filldraw (-1.7,0.35) circle (0.02);



\end{tikzpicture}

%% file: simplePlanars/SP3conns2.tex
\begin{tikzpicture}

\draw[dashed]  (-3,-1) rectangle (0,0.5);

\draw[line width=0.6mm] (-2.5,0.5) -- (-2.5,0.3);
\draw[line width=0.6mm] (-2.2,0.5) -- (-2.2,0.3);
\draw[line width=0.6mm] (-0.8,0.5) -- (-0.8,0.3);

\draw[line width=0.6mm] (-0.5,0.5) -- (-0.5,0.3);
\draw[line width=0.6mm] (-1.5,-1) to[bend left=80] (-0.5,-1) ;

\draw[line width=0.6mm] (-2.5,-1) -- (-2.5,-0.8);

\filldraw (-1.3,0.35) circle (0.02);    
\filldraw (-1.5,0.35) circle (0.02);
\filldraw (-1.7,0.35) circle (0.02);



\end{tikzpicture}

%% file: simplePlanars/SP3conns3.tex
\begin{tikzpicture}

\draw[dashed]  (-3,-1) rectangle (0,0.5);

\draw[line width=0.6mm] (-2.5,0.5) -- (-2.5,0.3);
\draw[line width=0.6mm] (-2.2,0.5) -- (-2.2,0.3);
\draw[line width=0.6mm] (-0.8,0.5) -- (-0.8,0.3);

\draw[line width=0.6mm] (-0.5,0.5) -- (-0.5,0.3);

\draw[line width=0.6mm] (-0.5,-1) -- (-0.5,-0.8);
\draw[line width=0.6mm] (-1.5,-1) -- (-1.5,-0.8);
\draw[line width=0.6mm] (-2.5,-1) -- (-2.5,-0.8);

\filldraw (-1.3,0.35) circle (0.02);    
\filldraw (-1.5,0.35) circle (0.02);
\filldraw (-1.7,0.35) circle (0.02);



\end{tikzpicture}

%% file: simplePlanars/nugatory.tex
\begin{tikzpicture}
\draw (-1.5,-0.5) rectangle (-0.4,1.5);

    \draw[thick] (-0.4,0) -- (0,0)  -- (0.35,0.35);
    \draw[thick] (0.65,0.65)  -- (1,1) -- (1.4,1);
    \draw[thick] (1.4,0) --(1,0) -- (0,1) -- (-0.4,1);

\draw (1.4,-0.5) rectangle (2.5,1.5);

\node at (-0.95,0.5) {$T_1$};
\node at (1.95,0.5) {$T_2$};

\begin{scope}[shift={(6,0)}]
    
\draw (-1.5,-0.5) rectangle (-0.4,1.5);

    \draw[thick] (-0.4,0) -- (1.4,0);
    \draw[thick] (-0.4,1) -- (1.4,1);

\draw (1.4,-0.5) rectangle (2.5,1.5);

\node at (-0.95,0.5) {$T_1$};
\node[yscale=-1]  at (1.95,0.5) {$T_2$};

\end{scope}


\end{tikzpicture}

%% file: Strictness_diags/alpha_beta.tex
\begin{tikzpicture}

\node at (-5, -0.825) [anchor=south west, scale=0.5] { 
            \input{Strictness_diags/braid}
        };  

\node at (-0.8, 0) [anchor=south west, scale=0.5] { 
            \input{Strictness_diags/alpha}
        };  
\node at (2.5,0) [anchor=south west, scale=0.5] { 
            \input{Strictness_diags/beta} 
        };

\node at (6, -0.825) [anchor=south west, scale=0.5] { 
            \input{Strictness_diags/Cw}
        };  
        
\node[black!70] at (1.3,0.25) {$1$};
\node[black!70] at (1.3,0.75) {$1$};
\node[black!70] at (1.3,1.25) {$0$};
\node[black!70] at (1.3,1.75) {$1$};
\node[black!70] at (1.3,2.25) {$1$};
\node[black!70] at (1.3,2.75) {$1$};

\node[black!70] at (4.6,0.25) {$1$};
\node[black!70] at (4.6,0.75) {$1$};
\node[black!70] at (4.6,1.25) {$1$};
\node[black!70] at (4.6,1.75) {$1$};
\node[black!70] at (4.6,2.25) {$1$};
\node[black!70] at (4.6,2.75) {$1$};

\node at (0.05,-0.3) {$\alpha$};
\node at (0.35+3,-0.3) {$\beta$};

\node at (7.5,-1) {$u_w$};


\node at (-2.9,3.42) {$\vdots$};

\node at (-3.5,-1) {$L_t$};

\node at (8,2.9) {$w_t$};
\node at (8,0.65) {$w_1$};

\node at (8.05,1.9) {$\vdots$};

\end{tikzpicture}

    

%% file: Strictness_diags/braid.tex
    \begin{tikzpicture}

    \newcommand{\cros}[2]{%
  \begin{scope}[shift={({#1},{#2})}]
    \draw[thick] (0,0) -- (1,1);
    \draw[thick] (0,1) -- (0.35,0.65);
    \draw[thick] (0.65,0.35) -- (1,0);
  \end{scope}
}

\cros{0}{0};
\cros{2}{1};
\cros{1}{2};
\cros{1}{3};
\cros{1}{4};
\cros{1}{5};
\cros{0}{7};
\cros{2}{8};

    \draw[thick] (0, 8) to (0, 9);
    \draw[thick] (1, 8) to (1, 9);
    \draw[thick] (2, 7) to (2, 8);
    \draw[thick] (3, 7) to (3, 8);

    \draw[thick] (0, 1) to (0, 2);
    \draw[thick] (0, 2) to (0, 3);
    \draw[thick] (0, 3) to (0, 4);
    \draw[thick] (0, 4) to (0, 5);
    \draw[thick] (0, 5) to (0, 6);
    \draw[thick] (2, 0) to (2, 1);
    \draw[thick] (3, 0) to (3, 1);
    \draw[thick] (3, 2) to (3, 3);
    \draw[thick] (3, 3) to (3, 4);
    \draw[thick] (3, 4) to (3, 5);
    \draw[thick] (3, 5) to (3, 6);
    \draw[thick] (1, 1) to (1, 2);


    \draw[thick] (0,0) to[bend left=50] (-1,0);
    \draw[thick, <-] (-1,0) to (-1,9) to[bend left=50] (0,9);
    \draw[thick] (1,0) to[bend left=50] (-1.5,0);
    \draw[thick, <-] (-1.5,0) to (-1.5,9) to[bend left=50] (1,9);
    \draw[thick] (2,0) to[bend left=50] (-2,0);
    \draw[thick, <-](-2,0) to (-2,9) to[bend left=50] (2,9);
    \draw[thick] (3,0) to[bend left=50] (-2.5,0);
    \draw[thick, <-](-2.5,0)to (-2.5,9) to[bend left=50] (3,9);


    \end{tikzpicture}

    

%% file: Strictness_diags/alpha.tex
    \begin{tikzpicture}
    \begin{scope}[color=blue]
\end{scope}
\begin{scope}[redline]
        \draw (0, 0) to[bend left=50] (1, 0);
    \draw (2, 0) to (2, 1);
    \draw (2, 1) to[bend left=50] (3, 1);
    \draw (3, 0) to (3, 1);
    \draw (0, 5) to (0, 6);
    \draw (0, 4) to (0, 5);
    \draw (0, 3) to (0, 4);
    \draw (0, 2) to (0, 3);
    \draw (0, 1) to (0, 2);
    \draw (0, 1) to[bend right=50] (1, 1);
    \draw (1, 1) to (1, 2);
    \draw (1, 2) to[bend right=50] (1, 3);
    \draw (1, 3) to[bend left=50] (2, 3);
    \draw (2, 2) to[bend left=50] (2, 3);
    \draw (2, 2) to[bend right=50] (3, 2);
    \draw (3, 2) to (3, 3);
    \draw (3, 3) to (3, 4);
    \draw (3, 4) to (3, 5);
    \draw (3, 5) to (3, 6);
    \draw (1, 6) to[bend right=50] (2, 6);

    \draw (1, 4) to[bend left=50] (2, 4);
    \draw (1, 4) to[bend right=50] (2, 4);
    \draw (1, 5) to[bend left=50] (2, 5);
    \draw (1, 5) to[bend right=50] (2, 5);
\end{scope}

    \end{tikzpicture}

    

%% file: Strictness_diags/beta.tex
    \begin{tikzpicture}
    \begin{scope}[color=blue]
\end{scope}
\begin{scope}[redline]
        \draw (0, 0) to[bend left=50] (1, 0);
    \draw (2, 0) to (2, 1);
    \draw (2, 1) to[bend left=50] (3, 1);
    \draw (3, 0) to (3, 1);
    \draw (0, 5) to (0, 6);
    \draw (0, 4) to (0, 5);
    \draw (0, 3) to (0, 4);
    \draw (0, 2) to (0, 3);
    \draw (0, 1) to (0, 2);
    \draw (0, 1) to[bend right=50] (1, 1);
    \draw (1, 1) to (1, 2);
    \draw (1, 2) to[bend left=50] (2, 2);
    \draw (2, 2) to[bend right=50] (3, 2);
    \draw (3, 2) to (3, 3);
    \draw (3, 3) to (3, 4);
    \draw (3, 4) to (3, 5);
    \draw (3, 5) to (3, 6);
    \draw (1, 6) to[bend right=50] (2, 6);

    \draw (1, 3) to[bend left=50] (2, 3);
    \draw (1, 3) to[bend right=50] (2, 3);
    \draw (1, 4) to[bend left=50] (2, 4);
    \draw (1, 4) to[bend right=50] (2, 4);
    \draw (1, 5) to[bend left=50] (2, 5);
    \draw (1, 5) to[bend right=50] (2, 5);
\end{scope}

    \end{tikzpicture}

    

%% file: Strictness_diags/Cw.tex
    \begin{tikzpicture}

    \newcommand{\cros}[2]{%
  \begin{scope}[shift={({#1},{#2})}]
    \draw[redline] (0,0) -- (1,1);
    \draw[redline] (0,1) -- (0.35,0.65);
    \draw[redline] (0.65,0.35) -- (1,0);
  \end{scope}
}
    
    \draw[redline] (0,8) to [bend right=50] (1,8);
    \draw[redline] (0,7) to [bend left=50] (1,7);

    \draw[redline] (2,8) to [bend left=50] (3,8);
    \draw[redline] (2,9) to [bend right=50] (3,9);

    \draw[redline] (0, 8) to (0, 9);
    \draw[redline] (1, 8) to (1, 9);
    \draw[redline] (2, 7) to (2, 8);
    \draw[redline] (3, 7) to (3, 8);

    \draw[redline] (0,0.5) to (0,0) to[bend left=50] (-1,0) to (-1,9) to[bend left=50] (0,9);
    \draw[redline] (1,0.5) to (1,0) to[bend left=50] (-1.5,0) to (-1.5,9) to[bend left=50] (1,9);
    \draw[redline] (2,0.5) to (2,0) to[bend left=50] (-2,0) to (-2,9) to[bend left=50] (2,9);
    \draw[redline] (3,0.5) to (3,0) to[bend left=50] (-2.5,0) to (-2.5,9) to[bend left=50] (3,9);

    \draw (-0.2,0.5) rectangle (3.2,2);
    \draw[redline] (0,2.5) -- (0,2);
    \draw[redline] (1,2.5) -- (1,2);
    \draw[redline] (2,2.5) -- (2,2);
    \draw[redline] (3,2.5) -- (3,2);

    \draw[redline] (0,4.5) -- (0,5);
    \draw[redline] (1,4.5) -- (1,5);
    \draw[redline] (2,4.5) -- (2,5);
    \draw[redline] (3,4.5) -- (3,5);

    \draw (-0.2,5) rectangle (3.2,2.5+4);
    
    \draw[redline] (0,6.5) -- (0,7);
    \draw[redline] (1,6.5) -- (1,7);
    \draw[redline] (2,6.5) -- (2,7);
    \draw[redline] (3,6.5) -- (3,7);

    \end{tikzpicture}

    

%% file: Strictness_diags/L-path/Lpath.tex
\begin{tikzpicture}
\node at (0, 0) [anchor=south west, scale=0.5] { 
            \input{Strictness_diags/L-path/ACBBBBAC110YYY1Y}
        };  
\node at (3,0) [anchor=south west, scale=0.5] { 
            \input{Strictness_diags/L-path/ACBBBBAC111XYY1Y} 
        };

\node at (6,0) [anchor=south west, scale=0.5] { 
            \input{Strictness_diags/L-path/ACBBBBAC1110YY1Y} 
        };

\node at (9,0) [anchor=south west, scale=0.5] { 
            \input{Strictness_diags/L-path/ACBBBBAC111YXY1Y} 
        };

\node at (1,-5) [anchor=south west, scale=0.5] { 
            \input{Strictness_diags/L-path/ACBBBBAC111Y0Y1Y} 
        };

\node at (4,-5) [anchor=south west, scale=0.5] { 
            \input{Strictness_diags/L-path/ACBBBBAC111YYX1Y} 
        };

\node at (7,-5) [anchor=south west, scale=0.5] { 
            \input{Strictness_diags/L-path/ACBBBBAC111YY01Y} 
        };

\node at (10,-5) [anchor=south west, scale=0.5] { 
            \input{Strictness_diags/L-path/ACBBBBAC111YYY1X} 
        };

\draw[->] (1.8,1.5) -- (2.8,2); 
\draw[->] (1.8+3,2) -- (2.8+3,1.5); 
\draw[->] (1.8+6,1.5) -- (2.8+6,2);
\draw[->] (1.8+9,2) -- (2.8+9,1.5);

\draw[->] (2.8-3,-3) -- (3.8-3,-3.5); 

\draw[->] (2.8,-3.5) -- (3.8,-3); 
\draw[->] (2.8+3,-3) -- (3.8+3,-3.5); 
\draw[->] (2.8+6,-3.5) -- (3.8+6,-3); 

\node at (0.8,-0.3) {$p_1$};
\node at (0.8+3,-0.3) {$p_2$};
\node at (0.8+6,-0.3) {$p_3$};
\node at (0.8+9,-0.3) {$p_4$};

\node at (1.8,-5.3) {$p_5$};
\node at (1.8+3,-5.3) {$p_6$};
\node at (1.8+6,-5.3) {$p_7$};
\node at (1.8+9,-5.3) {$p_8$};

\end{tikzpicture}

    

%% file: Strictness_diags/L-path/ACBBBBAC110YYY1Y.tex
    \begin{tikzpicture}
\begin{scope}[color=blue]
\end{scope}
\begin{scope}[redline]
        
        \draw (2, 8) to[bend right=50] (3, 8);
    \draw (0, 0) to[bend left=50] (1, 0);
    \draw (1, 7) to (1, 8);
    \draw (0, 7) to[bend right=50] (1, 7);
    \draw (0, 7) to (0, 8);
    \draw (2, 0) to (2, 1);
    \draw (2, 1) to[bend left=50] (3, 1);
    \draw (3, 0) to (3, 1);

    \draw (1, 3) to[bend left=50] (2, 3);
    \draw (2, 2) to[bend left=50] (2, 3);
    \draw (2, 2) to[bend right=50] (3, 2);
    \draw (3, 2) to (3, 3);
    \draw (3, 3) to (3, 4);
    \draw (3, 4) to (3, 5);
    \draw (3, 5) to (3, 6);
    \draw (3, 6) to (3, 7);
    \draw (2, 7) to[bend left=50] (3, 7);
    \draw (2, 6) to (2, 7);
    \draw (1, 6) to[bend right=50] (2, 6);
    \draw (0, 6) to[bend left=50] (1, 6);
    \draw (0, 5) to (0, 6);
    \draw (0, 4) to (0, 5);
    \draw (0, 3) to (0, 4);
    \draw (0, 2) to (0, 3);
    \draw (0, 1) to (0, 2);
    \draw (0, 1) to[bend right=50] (1, 1);
    \draw (1, 1) to (1, 2);
    \draw (1, 2) to[bend right=50] (1, 3);
    \draw (1, 4) to[bend left=50] (2, 4);
    \draw (1, 4) to[bend right=50] (2, 4);
    \draw (1, 5) to[bend left=50] (2, 5);
    \draw (1, 5) to[bend right=50] (2, 5);
    \draw (2, 7) to[bend left=50] (3, 7);
    \draw (3, 6) to (3, 7);
    \draw (3, 5) to (3, 6);
    \draw (3, 4) to (3, 5);
    \draw (3, 3) to (3, 4);
    \draw (3, 2) to (3, 3);
    \draw (2, 2) to[bend right=50] (3, 2);
    \draw (2, 2) to[bend left=50] (2, 3);
    \draw (1, 3) to[bend left=50] (2, 3);
    \draw (1, 2) to[bend right=50] (1, 3);
    \draw (1, 1) to (1, 2);
    \draw (0, 1) to[bend right=50] (1, 1);
    \draw (0, 1) to (0, 2);
    \draw (0, 2) to (0, 3);
    \draw (0, 3) to (0, 4);
    \draw (0, 4) to (0, 5);
    \draw (0, 5) to (0, 6);
    \draw (0, 6) to[bend left=50] (1, 6);
    \draw (1, 6) to[bend right=50] (2, 6);
    \draw (2, 6) to (2, 7);
\end{scope}

    \end{tikzpicture}

    

%% file: Strictness_diags/L-path/ACBBBBAC111XYY1Y.tex
    \begin{tikzpicture}

\begin{scope}[blueline]
    \draw (1, 3) to[bend left=50] (2, 3);
    \draw (1, 3) to[bend right=50] (2, 3);
\end{scope}
\begin{scope}[redline]
        \draw (2, 8) to[bend right=50] (3, 8);
    \draw (0, 0) to[bend left=50] (1, 0);
    \draw (1, 7) to (1, 8);
    \draw (0, 7) to[bend right=50] (1, 7);
    \draw (0, 7) to (0, 8);
    \draw (2, 0) to (2, 1);
    \draw (2, 1) to[bend left=50] (3, 1);
    \draw (3, 0) to (3, 1);

    \draw (1, 4) to[bend left=50] (2, 4);
    \draw (1, 4) to[bend right=50] (2, 4);
    \draw (1, 5) to[bend left=50] (2, 5);
    \draw (1, 5) to[bend right=50] (2, 5);
    \draw (2, 7) to[bend left=50] (3, 7);
    \draw (3, 6) to (3, 7);
    \draw (3, 5) to (3, 6);
    \draw (3, 4) to (3, 5);
    \draw (3, 3) to (3, 4);
    \draw (3, 2) to (3, 3);
    \draw (2, 2) to[bend right=50] (3, 2);
    \draw (1, 2) to[bend left=50] (2, 2);
    \draw (1, 1) to (1, 2);
    \draw (0, 1) to[bend right=50] (1, 1);
    \draw (0, 1) to (0, 2);
    \draw (0, 2) to (0, 3);
    \draw (0, 3) to (0, 4);
    \draw (0, 4) to (0, 5);
    \draw (0, 5) to (0, 6);
    \draw (0, 6) to[bend left=50] (1, 6);
    \draw (1, 6) to[bend right=50] (2, 6);
    \draw (2, 6) to (2, 7);
\end{scope}

    \end{tikzpicture}

    

%% file: Strictness_diags/L-path/ACBBBBAC1110YY1Y.tex
    \begin{tikzpicture}
\begin{scope}[color=blue]
\end{scope}
\begin{scope}[redline]

        \draw (2, 8) to[bend right=50] (3, 8);
    \draw (0, 0) to[bend left=50] (1, 0);
    \draw (1, 7) to (1, 8);
    \draw (0, 7) to[bend right=50] (1, 7);
    \draw (0, 7) to (0, 8);
    \draw (2, 0) to (2, 1);
    \draw (2, 1) to[bend left=50] (3, 1);
    \draw (3, 0) to (3, 1);

    \draw (1, 4) to[bend left=50] (2, 4);
    \draw (2, 3) to[bend left=50] (2, 4);
    \draw (1, 3) to[bend right=50] (2, 3);
    \draw (1, 3) to[bend right=50] (1, 4);
    \draw (1, 5) to[bend left=50] (2, 5);
    \draw (1, 5) to[bend right=50] (2, 5);
    \draw (2, 7) to[bend left=50] (3, 7);
    \draw (3, 6) to (3, 7);
    \draw (3, 5) to (3, 6);
    \draw (3, 4) to (3, 5);
    \draw (3, 3) to (3, 4);
    \draw (3, 2) to (3, 3);
    \draw (2, 2) to[bend right=50] (3, 2);
    \draw (1, 2) to[bend left=50] (2, 2);
    \draw (1, 1) to (1, 2);
    \draw (0, 1) to[bend right=50] (1, 1);
    \draw (0, 1) to (0, 2);
    \draw (0, 2) to (0, 3);
    \draw (0, 3) to (0, 4);
    \draw (0, 4) to (0, 5);
    \draw (0, 5) to (0, 6);
    \draw (0, 6) to[bend left=50] (1, 6);
    \draw (1, 6) to[bend right=50] (2, 6);
    \draw (2, 6) to (2, 7);
\end{scope}

    \end{tikzpicture}

    

%% file: Strictness_diags/L-path/ACBBBBAC111YXY1Y.tex
    \begin{tikzpicture}
\begin{scope}[blueline]
    \draw (1, 4) to[bend left=50] (2, 4);
    \draw (1, 4) to[bend right=50] (2, 4);
\end{scope}
\begin{scope}[redline]

        \draw (2, 8) to[bend right=50] (3, 8);
    \draw (0, 0) to[bend left=50] (1, 0);
    \draw (1, 7) to (1, 8);
    \draw (0, 7) to[bend right=50] (1, 7);
    \draw (0, 7) to (0, 8);
    \draw (2, 0) to (2, 1);
    \draw (2, 1) to[bend left=50] (3, 1);
    \draw (3, 0) to (3, 1);

    \draw (1, 3) to[bend left=50] (2, 3);
    \draw (1, 3) to[bend right=50] (2, 3);
    \draw (1, 5) to[bend left=50] (2, 5);
    \draw (1, 5) to[bend right=50] (2, 5);
    \draw (2, 7) to[bend left=50] (3, 7);
    \draw (3, 6) to (3, 7);
    \draw (3, 5) to (3, 6);
    \draw (3, 4) to (3, 5);
    \draw (3, 3) to (3, 4);
    \draw (3, 2) to (3, 3);
    \draw (2, 2) to[bend right=50] (3, 2);
    \draw (1, 2) to[bend left=50] (2, 2);
    \draw (1, 1) to (1, 2);
    \draw (0, 1) to[bend right=50] (1, 1);
    \draw (0, 1) to (0, 2);
    \draw (0, 2) to (0, 3);
    \draw (0, 3) to (0, 4);
    \draw (0, 4) to (0, 5);
    \draw (0, 5) to (0, 6);
    \draw (0, 6) to[bend left=50] (1, 6);
    \draw (1, 6) to[bend right=50] (2, 6);
    \draw (2, 6) to (2, 7);
\end{scope}

    \end{tikzpicture}

    

%% file: Strictness_diags/L-path/ACBBBBAC111Y0Y1Y.tex
    \begin{tikzpicture}
\begin{scope}[color=blue]
\end{scope}
\begin{scope}[redline]

        \draw (2, 8) to[bend right=50] (3, 8);
    \draw (0, 0) to[bend left=50] (1, 0);
    \draw (1, 7) to (1, 8);
    \draw (0, 7) to[bend right=50] (1, 7);
    \draw (0, 7) to (0, 8);
    \draw (2, 0) to (2, 1);
    \draw (2, 1) to[bend left=50] (3, 1);
    \draw (3, 0) to (3, 1);

    \draw (1, 3) to[bend left=50] (2, 3);
    \draw (1, 3) to[bend right=50] (2, 3);
    \draw (1, 5) to[bend left=50] (2, 5);
    \draw (2, 4) to[bend left=50] (2, 5);
    \draw (1, 4) to[bend right=50] (2, 4);
    \draw (1, 4) to[bend right=50] (1, 5);
    \draw (2, 7) to[bend left=50] (3, 7);
    \draw (3, 6) to (3, 7);
    \draw (3, 5) to (3, 6);
    \draw (3, 4) to (3, 5);
    \draw (3, 3) to (3, 4);
    \draw (3, 2) to (3, 3);
    \draw (2, 2) to[bend right=50] (3, 2);
    \draw (1, 2) to[bend left=50] (2, 2);
    \draw (1, 1) to (1, 2);
    \draw (0, 1) to[bend right=50] (1, 1);
    \draw (0, 1) to (0, 2);
    \draw (0, 2) to (0, 3);
    \draw (0, 3) to (0, 4);
    \draw (0, 4) to (0, 5);
    \draw (0, 5) to (0, 6);
    \draw (0, 6) to[bend left=50] (1, 6);
    \draw (1, 6) to[bend right=50] (2, 6);
    \draw (2, 6) to (2, 7);
\end{scope}

    \end{tikzpicture}

    

%% file: Strictness_diags/L-path/ACBBBBAC111YYX1Y.tex
    \begin{tikzpicture}
\begin{scope}[blueline]
    \draw (1, 5) to[bend left=50] (2, 5);
    \draw (1, 5) to[bend right=50] (2, 5);
\end{scope}
\begin{scope}[redline]

        \draw (2, 8) to[bend right=50] (3, 8);
    \draw (0, 0) to[bend left=50] (1, 0);
    \draw (1, 7) to (1, 8);
    \draw (0, 7) to[bend right=50] (1, 7);
    \draw (0, 7) to (0, 8);
    \draw (2, 0) to (2, 1);
    \draw (2, 1) to[bend left=50] (3, 1);
    \draw (3, 0) to (3, 1);

    \draw (1, 3) to[bend left=50] (2, 3);
    \draw (1, 3) to[bend right=50] (2, 3);
    \draw (1, 4) to[bend left=50] (2, 4);
    \draw (1, 4) to[bend right=50] (2, 4);
    \draw (2, 7) to[bend left=50] (3, 7);
    \draw (3, 6) to (3, 7);
    \draw (3, 5) to (3, 6);
    \draw (3, 4) to (3, 5);
    \draw (3, 3) to (3, 4);
    \draw (3, 2) to (3, 3);
    \draw (2, 2) to[bend right=50] (3, 2);
    \draw (1, 2) to[bend left=50] (2, 2);
    \draw (1, 1) to (1, 2);
    \draw (0, 1) to[bend right=50] (1, 1);
    \draw (0, 1) to (0, 2);
    \draw (0, 2) to (0, 3);
    \draw (0, 3) to (0, 4);
    \draw (0, 4) to (0, 5);
    \draw (0, 5) to (0, 6);
    \draw (0, 6) to[bend left=50] (1, 6);
    \draw (1, 6) to[bend right=50] (2, 6);
    \draw (2, 6) to (2, 7);
\end{scope}

    \end{tikzpicture}

    

%% file: Strictness_diags/L-path/ACBBBBAC111YY01Y.tex
    \begin{tikzpicture}
\begin{scope}[blueline]
\end{scope}
\begin{scope}[redline]
        \draw (2, 8) to[bend right=50] (3, 8);
    \draw (0, 0) to[bend left=50] (1, 0);
    \draw (1, 7) to (1, 8);
    \draw (0, 7) to[bend right=50] (1, 7);
    \draw (0, 7) to (0, 8);
    \draw (2, 0) to (2, 1);
    \draw (2, 1) to[bend left=50] (3, 1);
    \draw (3, 0) to (3, 1);

    \draw (1, 3) to[bend left=50] (2, 3);
    \draw (1, 3) to[bend right=50] (2, 3);
    \draw (1, 4) to[bend left=50] (2, 4);
    \draw (1, 4) to[bend right=50] (2, 4);
    \draw (2, 7) to[bend left=50] (3, 7);
    \draw (3, 6) to (3, 7);
    \draw (3, 5) to (3, 6);
    \draw (3, 4) to (3, 5);
    \draw (3, 3) to (3, 4);
    \draw (3, 2) to (3, 3);
    \draw (2, 2) to[bend right=50] (3, 2);
    \draw (1, 2) to[bend left=50] (2, 2);
    \draw (1, 1) to (1, 2);
    \draw (0, 1) to[bend right=50] (1, 1);
    \draw (0, 1) to (0, 2);
    \draw (0, 2) to (0, 3);
    \draw (0, 3) to (0, 4);
    \draw (0, 4) to (0, 5);
    \draw (0, 5) to (0, 6);
    \draw (0, 6) to[bend left=50] (1, 6);
    \draw (1, 5) to[bend right=50] (1, 6);
    \draw (1, 5) to[bend right=50] (2, 5);
    \draw (2, 5) to[bend left=50] (2, 6);
    \draw (2, 6) to (2, 7);
\end{scope}

    \end{tikzpicture}

    

%% file: Strictness_diags/L-path/ACBBBBAC111YYY1X.tex
    \begin{tikzpicture}
\begin{scope}[blueline]
    \draw (2, 7) to[bend left=50] (3, 7);
    \draw (3, 6) to (3, 7);
    \draw (3, 5) to (3, 6);
    \draw (3, 4) to (3, 5);
    \draw (3, 3) to (3, 4);
    \draw (3, 2) to (3, 3);
    \draw (2, 2) to[bend right=50] (3, 2);
    \draw (1, 2) to[bend left=50] (2, 2);
    \draw (1, 1) to (1, 2);
    \draw (0, 1) to[bend right=50] (1, 1);
    \draw (0, 1) to (0, 2);
    \draw (0, 2) to (0, 3);
    \draw (0, 3) to (0, 4);
    \draw (0, 4) to (0, 5);
    \draw (0, 5) to (0, 6);
    \draw (0, 6) to[bend left=50] (1, 6);
    \draw (1, 6) to[bend right=50] (2, 6);
    \draw (2, 6) to (2, 7);
\end{scope}
\begin{scope}[redline]

        \draw (2, 8) to[bend right=50] (3, 8);
    \draw (0, 0) to[bend left=50] (1, 0);
    \draw (1, 7) to (1, 8);
    \draw (0, 7) to[bend right=50] (1, 7);
    \draw (0, 7) to (0, 8);
    \draw (2, 0) to (2, 1);
    \draw (2, 1) to[bend left=50] (3, 1);
    \draw (3, 0) to (3, 1);

    \draw (1, 3) to[bend left=50] (2, 3);
    \draw (1, 3) to[bend right=50] (2, 3);
    \draw (1, 4) to[bend left=50] (2, 4);
    \draw (1, 4) to[bend right=50] (2, 4);
    \draw (1, 5) to[bend left=50] (2, 5);
    \draw (1, 5) to[bend right=50] (2, 5);
\end{scope}

    \end{tikzpicture}

    

%% file: Strictness_diags/L-path/boxS.tex
\begin{tikzpicture}

\node at (0, 0) { 
            \input{Strictness_diags/global_skeleton}
        };  

\node at (3.5, 0) { 
            \input{Strictness_diags/global_skeleton}
        };

\node at (0.62,0.7) {$p_1$};
\node at (0.62+3.5,0.7) {$p_8$};

\draw[->] (1.5,-0.25) -- (2,0);

\node at (0,-2.3) {$s_1$};
\node at (3.5,-2.3) {$s_2$};
\end{tikzpicture}

    

%% file: Strictness_diags/global_skeleton.tex
\begin{tikzpicture}

\draw (0,0) rectangle (1.25,1.6);

\draw[thinredline] (-0.5,0) to[bend right=50] (0.25,0); 
\draw[thinredline] (-0.75,0) to[bend right=50] (0.5,0); 
\draw[thinredline] (-1,0) to[bend right=50] (0.75,0); 
\draw[thinredline] (-1.25,0) to[bend right=50] (1,0);

\node[align=center] at(0.62,.8) {$w_n$ \\ $\vdots$ \\ $w_1$};

\draw[thinredline] (0.25,1.6) to (0.25,1.8); 
\draw[thinredline] (0.5,1.6) to (0.5,1.8); 
\draw[thinredline] (0.75,1.6) to (0.75,1.8); 
\draw[thinredline] (1,1.6) to (1,1.8);

\draw (0,1.8) rectangle (1.25,2.5);

\draw[thinredline] (0.25,2.5) to (0.25,2.8); 
\draw[thinredline] (0.5,2.5) to (0.5,2.8); 
\draw[thinredline] (0.75,2.5) to (0.75,2.8); 
\draw[thinredline] (1,2.5) to (1,2.8); 

\draw[thinredline] (-0.5,0) to (-0.5,2.8); 
\draw[thinredline] (-0.75,0) to (-0.75,2.8); 
\draw[thinredline] (-1,0) to (-1,2.8); 
\draw[thinredline] (-1.25,0) to (-1.25,2.8); 

\node at (0.62,3.2) {$\vdots$};
\node at (-1.12+0.25,3.2) {$\vdots$};

\end{tikzpicture}

    

%% file: Strictness_diags/L-path/boxL.tex
\begin{tikzpicture}

\node at (0, 0) { 
            \input{Strictness_diags/global_skeleton}
        };  

\node at (3.5, 0) { 
            \input{Strictness_diags/global_skeleton}
        };

\node at (9, 0) { 
            \input{Strictness_diags/global_skeleton}
        };

\node at (0.62,0.7) {$p_1$};
\node at (0.62+3.5,0.7) {$p_2$};
\node at (0.62+9,0.7) {$p_8$};

\draw[->] (1.5,-0.25) -- (2,0); 
\draw[->] (7,-0.25) -- (7.5,0); 

\draw[->] (1.5+3.5,0) -- (2+3.5,-0.25); 

\node at (6.3,-0.15) {$\dots$};

\node at (0,-2.3) {$l_1$};
\node at (3.5,-2.3) {$l_2$};
\node at (9,-2.3) {$l_8$};

\end{tikzpicture}

    